\documentclass{article}

\usepackage{microtype}
\usepackage{graphicx}
\usepackage{subcaption}
\usepackage{booktabs}
\usepackage[hyperfootnotes=false]{hyperref}

\AtBeginDocument{\raggedbottom}
\usepackage{multirow}

\usepackage[preprint]{icml2026}

\usepackage{amsmath}
\usepackage{amssymb}
\usepackage{mathtools}
\usepackage{amsthm}
\usepackage{bm}

\usepackage[capitalize,noabbrev]{cleveref}

\theoremstyle{plain}
\newtheorem{theorem}{Theorem}[section]
\newtheorem{proposition}{Proposition}[section]
\newtheorem{lemma}{Lemma}[section]
\newtheorem{corollary}{Corollary}[section]
\theoremstyle{definition}
\newtheorem{definition}{Definition}[section]
\newtheorem{assumption}{Assumption}[section]
\theoremstyle{remark}
\newtheorem{remark}{Remark}[section]

\newcommand{\zero}{\bm{0}}
\newcommand{\RR}{\mathbb{R}}

\newcommand{\EE}{\mathbb{E}}

\newcommand{\NN}{\mathbb{N}}

\newcommand{\prox}{\mathrm{prox}}

\newcommand{\bX}{\bm{X}}
\newcommand{\bx}{\bm{x}}
\newcommand{\bb}{\bm{b}}
\newcommand{\bc}{\bm{c}}

\newcommand{\be}{\bm{e}}

\newcommand{\bu}{\bm{u}}
\newcommand{\bz}{\bm{z}}
\newcommand{\ba}{\bm{a}}
\newcommand{\by}{\bm{y}}

\newcommand{\bY}{\bm{Y}}

\newcommand{\bd}{\bm{d}}
\newcommand{\bxi}{\bm{\xi}}

\DeclareMathOperator*{\argmin}{\mathrm{argmin}}

\DeclareMathOperator*{\Diag}{\mathrm{Diag}}
\newcommand{\calX}{\mathcal{X}}
\newcommand{\calY}{\mathcal{Y}}
\newcommand{\calB}{\mathcal{B}}
\newcommand{\calL}{\mathcal{L}}

\newcommand{\calS}{\mathcal{S}}

\newcommand{\calM}{\mathcal{M}}
\newcommand{\calO}{\mathcal{O}}
\newcommand{\calI}{\mathcal{I}}
\newcommand{\calE}{\mathcal{E}}

\newcommand{\bD}{\bm{D}}

\newcommand{\bv}{\bm{v}}
\newcommand{\bw}{\bm{w}}

\newcommand{\sumK}{\sum_{k=0}^{K-1}}

\newcommand{\bI}{\bm{I}}

\newcommand{\bS}{\bm{S}}

\newcommand{\bW}{\bm{W}}
\newcommand{\blambda}{\bm{\lambda}}

\newcommand{\bs}{\bm{s}}

\newcommand{\tf}{\mathbf{f}}

\newcommand{\eproof}{\hfill $\square$}

\newif\ifshowrevisioncolors
\showrevisioncolorstrue

\numberwithin{equation}{section}
\Crefname{algorithm}{Algorithm}{Algorithms}
\newcommand{\hidecomment}[1]{}

\icmltitlerunning{MoSSP for Nonconvex Constrained DC-Regularized Optimization}

\begin{document}

\twocolumn[
  \icmltitle{MoSSP: A Momentum-Based Single-Loop Stochastic Penalty Method for Nonconvex Constrained DC-Regularized Optimization}
  \begin{icmlauthorlist}
\icmlauthor{Luxuan Li}{buaa}
\icmlauthor{Chunfeng Cui}{buaa}
\icmlauthor{Xiao Wang}{sysu}
\end{icmlauthorlist}

\icmlaffiliation{buaa}{School of Mathematical Sciences, Beihang University, Beijing 100191, China}
\icmlaffiliation{sysu}{School of Computer Science and Engineering, Sun Yat-sen University, Guangzhou 510006, China}
\icmlcorrespondingauthor{Xiao Wang}{wangx936@mail.sysu.edu.cn}

\icmlkeywords{Nonsmooth optimization, DC optimization, Constrained optimization, Penalty methods, Smoothing approximation, Momentum, Oracle complexity}

  \vskip 0.3in
]

\printAffiliationsAndNotice{}
\begin{abstract}
In this paper, we study a structured class of nonconvex constrained stochastic problems with difference-of-convex (DC) regularization, where the feasible set is possibly nonconvex and the concave part of the DC regularizer is allowed to be nonsmooth.
The fundamental challenge lies in maintaining feasibility for nonconvex constraints while achieving favorable oracle complexity.
Although single-loop algorithms efficiently solve unconstrained DC optimization problems, their potential for constrained optimization with DC structure remains largely unexplored. To address this gap, we develop \textbf{MoSSP}, a \textbf{Mo}mentum-based \textbf{S}ingle-loop \textbf{S}tochastic \textbf{P}enalty method for such problems with provable complexity guarantees. The key idea is to apply a single stochastic proximal-gradient step to the Moreau envelope of the penalty plus the convex DC part, with the concave part's proximal mapping computed in parallel. We derive two algorithm variants: a Polyak-momentum version with $\calO(\varepsilon^{-4})$ oracle complexity for finding stochastic $\varepsilon$-KKT points, and an improved $\calO(\varepsilon^{-3})$ version incorporating recursive momentum. Experimental results demonstrate the effectiveness of the proposed algorithms.
\end{abstract}
\section{Introduction}
In this paper, we consider a class of nonconvex constrained stochastic difference-of-convex (DC)-regularized optimization problems
\begin{align}
\label{Prob: primal problem with DC in one block}
   \min_{\bx \in \RR^n } \quad  & F(\bx) = \{f(\bx) = \EE_{\xi}[\tf(\bx,\xi)]\} + h(\bx) - g(\bx), \notag  \\
   \text{s.t.} \quad & \bc(\bx) = \bm{0},  \tag{{\bf P}}
\end{align}
where $\xi$ is a random variable on a probability space $\Xi$, independent of $\bx$. The function $f:\RR^n\!\to\!\RR$ and mapping $\bc:\RR^n\!\to\!\RR^m$ are continuously differentiable, while $h:\RR^n\!\to\!\RR\cup\{+\infty\}$ and $g:\RR^n\!\to\!\RR$ are proper, closed, convex, and possibly nonsmooth functions. The feasible set is assumed to be nonempty.
The proximal mappings of \(h\) and \(g\) are assumed to be available
individually. Computing exact gradients of $f$ is often prohibitive because evaluating the expectation can be costly or the distribution of $\xi$ is not explicitly known. Therefore, we assume only access to stochastic gradients of $f$ at queried points.

Problem~\eqref{Prob: primal problem with DC in one block} captures a wide
range of applications in machine learning and statistical learning, where
\(f\) is a data-fidelity loss and the DC structure appears in the regularizer \(h-g\), which promotes desirable structures such as sparsity; see
\citep[Table~1]{gong2013general} and~\citep{xu2019stochastic,wen2018proximal}. Nonconvex constraints arise naturally from structure, resource, or
safety requirements; examples include energy budgets in DNN compression
\citep{yang2019ecc,chen2018constraint}, safety constraints in reinforcement
learning \citep{paternain2019constrained,zhang2025exchange,zhang2026augmented}, and low-rank or sphere constraints
\citep{Roy2018GeometryAC,witten2009penalized}. Crucially, we do not assume access to the proximal operator of the whole term \(h - g\); for instance, neither the truncated $\ell_{1-2}$ regularizer \citep{ma2017truncated} nor the truncated \(\ell_1\) regularizer \citep{luo2015new} admits a closed-form proximal mapping. Moreover, \(g\) in Problem \eqref{Prob: primal problem with DC in one block} is not necessarily differentiable. Prominent examples include the capped \(\ell_1\) regularization model \citep{gong2013general} and the \(\ell_{1-2}\) regularization model \citep{yin2015minimization} commonly used in compressed sensing.

\subsection{Motivating Examples}
We provide two motivating examples that can be formulated as instances of Problem~\eqref{Prob: primal problem with DC in one block}.

\textbf{DNN training under energy budgets} \citep{yang2019ecc}.
In energy-aware structured pruning, let $\bW=\{\bw_u\}_{u=1}^L$ and 
$\bS=\{s_u\}_{u=1}^L$ denote the layer-wise weight tensors and 
sparsity-level variables. A typical formulation is 
\begin{equation*}
\begin{aligned}
\min_{\bW,\bS} \quad &
f(\bW) + \lambda R_{\rm DC}(\bW), \\
\text{s.t.} \quad &
\phi(\bw_u) \leq s_u,\quad
\psi(\bS) \leq E_{\rm budget},\quad \forall u,
\end{aligned}
\end{equation*}
where $f(\bW):=\mathbb{E}_{\xi}[\tf(\bW;\xi)]$ is the expected training 
loss, $\phi(\bw_u)$ measures layer-wise sparsity, $\psi(\bS)$ models 
total energy consumption, and $R_{\rm DC}(\bW)$ is a DC-structured 
sparsity regularizer (e.g., capped-$\ell_1$ penalty). Both $\phi$ and 
$\psi$ can be nonlinear and nonconvex.

\textbf{Nonnegative sparse CCA} \citep{witten2009penalized}.
For paired data $\bX\in\RR^{N\times p}$ and $\bY\in\RR^{N\times q}$, 
sparse nonnegative canonical loading vectors can be obtained via
\begin{equation*}
\begin{aligned}
\min_{\bu,\bv} \quad &
-\bu^\top \bX^\top \bY \bv
+ \lambda_1 \rho(\bu) + \lambda_2 \rho(\bv), \\
\text{s.t.} \quad &
\|\bu\|_2^2 = 1,\quad
\|\bv\|_2^2 = 1,\quad
\bu, \bv \geq \zero,
\end{aligned}
\end{equation*}
where $\rho(\cdot)$ is a DC-structured sparsity penalty (e.g., capped 
$\ell_1$). The problem is nonconvex due to the bilinear objective, 
unit-sphere constraints, and the nonconvex regularizer.
\subsection{Related Work}
For DC optimization problems, prior work has largely focused on unconstrained or convex-set constrained settings, where feasibility can be maintained by projection. The classical DC algorithm (DCA) \citep{tao1986algorithms} linearizes the concave part and solves a sequence of convex subproblems. Following this idea of constructing tractable convex majorants, proximal, Bregman, and stochastic variants have been developed with global convergence or complexity guarantees \citep{wen2018proximal,liu2019refined,yang2025nonmonotone,liu2022inexact,yang2025inexact,nitanda2017stochastic,xu2019stochastic}. Handling general functional constraints, however, requires additional care. Recent efforts address convex inequality constraints in Problem \eqref{Prob: primal problem with DC in one block} via convex approximations of the feasible region, obtaining favorable convergence guarantees under suitable constraint qualifications in the deterministic setting; see \citep{kanzow2025adaptive,liu2025convergence,yu2021convergence}. The difficulty becomes more pronounced when nonconvex constraints are present, since feasibility is generally difficult to maintain. While there are algorithms tailored to specific nonconvex constraints, such as manifolds \citep{bergmann2024difference,jiang2025inexact} or conic sets \citep{xu2025smoothing}, such analyses are inherently geometry-dependent and do not extend easily to general nonconvex constraints. More recently, \citet{LeThiHuynhDinh2024MOR} studied DC composite optimization under nonconvex constraints via an exact penalty method, but non-asymptotic analysis in the stochastic setting remains unaddressed.

Note that most of the aforementioned methods are double-loop algorithms, in which a subproblem must be (approximately) solved in inner iterations. As a result, they can be relatively complex due to extensive hyperparameter tuning (e.g., the penalty parameter in \citep[Algorithm 3,][]{LeThiHuynhDinh2024MOR}) and the need for precise termination criteria for subproblem solvers \citep[e.g.,][]{yu2021convergence,liu2022inexact}. The growing scale of data and models motivates the need for more efficient optimization methods. In this context, single-loop designs become attractive, especially in stochastic settings.
For Problem \eqref{Prob: primal problem with DC in one block} with \(g\equiv 0\) or \(h-g\equiv 0\), single-loop stochastic penalty algorithms, including augmented Lagrangian methods, already provide compelling oracle complexity guarantees for finding approximate KKT points \citep{Alacaoglu2023ComplexityOS,shi2025momentum,liu2025single,lu2024variancereducedfirstordermethodsdeterministically}. Nevertheless, the effectiveness of such single-loop approaches normally depends on tractable subproblems, a condition typically met when \(h\) admits an efficiently computable proximal mapping. In Problem \eqref{Prob: primal problem with DC in one block}, however, the term \(h - g\) is generally non-prox-friendly. Consequently, single-loop algorithms capable of handling such nonsmooth DC regularizations under general nonconvex constraints remain scarce.

To address this inherent nonsmoothness in the DC structure, some recent work has explored alternative smoothing techniques. Specifically, several studies \citep{moudafi2021complete,sun2023algorithms,hu2024single,chayti2025stochastic} apply the Moreau envelope to each convex component in the DC structure and then take their difference, yielding the Difference-of-Moreau-Envelopes (DME) for non-asymptotic convergence analysis. The DME approach preserves the global DC structure, and stationary points are recoverable via the inexpensive proximal operator of each original component. Among these works, the method of \citet{hu2024single} is, to the best of our knowledge, the first single-loop method for stochastic difference-of-weakly-convex (DWC) optimization based on the DME idea, and achieves $\calO(\varepsilon^{-4})$ oracle complexity. More recently, \citet{chayti2025stochastic} introduced momentum into stochastic DC optimization by applying it to the concave component and obtained $\calO(\varepsilon^{-4})$ oracle complexity. Despite these advances, most existing methods still rely on double-loop schemes to achieve state-of-the-art non-asymptotic rates \citep{nitanda2017stochastic,xu2019stochastic}, or are confined to unconstrained problems; see \citep{nitanda2017stochastic,xu2019stochastic,hu2024single,chayti2025stochastic}. While \citet{sun2023algorithms} developed an augmented Lagrangian algorithm for linearly constrained DC problems in a deterministic setting, such constructions can become unstable under stochastic noise due to sensitive multiplier updates, thus hindering their practical adoption.\footnote{A broader review is deferred to Appendix~\ref{appendix:related_work}.}

In light of these limitations and to facilitate practical implementation, we develop an efficient single-loop, penalty-based method for solving Problem~\eqref{Prob: primal problem with DC in one block}. Inspired by the DME technique, the proposed method smooths the penalized formulation and solves the resulting subproblems using stochastic proximal gradient updates. A refined analysis is then developed to establish favorable complexity guarantees for the proposed method. The key contributions of this paper are summarized as follows.
\begin{itemize}
\item[$\blacktriangleright$]\textbf{Simple single-loop framework.} We propose \textbf{MoSSP}, a \textbf{mo}mentum-based \textbf{s}ingle-loop \textbf{s}tochastic
\textbf{p}enalty framework for solving Problem~\eqref{Prob: primal problem with DC in one block}. By fully leveraging the proximal operators of the components in the DC regularizer, MoSSP avoids solving inner subproblems and extensive hyperparameter tuning while flexibly integrating advanced variance reduction techniques.
\item[$\blacktriangleright$]\textbf{Comparable complexity guarantee.} To the best of our knowledge, we establish the first complexity results for nonconvex DC-regularized optimization with nonlinear constraints. Under mild conditions, \textbf{MoSSP-P} (with Polyak momentum) achieves $\calO(\varepsilon^{-4})$ oracle complexity for finding a stochastic $\varepsilon$-KKT point, while the recursive momentum variant, \textbf{MoSSP-R}, attains an improved $\calO(\varepsilon^{-3})$ rate under the mean-squared smoothness assumption, matching the lower bound of unconstrained stochastic optimization~\cite{arjevani2023lower} under the same assumption. A comparison with existing algorithms for solving (un)constrained DC(-regularized) problems is shown in \Cref{tab.com}.
\item[$\blacktriangleright$] \textbf{Novel theoretical analysis.} We develop a DC-aware complexity analysis that explicitly characterizes stochastic errors within the criticality measure. By constructing a DC-specific potential function, our analysis achieves coordinated control of smoothing, momentum, and penalty parameters, thereby maintaining near-optimal complexity compared to the unconstrained setting.
\end{itemize}

\section{Preliminaries}
\label{Section: Notations and Preliminaries}
\begin{table*}[t]
\caption{\small Comparison of algorithms for (un)constrained DC{(-regularized)} optimization in stochastic and deterministic settings.}
\label{tab.com}
\begin{center}
\resizebox{\textwidth}{!}{%
\begin{tabular}{lcccccc}
\toprule\footnotesize
Algorithm & {Cons.} Type & DC Struct. & Stoch. Assump. & Single\mbox{-}Loop & Iter. Comp. & Oracle Comp. \\
\midrule
SPD \cite{nitanda2017stochastic} & -- & $\checkmark^*$ & \(L\)-sm & -- & $\calO(\varepsilon^{-4})$ & -- \\
SSDC\mbox{-}SPG \cite{xu2019stochastic} & --  & $\checkmark$  &  $\nu$-H\"older &  -- & $\calO(\varepsilon^{-4/\nu})$  & -- \\
SMAG \cite{hu2024single} & --  & $\checkmark^*$  &  -- &  $\checkmark$ & $\calO(\varepsilon^{-4})$ & $\calO(\varepsilon^{-4})$ \\
Algo.~2\mbox{-}Polyak \cite{chayti2025stochastic} & --  & $\checkmark$  &  \(L\)-sm &  $\checkmark$ & $\calO(\varepsilon^{-4})$ & $\calO(\varepsilon^{-4})$ \\
Algo.~2\mbox{-}Recursive \cite{chayti2025stochastic} & --  & $\checkmark$  &  Assump.~3.1&  $\checkmark$ & $\calO(\varepsilon^{-4})$ & $\calO(\varepsilon^{-4})$ \\
\midrule
CLCDC\mbox{-}ALM \cite{sun2023algorithms} & linear & $\checkmark$  &  --    & -- & $\tilde{\calO}(\varepsilon^{-3})$ & -- \\
iMBAdc \cite{liu2025convergence} & cvx  & $\checkmark$  &  --   & -- & {$\calO(\varepsilon^{-2})$} & -- \\
MLALM \cite{shi2025momentum}   & ncvx  & -- & Assump.~3.1 & $\checkmark$ &  $\calO(\varepsilon^{-3})$ & $\calO(\varepsilon^{-3})$\\
\textbf{MoSSP\mbox{-}P (Algo.~\ref{alg:mossp-p})}
& \textbf{ncvx}  & $\checkmark$  &  {\(L\)}-sm & $\checkmark$  &  $\calO(\varepsilon^{-4})$ & $\calO(\varepsilon^{-4})$\\
\textbf{MoSSP\mbox{-}R (Algo.~\ref{alg:mossp-r})}
& \textbf{ncvx}  & $\checkmark$  &  Assump.~3.1 & $\checkmark$  & \textbf{$\calO(\varepsilon^{-3})$}  & \textbf{$\calO(\varepsilon^{-3})$}\\
\bottomrule
\end{tabular}}%
\par\vspace{1pt}
\begin{minipage}{\textwidth}
\footnotesize
\textbf{Notes.}\;
Abbreviations: Cons.~Type = constraint type; DC Struct.~= DC structure;
cvx = convex; ncvx = nonconvex; $\tilde{\calO}(\cdot)$ hides polylogarithmic factors.\;
{\(\checkmark^{*}\) marks related but different unconstrained DC-type
models: SPD~\citep{nitanda2017stochastic} studies stochastic DC programs
\(\min_x g(x)-h(x)\) with differentiable convex components, whereas
SMAG~\citep{hu2024single} studies stochastic DWC optimization under
weak-convexity assumptions; neither covers
the nonsmooth DC-regularized setting in
Problem~\eqref{Prob: primal problem with DC in one block}.}\;
Iter.~Comp.\ counts the number of (outer) iterations; Oracle Comp.\ counts the total
number of stochastic first-order oracle calls; Stoch.~Assump.\ specifies the stochastic
assumption, with deterministic methods (CLCDC-ALM \cite{sun2023algorithms}, iMBAdc \cite{liu2025convergence}) requiring none;
\(L\)-sm $=$ expected Lipschitz smoothness; Assump.~3.1 $=$ mean-squared
smoothness assumption; $\nu$-H\"older $=$ gradient is $\nu$-H\"older continuous.\;
For iMBAdc, the $\calO(\varepsilon^{-2})$ bound counts only outer iterations;
the total number of inner iterations is not quantified.\;
For Algo.~2\mbox{-}Polyak/Recursive \cite{chayti2025stochastic}, the concave component is assumed smooth,
while the convex component is only required to have bounded subgradients; applying
recursive momentum to the smooth concave part improves its per-sample oracle complexity
to $\calO(\varepsilon^{-3})$, but the convex-component queries remain
$\calO(\varepsilon^{-4})$ and therefore dominate the total complexity.
\end{minipage}
\end{center}
\vskip -0.1in
\end{table*}

\noindent\textbf{Notations.}
We use $\|\cdot\|$ for the Euclidean norm and $\langle\cdot,\cdot\rangle$ for the
Euclidean inner product. The gradient of a differentiable function $f$ at $\bx$ is
denoted by $\nabla f(\bx)$. For the constraint mapping $\bc:\RR^n\to\RR^m$, we use $\nabla\bc(\bx)$ to denote the transpose of its Jacobian, i.e., $\nabla\bc(\bx):=J_{\bc}(\bx)^\top\in\RR^{n\times m}$. For any point $\bx\in\RR^n$ and any set
$\calS\subseteq\RR^n$, $\mathrm{dist}(\bx,\calS):=\inf_{\by\in\calS}\|\bx-\by\|$
denotes the point-to-set distance; for $\calX,\calY\subseteq\RR^n$,
$\mathrm{dist}(\calX,\calY):=\inf_{\bx\in\calX,\,\by\in\calY}\|\bx-\by\|$.
For an extended-real-valued function $\varphi$, $\partial \varphi(\bx)$ denotes the
general (limiting) subdifferential, which reduces to the convex subdifferential when
$\varphi$ is convex. Let $\xi^{[k]}=\{\xi^{0},\ldots,\xi^{k}\}$ be the collection of i.i.d.
samples drawn up to iteration $k$. We use $\EE[\cdot]$ for expectation and
$\EE[\cdot\mid\xi^{[k]}]$ for conditional expectation given the sample history.
\subsection{Criticality in DC Optimization}
Consider the unconstrained DC optimization problem:
\begin{equation}
    \min_{\bx \in \RR^n} \Psi (\bx) := \phi(\bx) - g(\bx),
    \label{Eq: unconstrained DC problem}
\end{equation}
where $\phi: \RR^n \to \RR \cup \{+\infty\}$ is $m_{\phi}$-weakly convex (possibly nonsmooth), and $g: \RR^n \to \RR \cup \{+\infty\}$ is a proper, closed, and convex function.

As is well known, the \textit{Moreau envelope} provides favorable smoothing properties for weakly convex functions. For the component function $\phi$ and any $\mu \in (0, 1/m_{\phi})$, its Moreau envelope $\calM_{\mu\phi}$ and the associated proximal mapping are well-defined as:
\begin{align}
\calM_{\mu\phi}(\bz) &:= \min_{\bx \in \RR^n} \left\{ \phi(\bx) + \frac{1}{2\mu} \| \bx - \bz\|^2 \right\}, \notag \\
\prox_{\mu\phi}(\bz) &:=\arg\min_{\bx \in \RR^n} \left\{ \phi(\bx) + \frac{1}{2\mu} \| \bx - \bz \|^2 \right\}, \label{Eq: the proximal operator}
\end{align}
respectively.
Note that $\calM_{\mu \phi}$ serves as a smooth surrogate of $\phi$ with gradient \citep{Moreau1965}
\begin{align}
         &\nabla \calM_{\mu \phi}(\bz) \notag \\
         =&\, \mu^{-1}(\bz - \prox_{\mu \phi}(\bz)) \in \partial \phi(\prox_{\mu \phi}(\bz)).\label{Eq: the gradient of Moreau envelope}
\end{align}
For the weakly convex optimization problem $\min \phi(\bx)$, given \(\varepsilon > 0\), one typically seeks an $\varepsilon$-stationary point of its Moreau envelope $\calM_{\mu\phi}$, i.e., \(\bar{\bx}\) with \(\| \nabla\calM_{\mu\phi}(\bar{\bx})\| \leq \varepsilon \), as a relaxed convergence criterion \cite{davis2019stochastic}. However, this approach cannot be directly extended to Problem \eqref{Eq: unconstrained DC problem}, as the objective $\Psi$ may lack weak convexity. We thus introduce the notion of an \emph{$\varepsilon$-critical point} for DC optimization \cite{sun2023algorithms}, i.e., a point $\bar{\bx}\in \RR^{n}$ is an \emph{$\varepsilon$-critical point} of $\Psi$ if there exist
$\bar{\by} \in \RR^n$ and $\bar{\bu}\in \partial \phi (\bar{\bx}) - \partial g(\bar{\by})$ such that \(\max \{\|\bar{\bu}\|, \|\bar{\bx}- \bar{\by}\|\} \leq \varepsilon\). This two-point formulation serves as a natural stopping criterion for DCA-type methods in practice, since $\partial \phi$ and $\partial g$ are usually evaluated at different points at each iteration in standard DCA; for example, a subgradient $\xi^k \in \partial g(x^k)$ is used to compute $x^{k+1}$. Note that when \(\varepsilon = 0 \), this definition recovers the classical DC criticality condition \cite{pang2017computing}.

Moreover, the overall Moreau envelope of $\Psi$ in Problem \eqref{Eq: unconstrained DC problem} is computationally intractable. To this end, we apply the Moreau envelope to $\phi$ and $g$ individually and take their difference (DME) to define a smooth approximation of Problem \eqref{Eq: unconstrained DC problem} for any $\mu \in (0, 1/m_{\phi})$:
\begin{align}
\label{Eq: the DME smoothing problem}
 \min\limits_{\bz \in \RR^n }\Psi_{\mu}(\bz) := \calM_{\mu \phi}(\bz) - \calM_{\mu g}(\bz).
\end{align}
The smoothness of \( \Psi_{\mu}(\bz)\) was
established in \citet{hiriart1991regularize}.

However, it is not immediately clear how the approximate solutions of 
Problems \eqref{Eq: the DME smoothing problem} and \eqref{Eq: unconstrained DC problem} relate to each other. We establish that any $\varepsilon$-stationary point of $\Psi_{\mu}$ 
can be converted into an $\varepsilon$-critical point of $\Psi$, with the detailed proof deferred to Appendix~\ref{App: proof of prop2.2}. This correspondence plays a key role in our analysis. Complete properties of DME and the equivalence between solutions of these two problems are deferred to Appendices~\ref{Subsection: Difference-of-Moreau-Envelopes} and~\ref{App: Solution Correspondence Theory}, respectively.
\begin{proposition}
\label{Prop: approximate correspondence}
Given $\varepsilon > 0$, for any $0 < \mu < 1/m_{\phi}$, if $\bar{\bz}$ satisfies $\| \nabla \Psi_{\mu}(\bar{\bz})\| \leq \min\{1, \mu^{-1}\} \varepsilon$, then the point \(\bar{\bx} := \prox_{\mu \phi}(\bar{\bz})\) is an \textit{$\varepsilon$-critical point} of $\Psi$ with the auxiliary point \(\bar{\by} := \prox_{\mu g}(\bar{\bz}) \).
\end{proposition}
\subsection{Approximate Solutions for Constrained DC-Regularized Optimization}
\label{Subsection: Approximate solutions for constrained DC problem}
As is standard in constrained nonconvex optimization, we seek points satisfying the KKT conditions of Problem \eqref{Prob: primal problem with DC in one block}. Due to the stochasticity in the objective, it is natural to evaluate the optimality residuals in expectation. We next define two types of $\varepsilon$-approximate stochastic KKT solutions for Problem \eqref{Prob: primal problem with DC in one block}, 
together with the associated criticality measures used in our analysis. A discussion of \Cref{Def: the definition of approximate point} is provided in Appendix~\ref{Appendix: Discussion on Def 2.1}.
\begin{definition}
\label{Def: the definition of approximate point}
Given $\varepsilon > 0$, a point $\bar{\bx}$ is a stochastic $\varepsilon$-KKT point of Problem \eqref{Prob: primal problem with DC in one block} if there exist $(\bar{\bu}, \bar{\by}, \bar{\blambda}) \in \RR^n\times \RR^n\times \RR^m$ such that {almost surely,}
\begin{align}
\label{Def: the definition of u}
\bar{\bu}\in \nabla f(\bar{\bx}) + \nabla \bc(\bar{\bx}) \bar{\blambda}+ \partial h(\bar{\bx}) - \partial g(\bar{\by}),
\end{align}
and
\begin{align}
\label{Eq: approximate KKT point}
\begin{cases}
\max\{ \EE [\|\bar{\bu}\|^{2}], \EE [\|\bar{\bx}-\bar{\by}\|^{2} ]\} \leq \varepsilon^{2}  \quad &\text{{\bf (criticality)}}, \\
\EE [\|\bc(\bar{\bx})\|^{2} ]\leq \varepsilon^{2} \quad &\text{{(\bf feasibility)}}.
\end{cases}
\end{align}
A point $\bar{\bx}$ is a stochastic $\varepsilon$-stationary point of Problem \eqref{Prob: primal problem with DC in one block} if there exist $(\bar{\bu}, \bar{\by}, \bar{\blambda})$ satisfying \eqref{Def: the definition of u} {almost surely} and
 \begin{align}
    \label{Eq: approximate stationarity point}
\begin{cases}
\max\{\EE[\|\bar{\bu}\|^{2}], \EE[\|\bar{\bx}-\bar{\by}\|^{2}]\} \leq \varepsilon^{2}  & \quad\quad\text{{\bf (criticality)}},\\
\EE[\|\nabla \bc(\bar{\bx}) \bc(\bar{\bx})\|^{2}] \leq \varepsilon^{2} & \!\!\! \!\!\!\!\!\! \!\!\! \!\!\! \!\!\! \!\text{{\bf (infeasible stationarity)}}.
\end{cases}
\end{align}
\end{definition}
\subsection{Main Assumptions}
\label{Subsection: main assumptions}
Throughout the paper, we make the following assumptions for Problem \eqref{Prob: primal problem with DC in one block}. {We assume $F^*:=\inf_{\bx\in\RR^n} F(\bx)>-\infty$.}
\begin{assumption}
    \label{Ass: boundedness of c and f} Function $f(\cdot)$ is
    $L_f$-smooth. Function \(\bc\) is \(L_c\)-smooth. There exist $C, G>0$ such that
    \begin{align}
         &\|\nabla f(\bx) \| \leq G,\,
         \sup_{\bv_h\in\partial h(\bx)}\|\bv_h\| \leq G,\,
         \sup_{\bv_g\in\partial g(\bx)}\|\bv_g\| \leq G,\\ \notag
         &\|\nabla \bc(\bx) \| \leq G,\quad
         \| \bc(\bx) \| \leq C,\quad \forall \bx\in\RR^n.\notag
    \end{align}
\end{assumption}
The boundedness condition is crucial for ensuring reliable convergence of the iterative sequence in stochastic constrained optimization. The inherent randomness of the process makes it difficult to guarantee that all iterates remain within a specific level set. The necessity of \Cref{Ass: boundedness of c and f} is well established in \cite{BeraCurtRobiZhou21, na2021inequality, NaAnitKola22, wang2025complexity,sun2023algorithms}.
\begin{assumption}
    \label{Ass: unbiased gradient oracle} There exists a constant $\sigma > 0$ such that
    \begin{align}
    &\EE_{\xi}[\nabla \tf(\bx, \xi)] = \nabla f(\bx), \notag \\
    & \EE_{\xi}[\| \nabla \tf(\bx, \xi) - \nabla f(\bx) \|^{2}] \leq \sigma^{2}, \quad \forall\bx
    \in \RR^n. \notag
    \end{align}
\end{assumption}
For nonconvex constraints, we impose the following nonsingularity condition on the constraints to control feasibility of the generated solutions. Constraint qualifications are often required in solving constrained optimization problems; see~\citep{shi2025momentum,lu2024variancereducedfirstordermethodsdeterministically,liu2025single,curtis2024worst}.
\begin{assumption}
    \label{Ass: constraints conditions}
    For the iterate sequence $\{\bx^k\}_{k\in\NN}$ generated by the algorithms, there exists a constant $\delta > 0$ such that
    \begin{align*}
        \|\nabla \bc(\bx^k) \bc(\bx^k)\| \geq \delta \| \bc(\bx^k) \|, \quad \forall\, k \geq 0.
    \end{align*}
\end{assumption}

\section{Momentum-Based Single-Loop Stochastic Penalty Algorithms}
\label{sec:algorithm}
\label{Section: algorithms and main complexity results}
\subsection{Algorithmic Framework}
\label{Subsection: Algorithmic Framework}
We incorporate the nonconvex constraints in Problem \eqref{Prob: primal problem with DC in one block} through a sequence of quadratically penalized DC objectives with nondecreasing penalty parameters $\{\rho_k\}\subset(0,\infty)$:
\begin{align}
\label{Eq: the quadratic penalty problem}
\min_{\bx \in \RR^n} \big\{ F_{\rho_k}(\bx) := \underbrace{Q_{\rho_k}(\bx) + h(\bx)}_{\psi_{\rho_k}(\bx)} - g(\bx) \big\},
\end{align}
where $Q_{\rho}(\bx) = f(\bx) + \frac{\rho}{2} \|\bc(\bx)\|^2$. To address the nonsmoothness of the penalty problems, we take the difference of the Moreau envelopes of each component to construct a smoothed surrogate function $F_{\rho,\mu}(\bz)$:
\begin{align}
\label{Eq: the DME smoothing surrogates for penalty function}
F_{\rho,\mu}(\bz) = \calM_{\mu \psi_{\rho}}(\bz) - \calM_{\mu g}(\bz), \quad \text{with } \rho = \rho_k.
\end{align}
Crucially, by Proposition~\ref{Prop: approximate correspondence}, if $\bar{\bz}$ satisfies $\|\nabla F_{\rho,\mu}(\bar{\bz})\| \leq \min\{1, \mu^{-1}\} \varepsilon$, then any $\bar{\bx}$ with $\|\bar{\bx} - \prox_{\mu g}(\bar{\bz})\| \leq \varepsilon$ is an $\varepsilon$-critical point of $F_\rho$. 
This gives the approximate criticality criterion targeted in \eqref{Eq: approximate KKT point} or \eqref{Eq: approximate stationarity point}. Therefore, we aim to efficiently find a point $\bar{\bz}$ by iteratively minimizing $F_{\rho,\mu}(\bz)$.

A standard strategy is to apply gradient descent to \(F_{\rho,\mu}(\bz)\), which inspires our framework. Following the properties of the Moreau envelope, the gradient of \(F_{\rho,\mu}\) at the point \(\bz^k\) is given by
\begin{equation}
\label{Eq: gradient of DME}
\begin{aligned}
\nabla F_{\rho,\mu}(\bz^k)
&= \mu^{-1} (\bz^k - \prox_{\mu \psi_\rho}(\bz^k)) \\
&\quad - \mu^{-1}(\bz^k-\prox_{\mu g}(\bz^k)) \\
&= \mu^{-1} \bigl(\prox_{\mu g}(\bz^k)
- \prox_{\mu \psi_\rho}(\bz^k)\bigr).
\end{aligned}
\end{equation}
This involves the proximal operators of \(g\) and \(\psi_\rho\). However, 
computing the exact proximal operator of the composite function $\psi_{\rho}$ is typically intractable. To this end, we maintain a variable \(\bx\) as an estimator of \(\prox_{\mu \psi_{\rho}}(\bz)\) and compute \(\prox_{\mu g}(\bz)\) in parallel. Specifically, at each iteration, we update \(\bx^k\) via a stochastic proximal gradient descent step
\begin{equation}
\label{Eq: The basic update of x}
\begin{aligned}
\bx^{k+1}
&= \prox_{\mu_k h}\bigl(\bz^k - \mu_k \tilde{\nabla} Q_{\rho_k}(\bx^k) \bigr) \\
&= \argmin_{\bx \in \RR^n} \Bigl\{
\langle \tilde{\nabla} Q_{\rho_k}(\bx^k) , \bx - \bx^{k} \rangle + h(\bx) \\
&\qquad\qquad\qquad\qquad
+ \frac{1}{2 \mu_{k}} \| \bx - \bz^{k} \|^{2}
\Bigr\},
\end{aligned}
\end{equation}
where \(\tilde{\nabla} Q_{\rho_k}(\bx^k)\) is a stochastic estimator of \(\nabla Q_{\rho_k}(\bx^k)\). Finally, \(\bz^k\) is updated via a gradient-type step, where the gradient \eqref{Eq: gradient of DME} is now partially estimated by replacing \(\prox_{\mu_k \psi_{\rho_k}}(\bz^k)\) with \(\bx^{k+1}\) from \eqref{Eq: The basic update of x}:
\begin{align} 
\label{eq:bz_update}
\bz^{k+1} = \bz^k - \beta \left( \prox_{\mu_k g}(\bz^k) - \bx^{k+1} \right),
\end{align}
where \(\beta > 0\) is the stepsize. All parameter sequences are positive and will be specified later.

Our framework allows for flexible choices of \(\tilde{\nabla} Q_{\rho_k}(\bx^k)\) in the \(\bx\)-update step (see \eqref{Eq: The basic update of x}), enabling the integration of advanced momentum techniques to enhance both theoretical complexity and practical efficiency. We refer to this unified framework as \textbf{MoSSP} (\textbf{Mo}mentum-based \textbf{S}ingle-loop \textbf{S}tochastic \textbf{P}enalty).

\subsection{MoSSP-P: MoSSP with Polyak Momentum}
\label{Subsection: Polyak-momentum based algorithm and its main complexity results}
\begin{algorithm}[t]
    \caption{\textbf{MoSSP-P}: \textbf{S}ingle-Loop \textbf{S}tochastic \textbf{P}enalty Algorithm with \textbf{P}olyak \textbf{Mo}mentum}
    \label{alg:mossp-p}
    \begin{algorithmic}
        \STATE {\bfseries Input:} maximum number of iterations $K$, initial point $\bx^{0} = \bz^{0} \in \RR^{n}$, a sequence $\{\alpha_k\}\subset(0,1)$, positive parameters $\mu_k$, $\rho_k$, and $\beta$.
        \FOR{$k = 0, 1, 2, \ldots K-1$}
        \STATE Sample $\xi^k$ from $\Xi$ and compute $\bS^{k}$ from \eqref{Eq: enhanced gradient estimator Polyak}.
        \STATE Compute $\bx^{k+1}$ using \eqref{eq:bx_update_Polyak}.
        \STATE Compute $\bz^{k+1}$ using \eqref{eq:bz_update}.
        \ENDFOR
        \STATE {\bfseries Output:} $\bx^{R+1}$, where $R \in \{0,1,\ldots,K-1\}$ is uniformly and randomly chosen.
    \end{algorithmic}
\end{algorithm}

Polyak momentum, also known as the heavy-ball method and originally proposed by \citet{polyak1964some}, has been extensively studied in large-scale nonconvex optimization \cite{liu2020,jelassi2022,gao2024non}. Motivated by this, at each iteration $k$, we draw an i.i.d. sample $\xi^k \sim \Xi$ and construct a Polyak momentum-based stochastic estimator for $\nabla Q_{\rho_k}(\bx^k)$:
\begin{align}
\label{Eq: enhanced gradient estimator Polyak}
\bS^{k} = \bs^k + \rho_k \nabla \bc(\bx^{k}) \bc(\bx^{k}),
\end{align}
where $\bs^k$ is updated by
\begin{align}
\bs^{k} = \begin{cases} (1 - \alpha_{k-1}) \bs^{k-1} + \alpha_{k-1} \nabla \tf(\bx^k, \xi^k), & k \ge 1, \\ \nabla \tf(\bx^0, \xi^0), & k=0. \end{cases}
\end{align}
The variable $\bx^{k+1}$ is then updated as:
\begin{align}
\bx^{k+1} = 
\prox_{\mu_k h}\bigl(\bz^k - \mu_k \bS^{k}\bigr), \label{eq:bx_update_Polyak}
\end{align}
where $\mu_k > 0$ for $k \geq 0$. The complete algorithm, referred to as \textbf{MoSSP-P}, is summarized in \Cref{alg:mossp-p}.

\noindent\textbf{Oracle Complexity of MoSSP-P.}\; We now analyze the oracle complexity of
MoSSP-P for finding a stochastic \(\varepsilon\)-KKT point and a stochastic \(\varepsilon\)-stationary point, respectively. The maximum number of iterations is limited to a fixed integer \(K\).

Note that \(Q_{\rho}(\bx)\) is smooth with Lipschitz constant \(L_\rho\), where $L_{\rho}= \rho \tilde{L}$ and $\tilde{L}= {\rho_0^{-1}L_f}+ G^2 + CL_c$, which is crucial for the subsequent analysis (see \Cref{Lemma: smoothness-penalty-function}).

We aim to bound the residuals in \eqref{Eq: approximate KKT point} and \eqref{Eq: approximate stationarity point}. Classical analysis of penalty methods in nonconvex constrained optimization often relies on the descent property of the penalty function. However, the concave component $-g$ in our penalty function \(F_{\rho_k}\) breaks this descent property, rendering the standard potential function construction ineffective. To this end, we construct the following potential function:
\begin{align}
    \calL_{\rho, \mu}(\bw) = Q_{\rho}(\bx) + h(\bx) + \frac{1}{2\mu}\|\bx - \bz\|^{2} - \calM_{\mu g}(\bz),\notag
\end{align}
where $\bw = (\bx,\bz)$. The function $\calL_{\rho, \mu}(\bw)$ is bounded from below (see Appendix~\ref{Sub: constructed potential function}).

To quantify the criticality of iterates \(\bw^k\) generated by MoSSP-P, we define \(\bu^{k+1}\) as:
{
\begin{equation}
\label{Eq: u-define-stochastic-polyak}
\begin{aligned}
\bu^{k+1}
&= \nabla Q_{\rho_k}(\bx^{k+1}) - {\bS^k} + \mu_k^{-1}(\prox_{\mu_k g}(\bz^k)- \bx^{k+1}).
\end{aligned}
\end{equation}}

It can be proved that \(\bu^{k+1}\in \partial \psi_{\rho_k}(\bx^{k+1}) - \partial g(\prox_{\mu_k g}(\bz^k))\) (see Appendix~\ref{Subsection: auxiliary lemmas}).

\Cref{Lemma: criticality-stationarity-measure} gives a one-step bound on $\bu^{k+1}$ relative to the iterate change (see Appendix~\ref{Subsection: auxiliary lemmas}). The stochastic error in \eqref{Eq: the criticality measure} implies that proper parameter settings are required to control variance and ensure favorable complexity. This differs from the deterministic framework for constrained DC optimization of \citet{sun2023algorithms}. We next present an estimated convergence rate of MoSSP-P to guide parameter selection.
\begin{lemma}
    \label{Lemma: the estimation of the convergence rate of Polyak for infeasible stationary point} 
    Suppose Assumptions \ref{Ass: boundedness of c and f} and \ref{Ass: unbiased gradient oracle} hold. Let $\{\bw^{k}\}_{k \in \NN}$ be the sequence generated by MoSSP-P
    with the parameters satisfying
\begin{equation}
    \label{Eq: the parameter conditions for Polyak in main body}
\begin{aligned}
\rho_{k} &\equiv \calO( K^{l}), \quad
\mu_{k}\equiv  \calO(K^{{-\tau}}), \\
\alpha_{k} &\equiv  \calO(K^{{-\tau}}), \quad
0 < \beta \leq 1,
\end{aligned}
\end{equation}
where $0 < l \leq \tau \leq 2 l < 1$ is independent of $K$.
Then, it holds that
{\small
\begin{equation}
    \label{Eq: convergence rate of polyak}
\begin{aligned}
&\max\{ \EE [\| \bu^{R+1}\|^{2}],
\EE[\| \bx^{R+1}- \prox_{\mu_R g}(\bz^{R})\|^{2}]\} \\
&\quad = \calO\!\left(
\max\left\{
\begin{gathered}
K^{\tau -1}\left( \calL_{\rho, \mu}(\bw^{0})+1\right),\\
K^{\tau - 1}\EE [\| \be^{0}\|^{2}],\quad
K^{-\tau}
\end{gathered}
\right\}
\right), \\
&\EE [\|\nabla \bc(\bx^{R+1})\bc(\bx^{R+1})\|^{2}] \\
&\quad = \calO\!\left(
\max\left\{
\begin{gathered}
K^{-2l + \tau -1}\left( \calL_{\rho, \mu}(\bw^{0}) + 1 \right),\\
K^{-2l + \tau -1 }\EE [\| \be^{0}\|^{2}],\quad
K^{-2l}
\end{gathered}
\right\}
\right).
\end{aligned}
\end{equation}}
\end{lemma}
Note that \eqref{Eq: convergence rate of polyak} indicates that the term $\frac{\rho}{2}\|\bc(\bx^{0})\|^{2}$ in $\calL_{\rho, \mu}(\bw^{0})$ may dominate the complexity due to $\rho = \rho_0 K^{l}$. Motivated by \citet{xie2021complexity,jin2022stochastic}, we employ an approximately feasible initialization to mitigate this overhead. Let $\bx^{0}$ be an approximately feasible point satisfying $\|\bc(\bx^{0})\|^{2} = \calO(K^{-l})$ for some $0<l<1$. In practice, such an initialization can be produced by the two-phase procedure proposed by \citet{jin2025stochastic} for weakly convex minimization with nonconvex constraints, where the first phase yields an approximately feasible point that warm-starts the second.
The main complexity result is presented below, with detailed analysis deferred to Appendix~\ref{Subsection: Lemmas for Oracle complexity analysis of Algo 1}.
\begin{theorem}
    \label{Theorem: the convergence rate of polyak with initial feasibility}
Suppose Assumptions \ref{Ass: boundedness of c and f}-\ref{Ass: constraints conditions} hold. Let $\{\bw^k\}_{k \in \NN}$ be the sequence generated by MoSSP-P
with parameters satisfying 
\begin{align}
    \label{Eq: the parameter conditions in Polyak}
    &\rho_{k} \equiv \rho = {\rho_0} K^{1/4}, \quad \mu_{k} \equiv \mu = \frac{\mu_0}{K^{1/2}\max\{L_f, \tilde{L}\}}, \notag \\
    &\alpha_{k} \equiv \alpha = \frac{\alpha_0 \mu_0}{K^{1/2}}, \quad 0 < \beta \leq 1,
\end{align}
where $\rho_0, \mu_0, \alpha_0 > 0$ are constants independent of $K$. If $\| \bc (\bx^{0}) \|^{2}= \calO(K^{-1/4})$, then it holds that
{
    \begin{equation}
    \label{Eq: the convergence rate of Polyak}
\begin{aligned}
&\max\bigl\{ \EE[\|\bu^{R+1}\|^{2}],
\EE[\|\bx^{R+1}-\prox_{\mu g}(\bz^{R})\|^{2}] \bigr\} \\
&\quad = \calO(K^{-1/2}), \\
&\EE\bigl[\|\bc(\bx^{R+1})\|^{2}\bigr]
= \calO(K^{-1/2}).
\end{aligned}
    \end{equation}
}{Consequently, given $\varepsilon > 0$, the oracle complexity of MoSSP-P to obtain a stochastic $\varepsilon$-KKT point $\bx^{R+1}$ satisfying \eqref{Eq: approximate KKT point} is of order $\calO(\varepsilon^{-4})$. In the absence of \Cref{Ass: constraints conditions}, we have $\EE\bigl[\|\nabla \bc(\bx^{R+1})\,\bc(\bx^{R+1})\|^{2}\bigr] = \calO(K^{-1/2})$, and the oracle complexity of MoSSP-P to obtain a stochastic $\varepsilon$-stationary point $\bx^{R+1}$ satisfying \eqref{Eq: approximate stationarity point} is of order $\calO(\varepsilon^{-4})$.}
\end{theorem}
The result in \Cref{Theorem: the convergence rate of polyak with initial feasibility} matches that of unconstrained stochastic first-order methods (SFOMs) for Problem \eqref{Prob: primal problem with DC in one block} with \(\bc = \zero\) and \(g = 0\)~\cite{ghadimi2013stochastic,gao2024non}, as well as SFOMs for the unconstrained stochastic DC case~\cite{chayti2025stochastic,hu2024single,xu2019stochastic}.
\begin{remark}\label{Remark: role of momentum}
Existing analyses of penalty methods, including ALM analyses for nonconvex constrained problems, require penalty parameters to grow inversely with $\varepsilon$ to attain the best-known complexity. This worsens the problem conditioning, as the smoothness constant grows as $L_{\rho} = \Theta(\rho)$. Without momentum, one must compensate with a smaller stepsize or a larger batch size, potentially yielding complexity no better than $\calO(\varepsilon^{-5})$~\cite{jin2022stochastic}; in particular, a naive extension of unconstrained stochastic DC methods (e.g., \cite{hu2024single,xu2019stochastic}) to the penalty framework for Problem~\eqref{Prob: primal problem with DC in one block} risks such degradation. Under our penalty framework, momentum yields a genuine complexity
improvement by allowing the penalty, stepsize, and momentum parameters to be jointly calibrated under $\calO(1)$ batch size, thereby achieving the rate stated above at no additional cost from the nonconvex constraints.
\end{remark}

The oracle complexity of MoSSP-P without an approximately feasible initialization is presented as follows.
\begin{corollary}
    \label{Theorem: the basic convergence rate of Polyak without initial feasibility}
     Under the assumptions of \Cref{Theorem: the convergence rate of polyak with initial feasibility} with parameters set as $\rho = \calO(K^{1/5})$, $\mu = \calO(K^{-2/5})$, $\alpha = \calO(K^{-2/5})$, and $0 < \beta \leq 1$, the oracle complexity of MoSSP-P
     to obtain a stochastic $\varepsilon$-KKT point or a stochastic $\varepsilon$-stationary point (without \Cref{Ass: constraints conditions}) $\bx^{R+1}$ is of order $\calO(\varepsilon^{-5})$.
\end{corollary}
\begin{remark}
\Cref{Theorem: the basic convergence rate of Polyak without initial feasibility} highlights the role of approximate feasible initialization in improving the complexity order. Since the schedule $\rho = \calO(K^{l})$ is required to control $\mathbb{E}\bigl[\|\nabla \bc(\bx^{k}) \bc(\bx^{k})\|^2\bigr]$, it inflates $\tfrac{\rho}{2}\|\bc(\bx^{0})\|^{2}$ to $\calO(K^{l})$ when $\|\bc(\bx^{0})\|^{2} = \calO(1)$, dominating the bound in \eqref{Eq: convergence rate of polyak}. Such initialization confines this term to $\calO(1)$, thereby improving the complexity order. In fact, our analysis can also be extended to an iteration-indexed penalty update \cite{Alacaoglu2023ComplexityOS} without this requirement.
\end{remark}

\subsection{MoSSP-R: MoSSP with Recursive Momentum}
\label{Subsection: Recursive-momentum based algorithm and its main complexity results}
Alternatively, we incorporate another variance reduction technique, recursive momentum (STORM) \cite{cutkosky2019momentum}, into the \(\bx\)-update. STORM introduces a correction term based on a single sample to generate a variance-reduced gradient estimate, making it suitable for large-scale applications.

At iteration \(k\), we construct a recursive momentum-based stochastic estimator for \(\nabla Q_{\rho_k}(\bx^k)\) as: 
\begin{align}
    \label{Eq: gradient estimator storm}
    \bD^k = \bd^k + \rho_k \nabla \bc(\bx^{k}) \bc(\bx^{k}),
\end{align}
where \(\bd^k\) is updated recursively by
\begin{equation}
\label{Eq: the storm estimator}
\bd^k =
\begin{cases}
\frac{1}{|\calB_0|} \sum_{j \in \calB_0}
\nabla \tf(\bx^0,\xi_j^0), & k = 0, \\[3pt]
\begin{aligned}[t]
\nabla \tf(\bx^k, \xi^k)
&+ (1 - \alpha_{k-1}) \\
&\quad \left( \bd^{k-1} - \nabla \tf(\bx^{k-1},\xi^k) \right),
\end{aligned} & k \geq 1.
\end{cases}
\end{equation}

We begin at $k=0$ by independently drawing a sample set $\{\xi_j^0\}_{j\in\mathcal{B}_0}$ from the distribution $\Xi$, where the initial batch size is $|\mathcal{B}_0|=b_0$. For any \(k\ge1\), a single i.i.d. sample \(\xi^k\) is drawn from \(\Xi\). We then update \(\bx^{k+1}\) through
\begin{align}
\label{eq:bx_update_storm}
    \bx^{k+1} 
    = \prox_{\mu_k h}\bigl(\bz^k - \mu_k \bD^{k}\bigr).
\end{align}
The overall algorithm, termed \textbf{MoSSP-R}, is summarized in \Cref{alg:mossp-r}.

\begin{algorithm}[t]
    \caption{\textbf{MoSSP-R}: \textbf{S}ingle-Loop \textbf{S}tochastic \textbf{P}enalty Algorithm with \textbf{R}ecursive \textbf{Mo}mentum}
    \label{alg:mossp-r}
    \begin{algorithmic}
        \STATE {\bfseries Input:} maximum number of iterations $K$, initial point $\bx^{0} = \bz^{0} \in \RR^{n}$, a sequence $\{\alpha_k\}\subset(0,1)$, positive parameters $\mu_k$, $\rho_k$, $\beta$.
        \FOR{$k = 0, 1, 2,\ldots K-1 $}
        \STATE Sample $\xi^k$ from $\Xi$ and compute $\bD^{k}$ from \eqref{Eq: gradient estimator storm}.
        \STATE Compute $\bx^{k+1}$ using \eqref{eq:bx_update_storm}.
        \STATE Compute $\bz^{k+1}$ using \eqref{eq:bz_update}.
        \ENDFOR
        \STATE {\bfseries Output:} $\bx^{R+1}$, where $R \in \{0,1,\ldots,K-1\}$ is uniformly chosen.
    \end{algorithmic}
\end{algorithm}
\noindent\textbf{Oracle Complexity of MoSSP-R.}\; We now present complexity results for MoSSP-R for finding an approximate stationary point and an approximate KKT point, respectively, with detailed analysis deferred to Appendix~\ref{Subsection: Lemmas for Oracle complexity analysis of Algo 2}. In MoSSP-R, we construct $\bu^{k+1}$ as:
{
\begin{equation}
\label{Eq: u-define-stochastic-storm}
\begin{aligned}
\bu^{k+1}
&= {\nabla Q_{\rho_k}(\bx^{k+1}) - \bD^k} + \mu_k^{-1}(\prox_{\mu_k g}(\bz^k)- \bx^{k+1}).
\end{aligned}
\end{equation}}

Before proceeding, we introduce an additional standard assumption, also known as the mean-squared smoothness assumption, widely used in variance reduction methods \cite{nguyen2017sarah,fang2018spider,cutkosky2019momentum,xu2023momentum}.
\begin{assumption}
    \label{Ass: the expected smoothness of f} 
    For almost every $\xi \in \Xi$, $\tf(\cdot, \xi)$ is differentiable and 
    for all $\bx, \by \in \RR^{n}$,
    \[
        \EE_{\xi}\left[\left\|\nabla \tf(\bx; \xi) - \nabla \tf(\by; \xi)\right\|^{2}
        \right] \leq L_f^{2}\|\bx - \by\|^{2}.
    \]
\end{assumption}  
By Jensen's inequality, Assumption~\ref{Ass: the expected smoothness of f} implies that 
\(f=\EE_{\xi}[\tf(\cdot;\xi)]\) is \(L_f\)-smooth, whereas the converse does not hold in general.

As noted by \citet{xu2023momentum}, employing a moderately large initial batch size in variance reduction can improve the complexity order at a negligible cost. Consequently, we
establish the oracle complexity with an initial batch size $b_0 = \calO(K^{1/3})$.
\begin{theorem}
    \label{Theorem: the convergence rate of storm with initial feasibility}
Suppose Assumptions \ref{Ass: boundedness of c and f}-\ref{Ass: constraints conditions} and \ref{Ass: the expected smoothness of f} hold. Let $\{\bw^k\}_{k \in \NN}$ be the sequence generated by MoSSP-R with parameters satisfying
\begin{align}
\rho_{k} &\equiv \rho = {\rho_0} K^{1/3}, \quad \mu_{k} \equiv \mu= \frac{\mu_0}{K^{1/3}\max\{L_f, \tilde{L}\}},\notag \\
\alpha_{k} &\equiv \alpha =\frac{16\alpha_0 \mu_0^{2}}{K^{2/3}},
\quad 0 < \beta \leq 1,
\label{Eq: the parameter conditions in VR for theorem}
\end{align}
where $\rho_0, \mu_0, \alpha_0 > 0$ are constants independent of $K$. If $\| \bc (\bx^{0}) \|^{2}= \calO(K^{-1/3})$ with the initial batch size $b_0 = \calO(K^{1/3})$, then it holds that
\begin{equation}
\begin{aligned}
&\max\bigl\{ \EE[\|\bu^{R+1}\|^{2}],
\EE[\|\bx^{R+1}-\prox_{\mu g}(\bz^{R})\|^{2}] \bigr\} \\
&\quad = \calO(K^{-2/3}), \\
&\EE\bigl[\|\bc(\bx^{R+1})\|^{2}\bigr]
= \calO(K^{-2/3}).
\end{aligned}
    \end{equation}
Consequently, given $\varepsilon > 0$, the oracle complexity of MoSSP-R to obtain a stochastic $\varepsilon$-KKT point $\bx^{R+1}$ satisfying \eqref{Eq: approximate KKT point} is of order $\calO(\varepsilon^{-3})$. In the absence of \Cref{Ass: constraints conditions}, we have $\EE\bigl[\|\nabla \bc(\bx^{R+1})\,\bc(\bx^{R+1})\|^{2}\bigr] = \calO(K^{-2/3})$, and the oracle complexity of MoSSP-R to obtain a stochastic $\varepsilon$-stationary point $\bx^{R+1}$ satisfying \eqref{Eq: approximate stationarity point} is of order $\calO(\varepsilon^{-3})$.
\end{theorem}

The result in \Cref{Theorem: the convergence rate of storm with initial feasibility} matches the $\calO(\varepsilon^{-3})$ lower bound for unconstrained SFOMs with variance reduction~\cite{arjevani2023lower} and the best-known rate for nonconvex constrained stochastic optimization~\cite{shi2025momentum} under \Cref{Ass: the expected smoothness of f}, where the DC structure is absent. Crucially, it demonstrates that our single-loop penalty framework and analysis can effectively handle the DC structure without increasing the complexity order with respect to \(\varepsilon\). 

The following corollary establishes the oracle complexity result for MoSSP-R without assuming initial approximate feasibility.
\begin{corollary}
    \label{Theorem: the basic convergence rate of storm without initial feasibility}
    Under the assumptions of \Cref{Theorem: the convergence rate of storm with initial feasibility} with parameters set as $\rho = \calO(K^{1/4})$, $\mu = \calO(K^{-1/4})$, $\alpha = \calO(K^{-1/2})$, $0 < \beta \leq 1$, the oracle complexity of MoSSP-R to obtain a stochastic $\varepsilon$-KKT point or a stochastic $\varepsilon$-stationary point (without \Cref{Ass: constraints conditions}) $\bx^{R+1}$ is of order $\calO(\varepsilon^{-4})$.
\end{corollary}

\section{Numerical Experiments}
\label{Section: Numerical experiments}
\subsection{Experiment Setup}
We evaluate MoSSP-P (\Cref{alg:mossp-p}) and MoSSP-R (\Cref{alg:mossp-r}) on an equality-constrained binary classification problem with DC regularization \cite{hong2023constrained}:
\begin{align}
\label{eq.libsvm}
\min_{\bx \in \RR^n}&\quad
\frac{1}{N} \sum_{i=1}^N \log(1 + e^{-y_i (X_i^\top \bx)}) + \lambda \left( \|\bx\|_1 - \|\bx\|_2 \right) \notag \\
 \text{s.t.} &\quad \|\bx\|_2^2 = 1,
\end{align}
where \( X_i \in \RR^n \) denotes the feature vector and \( y_i \in \{-1, 1\} \) denotes the binary label for each \( i \in [N] \), with \(N\) denoting the total sample size. We compare our algorithms with two double-loop baselines: SPDC \cite{xu2019stochastic,nitanda2017stochastic} for simple convex-constrained DC(-regularized) optimization and SALM \cite{sun2023algorithms} for linearly constrained DC-regularized optimization. Each baseline is equipped with Polyak momentum (P) or recursive momentum (R), yielding four variants: SPDC-P, SPDC-R, SALM-P, SALM-R. We evaluate on three LIBSVM datasets \cite{chang2011libsvm}: \texttt{a9a} (32,561 samples, 123 features), \texttt{phishing} (11,055 samples, 68 features), and \texttt{australian} (690 samples, 14 features). The monitored metrics are objective value $F(\bx^k)$ and constraint violation $|\|\bx^k\|_2^2 - 1|$. \Cref{Fig: a9a_results} shows convergence trajectories on the \texttt{a9a} dataset (averaged over five runs), while \Cref{table:libsvm_result_polyak,table:libsvm_result_Recursive} report final values as mean $\pm$ std.

We also evaluate on Problem~\eqref{eq:multiple_quadeq} with $M$ quadratic equality constraints; the results are reported in Appendix~\ref{subsec:multiple_quadratic_equalities}. Complete results for the main experiments on all datasets, together with further details of the experimental setup, are provided in Appendices~\ref{app:implementation_details} and~\ref{app:additional_experimental_results}.
\begin{figure}[t]
    \centering
    \includegraphics[width=\columnwidth]{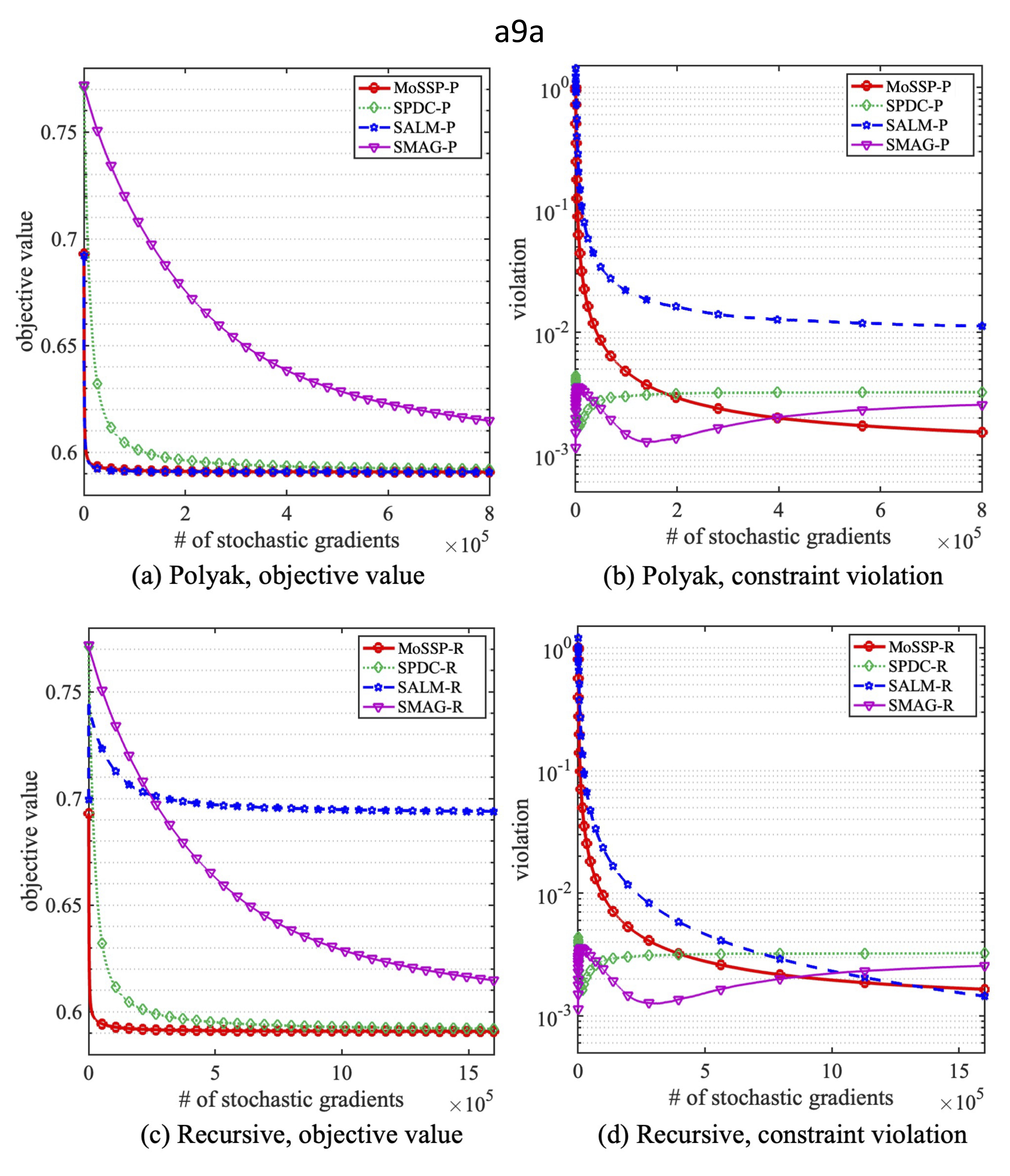}
    \caption{Comparison of MoSSP variants, SPDC, and SALM on \texttt{a9a}. (a), (b) Polyak momentum: objective value and constraint violation. (c), (d) Recursive momentum: objective value and constraint violation.}
    \label{Fig: a9a_results}
    \vskip -0.12in
\end{figure}
\begin{table*}[t]
\centering
\caption{Final performance comparison on three LIBSVM datasets. Results are reported over five independent runs. Bold font denotes the best result.}
\label{table:libsvm_results}
\begin{subtable}{\textwidth}
\centering
\caption{Mean \(\pm\) std of objective value (Obj. Value) and constraint violation (Const. Viol.) for MoSSP-P and two baseline methods with Polyak momentum.}
\label{table:libsvm_result_polyak}
\begin{small}
\begin{sc}
\resizebox{\textwidth}{!}{%
\begin{tabular}{lcccc}
\toprule
Dataset & Metric & MoSSP-P & SPDC-P & SALM-P \\
\midrule
\multirow{2}{*}{\texttt{a9a}}
 & Obj. Value & $\mathbf{0.5901 \pm 4.2\times 10^{-5}}$ & $0.5917 \pm 1.9\times 10^{-4}$ & $0.5917 \pm 1.17\times 10^{-4}$ \\
 & Const. Viol. & $\mathbf{1.53\times 10^{-3} \pm 1.24\times 10^{-5}}$ & $3.23\times 10^{-2} \pm 1.65\times 10^{-5}$ & $1.30\times 10^{-2} \pm 2.52\times 10^{-4}$ \\
\midrule
\multirow{2}{*}{\texttt{phishing}}
 & Obj. Value & $\mathbf{0.6045 \pm 1.6\times 10^{-5}}$ & $0.6051 \pm 4.8\times 10^{-5}$ & $0.6047 \pm 1.0\times 10^{-6}$ \\
 & Const. Viol. & $\mathbf{3.40\times 10^{-3} \pm 9.2\times 10^{-6}}$ & $3.66\times 10^{-2} \pm 5.73\times 10^{-5}$ & $1.21\times 10^{-2} \pm 1.52\times 10^{-4}$ \\
\midrule
\multirow{2}{*}{\texttt{australian}}
 & Obj. Value & $\mathbf{0.6239 \pm 3.6\times 10^{-5}}$ & $0.6265 \pm 6.74\times 10^{-4}$ & $0.6241 \pm 1.3\times 10^{-5}$ \\
 & Const. Viol. & $7.14\times 10^{-3} \pm 8.65\times 10^{-5}$ & $\mathbf{1.95\times 10^{-3} \pm 2.92\times 10^{-5}}$ & $2.19\times 10^{-2} \pm 5.11\times 10^{-4}$ \\
\bottomrule
\end{tabular}}%
\end{sc}
\end{small}
\end{subtable}

\vspace{0.6em}

\begin{subtable}{\textwidth}
\centering
\caption{Mean \(\pm\) std of objective value (Obj. Value) and constraint violation (Const. Viol.) for MoSSP-R and two baseline methods with recursive momentum.}
\label{table:libsvm_result_Recursive}
\begin{small}
\begin{sc}
\resizebox{\textwidth}{!}{%
\begin{tabular}{lcccc}
\toprule
Dataset & Metric & MoSSP-R & SPDC-R & SALM-R \\
\midrule
\multirow{2}{*}{\texttt{a9a}}
 & Obj. Value & $\mathbf{0.5809 \pm 3.8\times 10^{-4}}$ & $0.5921 \pm 2.14\times 10^{-4}$ & $0.6303 \pm 4.23\times 10^{-3}$ \\
 & Const. Viol. & $\mathbf{1.03\times 10^{-3} \pm 1.24\times 10^{-5}}$ & $3.23\times 10^{-2} \pm 7.35\times 10^{-6}$ & $1.30\times 10^{-2} \pm 1.17\times 10^{-4}$ \\
\midrule
\multirow{2}{*}{\texttt{phishing}}
 & Obj. Value & $\mathbf{0.6045 \pm 1.1\times 10^{-5}}$ & $0.6052 \pm 4.7\times 10^{-5}$ & $0.6054 \pm 9.7\times 10^{-5}$ \\
 & Const. Viol. & $\mathbf{4.63\times 10^{-3} \pm 7.12\times 10^{-6}}$ & $1.10\times 10^{-1} \pm 5.68\times 10^{-4}$ & $7.42\times 10^{-3} \pm 2.88\times 10^{-4}$ \\
\midrule
\multirow{2}{*}{\texttt{australian}}
 & Obj. Value & $\mathbf{0.6240 \pm 6.8\times 10^{-5}}$ & $0.6265 \pm 6.74\times 10^{-4}$ & $0.6243 \pm 1.5\times 10^{-5}$ \\
 & Const. Viol. & $6.88\times 10^{-3} \pm 1.03\times 10^{-4}$ & $\mathbf{1.94\times 10^{-3} \pm 2.13\times 10^{-5}}$ & $1.92\times 10^{-2} \pm 4.86\times 10^{-4}$ \\
\bottomrule
\end{tabular}}%
\end{sc}
\end{small}
\end{subtable}
\vskip -0.1in
\end{table*}
\subsection{Results Analysis}
\label{Sec: Results Analysis}
\Cref{Fig: a9a_results} shows that the MoSSP variants converge much faster than the baselines on the large-scale \texttt{a9a} dataset. Most notably, the MoSSP variants quickly reduce early-iterate constraint violations from \(10^{0}\) to \(10^{-3}\) with only a small number of gradient evaluations, demonstrating the efficiency of the single-loop framework. In contrast, SPDC and SALM are much slower at achieving feasibility, with violations remaining at \(10^{-2}\) or higher.
This speed advantage is consistently observed in both the Polyak momentum (\Cref{Fig: a9a_results}(a)--(b)) and recursive momentum (\Cref{Fig: a9a_results}(c)--(d)) variants, confirming that the gain comes from the single-loop design itself.

\Cref{table:libsvm_result_polyak,table:libsvm_result_Recursive} provide quantitative evidence. Both MoSSP-P and MoSSP-R achieve competitive objective values and final constraint violations across the tested datasets, showing that the proposed framework is robust to the choice of momentum technique. On the large-scale \texttt{a9a} and \texttt{phishing} datasets, the MoSSP variants rapidly reduce constraint violations and achieve competitive final feasibility. On the smaller \texttt{australian} dataset, SPDC attains a lower final constraint violation, as discussed in Appendix~\ref{app:additional_experimental_results}. Overall, the proposed single-loop momentum framework achieves fast 
objective decrease and competitive feasibility control for nonconvex 
constrained DC-regularized optimization.

\section{Conclusion and Discussion}
This paper proposes \textbf{MoSSP}, a simple single-loop stochastic penalty framework for general nonconvex constrained DC-regularized optimization. We develop two momentum-based variants, MoSSP-P and MoSSP-R, and establish their oracle complexity orders for finding two types of stochastic approximate solutions for Problem~\eqref{Prob: primal problem with DC in one block}. Our analysis shows that exploiting the DC structure does not worsen the complexity order, with or without the constraint qualification. To the best of our knowledge, this is the first systematic complexity study of single-loop stochastic methods for nonconvex, functionally constrained optimization involving DC structure.

Although we focus on equality constraints in Problem~\eqref{Prob: primal problem with DC in one block}, the framework extends directly to inequality constraints:
\begin{equation*}
\begin{aligned}
   \min_{\bx \in \RR^n} \quad & F(\bx) = f(\bx) + h(\bx) - g(\bx), \\
   \text{s.t.} \quad & \bc_{\calE}(\bx) = \zero, \ \ \bc_{\calI}(\bx) \le \zero,
\end{aligned}
\end{equation*}
where $\bc_{\calE}: \RR^n \to \RR^{m_1}$ and $\bc_{\calI}: \RR^n \to \RR^{m_2}$ are smooth mappings, and $f$, $h$, $g$ are as before. The quadratic penalty \eqref{Eq: the quadratic penalty problem} becomes
\begin{equation*}
Q_{\rho}(\bx) = f(\bx) + \tfrac{\rho}{2}\bigl(\|\bc_{\calE}(\bx)\|^2 + \|[\bc_{\calI}(\bx)]_+\|^2\bigr),
\end{equation*}
where $[\,\cdot\,]_+ \coloneqq \max\{\cdot,0\}$ is applied componentwise. Correspondingly, the gradient estimators in \eqref{Eq: enhanced gradient estimator Polyak} and \eqref{Eq: gradient estimator storm} are updated by replacing $\nabla \bc(\bx^k)\bc(\bx^k)$ with $\bm{J}(\bx^k) \coloneqq \nabla \bc_{\calE}(\bx^k)\bc_{\calE}(\bx^k) + \nabla \bc_{\calI}(\bx^k)[\bc_{\calI}(\bx^k)]_+$. This replacement preserves the smoothness of $Q_{\rho}$, and the error-bound condition in
\Cref{Ass: constraints conditions} can be replaced by
\[
\left\| \bm{J}(\bx^k)\right\|
\ge
\delta
\left(
\|\bc_{\calE}(\bx^k)\|^2
+
\|[\bc_{\calI}(\bx^k)]_+\|^2
\right)^{1/2}, \forall k\ge 0.
\]

A natural direction is to extend our analysis framework to general DC optimization under nonlinear constraints, where the concave component admits the form $-\EE_{\xi}\!\left[g(\bx;\xi)\right]$ and $\prox_{\mu g}(\cdot)$ is no longer available in closed form, as in problems such as PU learning and partial AUC maximization.
\section*{Acknowledgements and Disclosure of Funding}
Xiao Wang is supported by the National Natural Science Foundation of China (No. 12271278). Luxuan Li and Chunfeng Cui are supported by the National Natural Science Foundation of China (Nos.~12471282 and 12131004). The authors would like to thank Dr.\ Lei Yang for helpful discussions and insightful comments. The authors are grateful to the Area Chairs and the anonymous reviewers for their constructive comments.
\section*{Impact Statement}
This paper presents work whose goal is to advance the field of Machine
Learning. There are many potential societal consequences of our work, none of which we feel must be specifically highlighted here.

\bibliography{refers}
\bibliographystyle{icml2026}
\newpage
\appendix
\onecolumn

\section*{Contents of Appendices}
\parskip=0.5em

\noindent \textbf{Appendix A} \quad More Related Work \dotfill \pageref{appendix:related_work} \\
\noindent \hspace*{2em} \textbf{A.1} \quad Constrained DC and DC-Regularized Optimization \dotfill \pageref{app:related_work_constrained_dc} \\
\noindent \hspace*{2em} \textbf{A.2} \quad Stochastic Nonconvex Constrained Optimization \dotfill \pageref{app:related_work_stochastic_nonconvex_constrained} \\
\noindent \hspace*{2em} \textbf{A.3} \quad Unconstrained Stochastic DC Optimization \dotfill \pageref{app:related_work_stochastic_dc} \\

\noindent \textbf{Appendix B} \quad Preliminaries \dotfill \pageref{appendix:preliminaries} \\
\noindent \hspace*{2em} \textbf{B.1} \quad Properties of Moreau Envelope \dotfill \pageref{app:moreau_properties} \\
\noindent \hspace*{2em} \textbf{B.2} \quad Criticality in DC Optimization \dotfill \pageref{app:dc_criticality} \\
\noindent \hspace*{2em} \textbf{B.3} \quad Difference-of-Moreau-Envelopes (DME) Smoothing \dotfill \pageref{Subsection: Difference-of-Moreau-Envelopes} \\
\noindent \hspace*{2em} \textbf{B.4} \quad Solution Correspondence between Original Problem and its DME Surrogate \dotfill \pageref{App: Solution Correspondence Theory} \\
\noindent \hspace*{2em} \textbf{B.5} \quad Proof of \Cref{Prop: approximate correspondence} \dotfill \pageref{App: proof of prop2.2} \\
\noindent \hspace*{2em} \textbf{B.6} \quad Discussion on \Cref{Def: the definition of approximate point} \dotfill \pageref{Appendix: Discussion on Def 2.1} \\

\noindent \textbf{Appendix C} \quad Proofs of Complexity Analysis \dotfill \pageref{appendix:proofs of the main results} \\
\noindent \hspace*{2em} \textbf{C.1} \quad Proof Sketch \dotfill \pageref{sub:proof-sketch} \\
\noindent \hspace*{2em} \textbf{C.2} \quad Constructed Potential Function \dotfill \pageref{Sub: constructed potential function} \\
\noindent \hspace*{2em} \textbf{C.3} \quad Common Auxiliary Lemmas \dotfill \pageref{Subsection: auxiliary lemmas} \\
\noindent \hspace*{2em} \textbf{C.4} \quad Oracle Complexity Analysis of MoSSP-P \dotfill \pageref{Subsection: Lemmas for Oracle complexity analysis of Algo 1} \\
\noindent \hspace*{2em} \textbf{C.5} \quad Oracle Complexity Analysis of MoSSP-R \dotfill \pageref{Subsection: Lemmas for Oracle complexity analysis of Algo 2} \\

\noindent \textbf{Appendix D} \quad Experimental Results \dotfill \pageref{appendix:complete experiment results} \\
\noindent \hspace*{2em} \textbf{D.1} \quad Implementation Details \dotfill \pageref{app:implementation_details} \\
\noindent \hspace*{2em} \textbf{D.2} \quad Additional Experimental Results \dotfill \pageref{app:additional_experimental_results} \\
\noindent \hspace*{2em} \textbf{D.3} \quad Experimental Results on Multiple Quadratic Equality Constraints \dotfill \pageref{subsec:multiple_quadratic_equalities} \\

\hrulefill
\newpage

\section{More Related Work}
\label{appendix:related_work}
\subsection{Constrained DC and DC-Regularized Optimization}
\label{app:related_work_constrained_dc}

DC structures in learning problems commonly arise from two different sources. First, they may appear as an algebraic consequence of max-structured or minimax-type risk formulations, such as positive-unlabeled learning~\citep{kiryo2017positive}, partial AUC optimization~\citep{yao2022large}, and related adversarial formulations, where the objective can often be rewritten or approximated as a DC function. Second, DC structure is deliberately introduced through sparsity-promoting nonconvex regularizers, such as capped-\(\ell_1\), SCAD, MCP, or \(\ell_1-\ell_2\), which approximate \(\ell_0\)-type sparsity while retaining a tractable convex-concave decomposition; see, e.g., \citep[Table~1]{gong2013general}.

In the deterministic setting, several methods have been developed for constrained DC and DC-regularized optimization problems. \citet{zhou2024proximal} proposed a proximal ADMM for structured DC programs whose objective is the difference of two possibly nonsmooth convex functions and whose constraints are given by a linear mapping into a closed convex set. Under the Kurdyka--{\L}ojasiewicz property, they established convergence of the generated sequence to a critical point. \citet{lu2022penalty} developed penalty and augmented-Lagrangian methods for structured nonsmooth constrained DC programs with DC inequality constraints. Under a pointwise Slater condition, they showed that any feasible accumulation point generated by the penalty method is B-stationary for the original problem; the augmented-Lagrangian variant further yields KKT-type optimality conditions together with accumulation points of auxiliary multiplier sequences.

Several related deterministic schemes, including augmented-Lagrangian, sequential-convexification, and moving-balls-approximation (MBA) type methods, have also been studied for inequality-constrained DC or DC-regularized problems \citep{kanzow2025adaptive,liu2025convergence,yu2021convergence}. Specifically, \citet{kanzow2025adaptive} proposed a safeguarded augmented Lagrangian method for DC optimization with linear equality constraints and convex inequality constraints. Under a modified Slater constraint qualification, they established subsequential convergence to generalized KKT points. Nevertheless, their theoretical analysis relies on the convexity of the constraints, making its extension to nonconvex feasible sets such as Problem~\eqref{Prob: primal problem with DC in one block} challenging. \citet{yu2021convergence} studied a line-search variant of MBA-type algorithms for DC-regularized problems with smooth inequality constraints. Under the MFCQ, they established convergence to stationary points when the constructed potential function satisfies the KL property. More recently, \citet{liu2025convergence} developed an inexact MBA method for DC-regularized optimization coupled with differentiable inequality constraints whose gradients are locally Lipschitz continuous. Under the MSCQ, the partial bounded multiplier property, boundedness of the iterates, and the KL property of the constructed potential function, they achieved full sequence convergence to strong stationary points. They also established an iteration complexity of \(\calO(\varepsilon^{-2})\) for finding an \(\varepsilon\)-KKT point under the MFCQ and smoothness assumptions on the constraints. However, the MBA-type methods above are designed for inequality-constrained problems and rely on feasibility/interiority-type assumptions. Hence, they do not directly apply to equality-constrained feasible sets, such as Problem~\eqref{Prob: primal problem with DC in one block}.

\subsection{Stochastic Nonconvex Constrained Optimization}
\label{app:related_work_stochastic_nonconvex_constrained}
For stochastic optimization with nonconvex constraints, such as Problem \eqref{Prob: primal problem with DC in one block} with $g = 0$, stochastic penalty methods, including augmented Lagrangian (AL) approaches, have been extensively studied \citep{kushner1974penalty,wang2017penalty,xu2020primal}. Stochastic sequential quadratic programming (SQP) methods represent another prominent class of algorithms for addressing equality-constrained stochastic problems. Recent advances in stochastic SQP encompass convergence guarantees \citep{BeraCurtRobiZhou21, na2021inequality}, analysis in expectation \citep{NaAnitKola22}, almost-sure convergence \citep{curtis2023almost}, and worst-case complexity bounds \citep{BeraBollZhou22, NaMaho22,curtis2024worst}. {Notably, these methods differ in their computational structure:} penalty methods require updating the penalty parameter in the outer loop and (approximately) minimizing the penalized subproblem in the inner loop, while SQP methods compute search directions by solving a sequence of quadratic programming subproblems.

Alternatively, single-loop penalty methods have gained traction for such problems, and many studies have focused on different settings of Problem \eqref{Prob: primal problem with DC in one block}. For Problem \eqref{Prob: primal problem with DC in one block} with \(h - g = 0\), \citet{Alacaoglu2023ComplexityOS} proposed a linearized quadratic penalty approach incorporating recursive momentum \citep{cutkosky2019momentum}, achieving a sample complexity of $\tilde{\calO}(\varepsilon^{-4})$ via adaptive penalty parameters. To identify a stronger $\varepsilon$-stationary point with constraint violation at most $\varepsilon$, \citet{lu2024variancereducedfirstordermethodsdeterministically} introduced a linearized quadratic penalty method with truncated momentum. Under a regularity condition on constraint gradients, they achieved a nearly optimal complexity of $\tilde{\calO}(\varepsilon^{-3})$ using recursive momentum, whereas Polyak momentum yields $\tilde{\calO}(\varepsilon^{-4})$ with diminishing step sizes. {Recently, \citet{zuo2025adaptive} studied an adaptive single-loop stochastic penalty method for solving nonconvex optimization with equality and inequality constraints, proposing an adaptive penalty parameter update scheme and establishing an $\calO(\varepsilon^{-4})$ oracle complexity bound under a local CQ condition.} For Problem \eqref{Prob: primal problem with DC in one block} with $g = 0$, \citet{jin2022stochastic} proposed a single-loop stochastic primal-dual method based on the linearized AL function for nonconvex problems with numerous functional constraints. To find an $\varepsilon$-KKT point, their approach achieves an oracle complexity of $\calO(\varepsilon^{-6})$ starting from an arbitrary point, and $\calO(\varepsilon^{-5})$ with a feasible initial point. Subsequently, under certain constraint qualifications, \citet{shi2025momentum} developed a linearized AL method, establishing an oracle complexity of $\calO(\varepsilon^{-4})$ (or a nearly optimal $\calO(\varepsilon^{-3})$ with a nearly feasible initial point).

Within the exact penalty framework, \citet{yang2026singleloop} proposed a single-loop stochastic algorithm based on a hinge penalty method for Problem \eqref{Prob: primal problem with DC in one block} with weakly convex objective and constraint functions, achieving an $\calO(\varepsilon^{-6})$ oracle complexity for finding a near-$\varepsilon$-KKT point under a regularity condition. By combining the SPIDER-type variance reduction technique \citep{fang2018spider}, \citet{liu2025single} presented a single-loop stochastic subgradient method for nonconvex nonsmooth stochastic optimization with weak inequality constraints in expectation. Under a (uniform) Slater-type constraint qualification, the method achieves an $\calO(\varepsilon^{-4})$ sample complexity for evaluations of both the objective and constraint function subgradients, and $\calO(\varepsilon^{-6})$ for evaluations of the constraint function values to produce an $(\varepsilon, \varepsilon)$-KKT point. {For problems with deterministic equality constraints, \citet{cui2025exact} proposed a method that achieves an $\calO(\varepsilon^{-3})$ sample complexity to find an $\varepsilon$-KKT point.}

\subsection{Unconstrained Stochastic DC Optimization}
\label{app:related_work_stochastic_dc}
For the unconstrained case of Problem \eqref{Prob: primal problem with DC in one block}, while deterministic DC algorithms have been extensively studied \citep{le2018dc}, their stochastic counterparts (SDCA) with non-asymptotic convergence analysis remain relatively less explored, appearing only in a few works such as \citep{nitanda2017stochastic,xu2019stochastic}. However, these analyses require smoothness \citep{nitanda2017stochastic} or H\"{o}lder continuity assumptions on the gradients of DC components \citep{xu2019stochastic}, which may be restrictive for many nonsmooth functions, such as the \(\ell_{1-2}\) norm. The Moreau envelope provides a powerful smoothing technique for handling nonsmoothness via the proximal operator and has been widely used in stochastic weakly convex optimization \citep{davis2019stochastic}. Subsequently, \citet{sun2023algorithms} applied Moreau envelope smoothing to both DC components, enabling non-asymptotic convergence analysis for unconstrained DC problems under a relaxed criticality criterion. \citet{hu2024single} further proposed a single-loop stochastic method for nonsmooth difference-of-weakly-convex problems with an $\calO(\varepsilon^{-4})$ oracle complexity; however, their setting is unconstrained and their DWC analysis requires uniformly bounded second moments of stochastic subgradients for both components (see Assump. 4.6(iii) in \citet{hu2024single}). More recently, \citet{chayti2025stochastic} introduced momentum-based variance reduction for stochastic DC optimization, but their single-loop analysis assumes a smooth concave component and smooths only the convex component, which differs from our setting where the concave component in DC structure can be nonsmooth and the problem is further subject to nonconvex constraints.

\section{Preliminaries}
\label{appendix:preliminaries}
For the convergence analysis, we first introduce several definitions and preliminary results for nonsmooth DC optimization and Moreau-envelope smoothing. We use the following definitions of general subgradient and subdifferential \cite{davis2019stochastic,rockafellar2009variational}.
\begin{definition}
    \label{subgradient and subdifferential}
    Consider a function $f: \RR^n \to \RR \cup \{+\infty\}$ and a point $\bx \in \RR^n$, with $f(\bx)$ being finite. A vector $\bv \in \RR^n$ is a general subgradient of \(f\) at \(\bx\) if 
\begin{align}
f(\by) \geq f(\bx) + \langle \bv, \by - \bx \rangle + o(\|\by - \bx\|) \text{ as } \by \to \bx. \notag
\end{align}
\end{definition}
The subdifferential \(\partial f(\bx)\) is the set of subgradients of $f$ at $\bx$. For a continuously differentiable function $f$, $\partial f(\bx) = \{\nabla f(\bx)\}$; for convex functions, this coincides with the convex subdifferential. For notational simplicity, we abuse the notation $\partial f(\bx)$ to denote an arbitrary subgradient from the corresponding subdifferential when the context is clear.

A mapping $\calM : \mathcal{D} \to \RR^l$ is said to be $C$-Lipschitz continuous if $\|\calM(\bx) - \calM(\bx')\| \leq C\|\bx - \bx'\|$ for all $\bx, \bx' \in \mathcal{D}$. A function $f: \RR^n\rightarrow \RR$ is said to be $L_f$-smooth if it is continuously differentiable and has an $L_f$-Lipschitz continuous gradient, that is $\| \nabla f(\bx) - \nabla f(\bx') \| \leq L_f \| \bx - \bx'\|$ for all $\bx,\, \bx' \in \RR^n$. It then satisfies the following inequality:
\begin{align}
    f(\bx) \leq f(\bx') + \langle \nabla f(\bx'), \bx - \bx' \rangle + \frac{L_f}{2} \| \bx - \bx'\|^2,  \notag
    \quad \forall \bx,\bx' \in \RR^n.
\end{align}
A function $\phi: \RR^n \rightarrow \RR \cup \{ \infty \}$ is $m_\phi$-weakly convex if $\phi(\bx) + \frac{m_\phi}{2} \|\bx\|^2$ is convex.

\subsection{Properties of \textit{Moreau Envelope}}
\label{app:moreau_properties}
Regarding the definition of the \textit{Moreau envelope} of \(\phi\) in \eqref{Eq: the proximal operator}, it is straightforward to verify that
\begin{align*}
\calM_{\mu\phi}(\bz)= \phi(\prox_{\mu \phi}(\bz)) + \frac{1}{2\mu} \|\prox_{\mu \phi}(\bz)- \bz\|^2.
\end{align*}
For any \(\mu \in(0, \frac{1}{m_{\phi}})\), \(\prox_{\mu \phi}(\bz)\) is unique due to the strong convexity of the minimization in \eqref{Eq: the proximal operator} and is $\frac{1}{1 - \mu m_{\phi}}$-Lipschitz continuous. It is well known that if $\phi$ is convex, then its Moreau envelope $\calM_{\mu\phi}$ is also convex; see, e.g., \citep[Theorem~2.26]{rockafellar2009variational}. Moreover, $\calM_{\mu \phi}$ maintains the global minimization structure of $\phi$, i.e.,
\begin{align}
   \min_{\bx} \phi(\bx) \leq \phi(\prox_{\mu \phi}(\bz)) \leq \calM_{\mu \phi}(\bz) \leq \phi(\bz) \quad \text{for all } \bz.
   \label{Eq: global bound}
\end{align}
\subsection{Criticality in DC Optimization}
\label{app:dc_criticality}
Consider the unconstrained DC optimization problem:
\begin{equation}
    \min_{\bx \in \RR^n} \Psi (\bx) := \phi(\bx) - g(\bx).
    \label{Eq: unconstrained DC problem in appendix}
\end{equation}
A point $\bar{\bx}\in \RR^n$ is called a {\it critical point} \cite{Oliveira2019ProximalBM} of Problem \eqref{Eq: unconstrained DC problem in appendix} if
\begin{align}
    \label{Def: Critical point in DC}
    \zero \in \partial \phi (\bar{\bx}) - \partial g(\bar{\bx}), \text{ or equivalently, }\partial \phi(\bar{\bx}) \cap \partial g(\bar{\bx}) \neq \emptyset.
\end{align}
In particular, when $g$ is continuously differentiable or constant, \eqref{Def: Critical point in DC} simplifies to the general stationarity condition $\zero \in \partial \Psi(\bar{\bx})$ \cite{rockafellar2009variational}. Since it is generally hard to find an exact critical point in a finite number of iterations, iterative algorithms normally pursue an \textit{$\varepsilon$-critical point} $\bar{\bx}$ for a given \(\varepsilon > 0\), satisfying
\begin{align}
    \mathrm{dist}\left(\partial g(\bar{\bx}), \partial \phi (\bar{\bx}) \right) \leq \varepsilon, \notag
\end{align}
or equivalently, there exist $\bar{\bu}_{1} \in \partial \phi (\bar{\bx})$ and $\bar{\bu}_{2} \in \partial g(\bar{\bx})$ such that \(\|\bar{\bu}_{1} - \bar{\bu}_{2}\| \leq \varepsilon\). When $g$ is non-differentiable, however, finding such a point remains challenging \cite{xu2019stochastic}. Instead, \citet{xu2019stochastic} focuses on seeking a \textit{nearly $\varepsilon$-stationary point \( \bar{\bx}\)}; that is, there exists $\bx^*$ such that \( \| \bar{\bx} - \bx^*\| \leq \calO(\varepsilon)\) and \(\mathrm{dist}(\partial g(\bx^*), \partial \phi (\bx^*) ) \leq \varepsilon\).

In this paper, we introduce the notion of \textit{$\varepsilon$-critical points} \cite{moudafi2021complete,sun2023algorithms}, defined as follows.
\begin{definition}[$\varepsilon$-critical point]
    \label{Def: epsilon-critical point in DC} 
    Given $\varepsilon > 0$, a point $\bar{\bx}\in \RR^{n}$ is called an $\varepsilon$-critical point of Problem \eqref{Eq: unconstrained DC problem in appendix} if there exist $\bar{\bu}\in \partial \phi (\bar{\bx}) - \partial g(\bar{\by})$ and \(\bar{\by} \in \RR^n\) such that
    \begin{align}
    \label{Eq: the criticality in def-varepsilon-critical}
        \max \{\|\bar{\bu}\|, \|\bar{\bx}- \bar{\by}\|\} \leq \varepsilon.
    \end{align}
\end{definition}

A weaker notion of \Cref{Def: epsilon-critical point in DC} is used in \cite{hu2024single,yao2022large} when proximal mappings of $\phi$ and $g$ can only be solved inexactly.
\begin{definition}[Nearly $\varepsilon$-critical point \cite{yao2022large}]
    \label{Def: epsilon-critical point 3} 
    Given $\varepsilon > 0$, a point $\bar{\bx}\in \RR^{n}$ is called a nearly $\varepsilon$-critical point of Problem \eqref{Eq: unconstrained DC problem in appendix} if there exist $\bar{\bu}\in \partial \phi (\bar{\by}_1) - \partial g(\bar{\by}_2)$ and $\bar{\by}_1, \bar{\by}_2 \in \RR^n$ such that
    \begin{equation}
        \label{Eq: epsilon-critical point 3}
        \max \{ \|\bar{\bu}\|, \|\bar{\bx}- \bar{\by}_1\|, \|\bar{\bx}- \bar{\by}_2\| \} \leq \varepsilon.
    \end{equation}
\end{definition}
Notably, when $\bar{\by}_1 = \bar{\bx}$, Definition~\ref{Def: epsilon-critical point 3} reduces to Definition~\ref{Def: epsilon-critical point in DC}, and when $\varepsilon = 0$, this definition recovers \Cref{Def: Critical point in DC}.

\subsection{\textit{Difference-of-Moreau-Envelopes} (DME) Smoothing}
\label{Subsection: Difference-of-Moreau-Envelopes}
The DME smoothing approximation of \(\Psi\) in Problem \eqref{Eq: unconstrained DC problem in appendix} for any $\mu \in (0, 1/m_{\phi})$ is defined as
\begin{align}
\label{Eq: the DME smoothing function in appendix}
 \Psi_{\mu}(\bz) := \calM_{\mu \phi}(\bz) - \calM_{\mu g}(\bz).
\end{align}
The smoothness properties and approximation bounds of $\Psi_{\mu}$ are summarized below \cite{hiriart1991regularize}.
\begin{proposition}
\label{Prop: smoothness of DME}
Consider Problem \eqref{Eq: unconstrained DC problem in appendix} and \eqref{Eq: the DME smoothing function in appendix}. For any \(0 < \mu < \frac{1}{m_\phi}\), the following hold.
 \begin{enumerate}
 \item[(i)] \(\Psi_{\mu}\) is continuously differentiable with gradient \( \nabla \Psi_{\mu}(\bz) = \mu^{-1} (\prox_{\mu g}(\bz) - \prox_{\mu \phi}(\bz))\).
 
 \item[(ii)] \(\nabla \Psi_{\mu}\) is $(\frac{2- \mu m_{\phi}}{\mu - \mu^2 m_{\phi}})$-Lipschitz continuous.

 \item[(iii)] For any $\bz \in \RR^n $, it holds that
 \[
\Psi(\prox_{\mu \phi}(\bz)) \leq \Psi_{\mu}(\bz) \leq \Psi(\prox_{\mu g}(\bz)).\]
 \end{enumerate}
 \end{proposition}

\subsection{{Solution Correspondence between Original Problem and its DME Surrogate}}
\label{App: Solution Correspondence Theory}
For completeness, we recall the theoretical relationship between the solutions of Problem~\eqref{Eq: unconstrained DC problem in appendix} and its DME surrogate~\eqref{Eq: the DME smoothing function in appendix}, as established in~\citet{hiriart1991regularize} and~\citet{sun2023algorithms}.

Let $\calX_{\Psi} := \argmin_{\bx} \Psi(\bx)$ and $\hat{\Psi} := \min_{\bx} \Psi(\bx)$ denote its solution set and optimal value, respectively. For the DME surrogate, we similarly define its optimal value as $\hat{\Psi}_{\mu} := \inf_{\bz} \Psi_{\mu}(\bz)$ and its solution set as $\calX_{\Psi_{\mu}} := \argmin_{\bz} \Psi_{\mu}(\bz)$.

\begin{proposition}[Stationary Point and Global Minimizer Correspondence \cite{hiriart1991regularize}]
\label{Prop: correspondence of stationary point and global minimizer}
Consider Problem~\eqref{Eq: unconstrained DC problem in appendix} and its DME surrogate~\eqref{Eq: the DME smoothing function in appendix}. Suppose that $\calX_{\Psi} \neq\emptyset$. For any $\mu \in (0, 1/m_{\phi})$, the following statements hold:

\textbf{(i) Stationary point correspondence.} If $\bar{\bz}$ is a stationary point of $\Psi_{\mu}$ (i.e., $\nabla \Psi_{\mu}(\bar{\bz}) = \zero$), then $\bar{\bx} := \prox_{\mu\phi}(\bar{\bz})$ is a critical point of $\Psi$, satisfying
\begin{equation}
\label{Eq: stationary identity}
\prox_{\mu\phi}(\bar{\bz}) = \prox_{\mu g}(\bar{\bz}), \quad \Psi(\bar{\bx}) = \Psi_{\mu}(\bar{\bz}).
\end{equation}
Conversely, if $\bar{\bx}$ is a critical point of $\Psi$ and $\bar{\bxi} \in \partial\phi(\bar{\bx}) \cap \partial g(\bar{\bx})$, then $\bar{\bz} := \bar{\bx} + \mu \bar{\bxi}$ is a stationary point of $\Psi_{\mu}$ with $\prox_{\mu \phi}(\bar{\bz}) = \prox_{\mu g}(\bar{\bz}) = \bar{\bx}$ and $\Psi(\bar{\bx}) = \Psi_{\mu}(\bar{\bz})$.

\textbf{(ii) Global minimizer correspondence.} The set $\calX_{\Psi_{\mu}}$ is nonempty and $\hat{\Psi} = \hat{\Psi}_{\mu}$. Furthermore, if $\bar{\bz} \in \calX_{\Psi_{\mu}}$, then $\prox_{\mu \phi}(\bar{\bz}) \in \calX_{\Psi}$. Conversely, if $\bar{\bx} \in \calX_\Psi$ and $\bar{\bxi} \in \partial\phi(\bar{\bx}) \cap \partial g(\bar{\bx})$, then $\bar{\bz} := \bar{\bx} + \mu \bar{\bxi} \in \calX_{\Psi_{\mu}}$.
\end{proposition}

\subsection{Proof of \Cref{Prop: approximate correspondence}}
\label{App: proof of prop2.2}
\begin{proof}
Let $\bar{\bz}$ be a point satisfying
$\|\nabla \Psi_{\mu}(\bar{\bz})\|\le \min\{1,\mu^{-1}\}\varepsilon$ and set
\(
  \bar{\bx} := \prox_{\mu \phi}(\bar{\bz}),\,
  \bar{\by} := \prox_{\mu g}(\bar{\bz}).
\)
Using \eqref{Eq: the gradient of Moreau envelope}, we have
\(
  \bar{\bxi}_{\phi} := \mu^{-1}(\bar{\bz}-\bar{\bx}) \in \partial \phi(\bar{\bx})
\) and \(
  \bar{\bxi}_{g} := \mu^{-1}(\bar{\bz}-\bar{\by}) \in \partial g(\bar{\by}).
\)
Combining these with \eqref{Eq: the gradient of Moreau envelope} yields
\(
  \bar{\bxi}_{\phi} - \bar{\bxi}_{g}
    = \mu^{-1}(\bar{\by}-\bar{\bx})
    = \nabla \Psi_{\mu}(\bar{\bz}).
\)
Setting $\bar{\bu}:=\bar{\bxi}_{\phi}-\bar{\bxi}_{g}\in\partial\phi(\bar{\bx})-\partial g(\bar{\by})$, it follows that
\begin{align*}
  \max\bigl\{\| \bar{\bu} \|,
             \|\bar{\bx}-\bar{\by}\|\bigr\}  = \max\bigl\{\|\nabla \Psi_{\mu}(\bar{\bz})\|,
                  \mu\|\nabla \Psi_{\mu}(\bar{\bz})\|\bigr\} = \max\{1,\mu\}\,\|\nabla \Psi_{\mu}(\bar{\bz})\| \leq \varepsilon,
\end{align*}
which implies that $\bar{\bx}$ is an $\varepsilon$-critical point of \(\Psi\) in the sense of Definition~\ref{Def: epsilon-critical point in DC}.
\end{proof}

\subsection{Discussion on \Cref{Def: the definition of approximate point}}
\label{Appendix: Discussion on Def 2.1}
There are connections between $\varepsilon$-KKT points and $\varepsilon$-stationary points, but the two notions are not equivalent in general.
For a given point $\bx$, if the constraint violation $\|\bc(\bx)\|$ is small, then the infeasible-stationarity measure
$\|\nabla \bc(\bx)\bc(\bx)\|$ is also small. However, the reverse implication does not generally hold. In practical computations, an iterate may get trapped at a stationary point of the constraint-violation minimization problem $\min_{\bx\in\RR^{n}} \tfrac12\|\bc(\bx)\|^{2}$, satisfying $\nabla \bc(\bx^*)\bc(\bx^*)=0$ while $\bc(\bx^*)\neq 0$, which is referred to as an
infeasible stationary point. Therefore, to guarantee that small infeasible stationarity implies small constraint violation and hence approximate feasibility,
a constraint qualification condition such as \Cref{Ass: constraints conditions} is required. Under such a condition, an $\varepsilon$-stationary point further implies an $\calO(\varepsilon)$-KKT point in expectation.

\section{Proofs of Complexity Analysis}
\label{appendix:proofs of the main results}
\subsection{Proof Sketch}
\label{sub:proof-sketch}

We provide a proof sketch of \Cref{Theorem: the convergence rate of polyak with initial feasibility,Theorem: the convergence rate of storm with initial feasibility}. We outline how the DME-based residual generated by MoSSP-P and MoSSP-R yields the stochastic approximate solutions in \Cref{Def: the definition of approximate point}. The proof proceeds in four steps.

\subsubsection*{Step 1: From the algorithmic residual to the stochastic KKT inclusion.}
Let us first set $\by^k := \prox_{\mu_k g}(\bz^k)$. The purpose of this step is to identify the random variables required in Definition~\ref{Def: the definition of approximate point}. For both MoSSP-P and MoSSP-R, the optimality condition of the \(\bx\)-update implies, as shown in Appendix~\ref{Subsection: auxiliary lemmas}, that
\[
\bu^{k+1} \in \nabla Q_{\rho_k}(\bx^{k+1}) +\partial h(\bx^{k+1}) -\partial g(\by^k).
\]
Equivalently, by setting $\blambda^{k+1}:=\rho_k\bc(\bx^{k+1})$, we obtain
\[
\bu^{k+1} \in \nabla f(\bx^{k+1}) +\nabla \bc(\bx^{k+1})\blambda^{k+1} +\partial h(\bx^{k+1}) -\partial g(\by^k).
\]
Therefore, for the randomly chosen index \(R\), the tuple
\[
    \bar{\bx}:=\bx^{R+1},\qquad \bar{\by}:=\by^R,\qquad \bar{\blambda}:=\rho_R\bc(\bx^{R+1}),\qquad \bar{\bu}:=\bu^{R+1},
\]
satisfies the KKT inclusion \eqref{Def: the definition of u} almost surely. Hence, it remains to bound
\[
    \|\bu^{R+1}\|,\qquad \|\bx^{R+1}-\by^R\|,\qquad \|\nabla\bc(\bx^{R+1})\bc(\bx^{R+1})\|,
\]
and, under \Cref{Ass: constraints conditions}, the feasibility violation \(\|\bc(\bx^{R+1})\|\).

\subsubsection*{Step 2: Establishing descent of the DC-structured potential function.}
The concave term \(-g\) prevents the standard penalty objective from admitting a sufficient descent property. We therefore introduce the potential function
\[
\calL_{\rho,\mu}(\bw) = Q_\rho(\bx)+h(\bx)+\frac{1}{2\mu}\|\bx-\bz\|^2-\calM_{\mu g}(\bz), \qquad \bw=(\bx,\bz),
\]
which encodes the two proximal updates of the proposed algorithms. Under the mild parameter condition \(\mu_k L_{\rho_k}\le \tfrac{1}{4}\), \Cref{Lemma: approximate-sufficient-descent} yields
\[
\calL_{\rho_{k+1},\mu_{k+1}}(\bw^{k+1}) \le \calL_{\rho_k,\mu_k}(\bw^k) -\Omega(\mu_k^{-1})\|\bw^{k+1}-\bw^k\|^2 +\Delta_{k+1} +\frac{\rho_{k+1}-\rho_k}{2}C^2 +\mu_k\|\be^k\|^2 .
\]
Here \(\be^k\) is the stochastic gradient estimation error, and \(\Delta_{k+1}\) accounts for the change of the smoothing parameter \(\mu_k\). Thus, after summing over \(k\), the iterate variation
\[
    \sum_{k=0}^{K-1} \mu_k^{-1}\|\bw^{k+1}-\bw^k\|^2
\]
is controlled by the initial potential difference, the parameter variation terms, and the cumulative stochastic error. In the constant-parameter setting used in the main theorems, the variation terms \(\Delta_{k+1}\) and \(\rho_{k+1}-\rho_k\) vanish, leaving only the initial potential difference and the stochastic error. This iterate variation bound is then used to control the criticality measure in the next steps.

\subsubsection*{Step 3: From iterate variation to residual bounds.}
The one-step residual estimate in \Cref{Lemma: criticality-stationarity-measure} shows that there exists a constant $\kappa_1>0$ independent of $K$ such that
\[
\max\{\|\bu^{k+1}\|^2,\|\bx^{k+1}-\by^k\|^2\} \leq \kappa_1\left(\|\be^k\|^2 + \bigl(L_{\rho_k}^2+\mu_k^{-2}\bigr) \|\bw^{k+1}-\bw^k\|^2\right).
\]
Thus the descent estimate from Step 2 reduces the averaged DME-induced criticality residual
$\max\{\|\bu^{k+1}\|^2,\|\bx^{k+1}-\by^k\|^2\}$
to the iterate-variation term and the cumulative stochastic error.

The infeasible stationarity measure is then controlled via \Cref{Lemma: the error bound of infeasible stationarity and kkt conditions}, which gives, for some constant $\kappa_2>0$ independent of $K$, with \(\underline{\rho}_K:=\min_{0\leq k\leq K-1}\rho_k\),
\[
\frac{1}{K}\sum_{k=0}^{K-1} \|\nabla \bc(\bx^{k+1})\bc(\bx^{k+1})\|^2 \leq \frac{\kappa_2}{\underline{\rho}_K^2} \left( \frac{1}{K}\sum_{k=0}^{K-1}\|\bu^{k+1}\|^2 +1 \right).
\]
This estimate makes explicit the role of the penalty parameter: a larger \(\rho_k\) improves feasibility control through the factor \(\underline{\rho}_K^{-2}\), but it simultaneously increases the smoothness constant \(L_{\rho_k}\) in the criticality bound above. Thus, the penalty, smoothing, and momentum parameters must be carefully balanced. Under \Cref{Ass: constraints conditions}, the infeasible stationarity bound further implies feasibility.

\subsubsection*{Step 4: Stochastic error control and parameter selection.}
The two algorithmic variants differ mainly in how the cumulative stochastic error identified in Step 3 is controlled.

(i) For MoSSP-P, the Polyak error recursion contains an \(\alpha_k^{-1}\)-weighted variation term ($\alpha_k^{-1}L_f^2\| \bw^{k+1} - \bw^k\|^2$ in \eqref{Eq: the recursive gradient relationship in Polyak momentum scheme}), which necessitates the parameter coupling \(\alpha=\Theta(\mu)\). Consequently, the dominant criticality bound takes the form
\[
    \mathcal R_P := \max\Bigl\{ \EE[\|\bu^{R+1}\|^2], \EE[\|\bx^{R+1}-\prox_{\mu g}(\bz^R)\|^2] \Bigr\} \lesssim \frac{1}{\mu K}+\mu .
\]
Balancing the two terms gives \(\mu=\Theta(K^{-1/2}),\ \alpha=\Theta(K^{-1/2})\), and hence \(\mathcal R_P=\calO(K^{-1/2})\). With \(\rho=\Theta(K^{1/4}),\ \|\bc(\bx^0)\|^2=\calO(K^{-1/4})\), the feasibility residual is also of order \(\calO(K^{-1/2})\) under \Cref{Ass: constraints conditions}. Thus \(K=\calO(\varepsilon^{-4})\) suffices, yielding the \(\calO(\varepsilon^{-4})\) oracle complexity of MoSSP-P.

(ii) For MoSSP-R, the recursive momentum estimator does not introduce the \(\alpha_k^{-1}\)-weighted variation term (see \eqref{Eq: the stochastic error of the variance estimator}). This allows for the weaker coupling \(\alpha=\Theta(\mu^2)\), and the dominant criticality bound becomes
\[
    \mathcal R_R := \max\Bigl\{ \EE[\|\bu^{R+1}\|^2], \EE[\|\bx^{R+1}-\prox_{\mu g}(\bz^R)\|^2] \Bigr\} \lesssim \frac{1}{\mu K}+\mu^2 .
\]
Balancing the two terms gives \(\mu=\Theta(K^{-1/3}),\ \alpha=\Theta(K^{-2/3})\). Together with the initial batch size \(b_0=\Theta(K^{1/3})\), which controls the initial estimation error, we obtain \(\mathcal R_R=\calO(K^{-2/3})\). With \(\rho=\Theta(K^{1/3}) \mbox{ and } \|\bc(\bx^0)\|^2=\calO(K^{-1/3})\), the feasibility residual is also of order \(\calO(K^{-2/3})\) under \Cref{Ass: constraints conditions}. Hence the setting \(K=\calO(\varepsilon^{-3})\) ensures an $\varepsilon$-approximate solution. Since \(b_0=\calO(K^{1/3})\) is dominated by the total number of per-iteration oracle calls, MoSSP-R achieves \(\calO(\varepsilon^{-3})\) oracle complexity.
\subsection{Constructed Potential Function}
\label{Sub: constructed potential function}
We define the potential function
\begin{equation}
    \label{Eq: the potential function}
    \calL_{\rho, \mu}(\bw) := Q_{\rho}(\bx) + h(\bx) + \frac{1}{2\mu}\|\bx - \bz\|^{2} - \calM_{\mu g}(\bz),
\end{equation}
where $\bw = (\bx,\bz)$. Recall that $F^*$ is the finite infimum introduced in Problem~\eqref{Prob: primal problem with DC in one block}. Since $Q_\rho(\bx) = f(\bx) + \frac{\rho}{2}\|\bc(\bx)\|^2 \ge f(\bx)$, we obtain the uniform lower bound:
\begin{equation}
    \label{Eq: the lower bound of penalty function}
    \inf_{\bx \in \RR^n} \left\{ Q_{\rho}(\bx) + h(\bx) - g(\bx) \right\}
    \ge \inf_{\bx \in \RR^n} F(\bx) = F^* > -\infty, \quad \forall\, \rho > 0.
\end{equation}

Unlike the convex case (where $g=0$), the concave component $-g$ prevents the potential sequence $\{\calL_{\rho_k, \mu_k}(\bw^k)\}_{k \in \NN}$ from being trivially lower-bounded. However, since $g$ is convex and $\sup_{\bv_g\in\partial g(\bx)}\|\bv_g\|\le G$ for all $\bx$ (\Cref{Ass: boundedness of c and f}), $g$ is $G$-Lipschitz continuous, and we can bound the potential function for any $\mu > 0$ and $\rho > 0$ as follows:
\begin{align}
    \calL_{\rho, \mu}(\bw)
    &\geq Q_{\rho}(\bx) + h(\bx) + \frac{1}{2\mu}\|\bx - \bz\|^2 - g(\bz) \notag \\
    &\geq Q_{\rho}(\bx) + h(\bx) - g(\bx) + \frac{1}{2\mu}\|\bx - \bz\|^2 - G\|\bx - \bz\| \notag \\
    &\geq Q_{\rho}(\bx) + h(\bx) - g(\bx) - \frac{G^2\mu}{2} \notag \\
    &\geq F^* - \frac{G^2\mu}{2}, \label{Eq: the lower bound of potential function}
\end{align}
where the first inequality follows from $\calM_{\mu g}(\bz) \leq g(\bz)$ in \eqref{Eq: global bound}, the second utilizes the Lipschitz continuity of $g$, and the third applies Young's inequality. Under \eqref{Eq: basic parameter conditions for both algorithms}, the sequence $\{\mu_k\}$ is non-increasing, hence there exists a constant $\bar{\mu}$ independent of $K$ satisfying $\mu_k \leq \bar{\mu}$ for all $k \in \NN$. Then \eqref{Eq: the lower bound of potential function} yields the uniform lower bound
\begin{equation}
    \label{Eq: uniform lower bound L star}
    \calL_{\rho_k, \mu_k}(\bw^k) \geq F^* - \frac{G^2 \bar{\mu}}{2} =: \calL^*, \quad \forall\, k \in \NN,
\end{equation}
which we use throughout the proofs.
\subsection{{Common Auxiliary Lemmas}}
\label{Subsection: auxiliary lemmas}
For notational convenience, define the stochastic gradient error
\begin{equation}
\label{Eq: the stochastic error}
\be^k :=
\begin{cases}
    \bS^k-\nabla Q_{\rho_k}(\bx^k), & \text{for MoSSP-P},\\
    \bD^k-\nabla Q_{\rho_k}(\bx^k), & \text{for MoSSP-R}.
\end{cases}
\end{equation}
Also set
\[
    \by^k:=\prox_{\mu_k g}(\bz^k).
\]
We first verify that the residual \(\bu^{k+1}\) defined in
\eqref{Eq: u-define-stochastic-polyak} or
\eqref{Eq: u-define-stochastic-storm} satisfies the residual inclusion
\begin{align}
    \bu^{k+1} \in \partial \psi_{\rho_k}(\bx^{k+1}) - \partial g(\prox_{\mu_k g}(\bz^k)).
    \label{Eq: the definition of u for analysis}
\end{align}
Indeed, both \eqref{eq:bx_update_Polyak} and
\eqref{eq:bx_update_storm} can be written in the unified form
\[
    \bx^{k+1}
    =
    \prox_{\mu_k h}
    \bigl(\bz^k-\mu_k \tilde{\nabla} Q_{\rho_k}(\bx^k) \bigr).
\]
The optimality condition of this proximal update gives
\[
    \zero
    \in
    \tilde{\nabla} Q_{\rho_k}(\bx^k)
    +\partial h(\bx^{k+1})
    +\mu_k^{-1}(\bx^{k+1}-\bz^k).
\]
Hence, there exists
\[
    \bv_h^{k+1}
    :=
    \mu_k^{-1}(\bz^k-\bx^{k+1})
    -
    \tilde{\nabla} Q_{\rho_k}(\bx^k)
    \in
    \partial h(\bx^{k+1}).
\]
Moreover, by the definition of \(\by^k=\prox_{\mu_k g}(\bz^k)\), we have
\[
    \bv_g^k
    :=
    \mu_k^{-1}(\bz^k-\by^k)
    \in
    \partial g(\by^k).
\]
Using the unified definition
\[
    \bu^{k+1}
    :=
    \nabla Q_{\rho_k}(\bx^{k+1})
    -
    \tilde{\nabla} Q_{\rho_k}(\bx^k)
    +
    \mu_k^{-1}(\by^k-\bx^{k+1}),
\]
which coincides with \eqref{Eq: u-define-stochastic-polyak} for MoSSP-P
and with \eqref{Eq: u-define-stochastic-storm} for MoSSP-R, we obtain
\[
\begin{aligned}
    \bu^{k+1}
    &=
    \nabla Q_{\rho_k}(\bx^{k+1})
    +
    \bv_h^{k+1}
    -
    \bv_g^k   \in
    \nabla Q_{\rho_k}(\bx^{k+1})
    +
    \partial h(\bx^{k+1})
    -
    \partial g(\by^k).
\end{aligned}
\]
Since \(Q_{\rho_k}\) is smooth and
\(\psi_{\rho_k}=Q_{\rho_k}+h\), we have
\[
    \nabla Q_{\rho_k}(\bx^{k+1})
    +
    \partial h(\bx^{k+1})
    =
    \partial\psi_{\rho_k}(\bx^{k+1}).
\]
This proves \eqref{Eq: the definition of u for analysis}. Equivalently,
because
\[
    \nabla Q_{\rho_k}(\bx^{k+1})
    =
    \nabla f(\bx^{k+1})
    +
    \rho_k\nabla\bc(\bx^{k+1})\bc(\bx^{k+1}),
\]
the choice
\[
    \blambda^{k+1}:=\rho_k\bc(\bx^{k+1})
\]
yields
\[
\bu^{k+1}
\in
\nabla f(\bx^{k+1})
+
\nabla\bc(\bx^{k+1})\blambda^{k+1}
+
\partial h(\bx^{k+1})
-
\partial g(\by^k).
\]
Thus the residual \(\bu^{k+1}\), together with
\(\by^k\) and \(\blambda^{k+1}\), provides the KKT-type inclusion used
throughout the proof. \eproof

Under Assumption \ref{Ass: boundedness of c and f}, we can guarantee the smoothness of \(Q_{\rho}(\bx)\), which is essential for our analysis.
\begin{lemma}
    \label{Lemma: smoothness-penalty-function} Suppose that Assumption \ref{Ass: boundedness of c and f} holds.
    Then, for any $k \geq 1$ and $\rho \ge \rho_0$,
    the function
    $Q_{\rho}(\bx)$ is $L_{\rho}$-smooth on $\RR^{n}$, where
    $L_{\rho}= \rho \tilde{L}$ with $\tilde{L}= \rho_0^{-1} L_f+ G^2 + CL_c$.
\end{lemma}
The following lemma shows that the potential function value sequence $\{\calL_{\rho_k, \mu_k}(\bw^k)\}_{k \in \NN}$ in \eqref{Eq: the potential function} is non-increasing up to a stochastic gradient error term. To ensure this descent, we assume the parameters ${\mu_k}$, ${\rho_k}$ and $\beta$ satisfy
\begin{align}
    \label{Eq: basic parameter conditions for both algorithms}
    &\mu_{k}L_{\rho_k} \leq \frac{1}{4}, \quad
    \mu_{k+1} \leq \mu_{k}, \quad
    \rho_{k} \leq \rho_{k+1}, \quad
    0 < \beta \leq 1, \quad \forall\, k \geq 0.
\end{align}

\begin{lemma}
    \label{Lemma: approximate-sufficient-descent}
    Suppose that \Cref{Ass: boundedness of c and f}
    holds, and the parameters $\mu_{k}$, $\rho_{k}$ and $\beta$ satisfy \eqref{Eq: basic parameter conditions for both algorithms}.
    Then, for any $k \geq 0$, it holds that
    \begin{align}
        \label{eq:approximate-sufficient-descent}
        \calL_{\rho_{k+1}, \mu_{k+1}}(\bw^{k+1}) \leq&\, \calL_{\rho_k, \mu_k}(\bw^{k}) - \frac{(2\mu_{k})^{-1}- L_{\rho_k}}{2}\| \bw^{k+1}- \bw^{k}\|^{2} + \frac{\rho_{k+1}- \rho_{k}}{2}C^{2} + \Delta_{k+1} +  \mu_{k}\| \be^{k}\|^{2},
    \end{align}
where $\Delta_{k+1}
:= \frac{|\mu_{k}-\mu_{k+1}|}{2\mu_{k+1}^2}
(C^{2} + \|\bx^{k+1}-\bz^{k+1}\|^{2})$. Moreover, it holds that
\begin{align}
\label{eq:sequence-residual-bound}
\| \bw^{k+1}- \bw^{k}\|^{2}
\leq&\, \frac{4 \mu_{k}}{1 - 2 \mu_{k}L_{\rho_k}}
   \bigl(
      \calL_{\rho_k, \mu_k}(\bw^{k})
      - \calL_{\rho_{k+1}, \mu_{k+1}}(\bw^{k+1})
      + \Delta_{k+1} + \frac{\rho_{k+1}- \rho_{k}}{2}C^{2}
      + \mu_{k}\| \be^{k}\|^{2}
   \bigr).
\end{align}
\end{lemma}

\begin{proof}
    Since $\mu_{k}L_{\rho_k}\leq \frac{1}{4}$, it follows that
    $\frac{(2\mu_{k})^{-1}- L_{\rho_k}}{2}= \frac{1}{4 \mu_{k}}- \frac{L_{\rho_k}}{2}
    = \frac{1 - 2\mu_k L_{\rho_k}}{4\mu_k}
    \geq \frac{1}{8 \mu_{k}}> 0$.
    To prove \eqref{eq:approximate-sufficient-descent},
    we bound $\calL_{\rho_{k+1}, \mu_{k+1}}(\bw^{k+1}) - \calL_{\rho_k, \mu_k}(\bw^{k})$ by decomposing
    it as
    \begin{flalign}
        &\calL_{\rho_{k+1}, \mu_{k+1}}(\bw^{k+1}) - \calL_{\rho_k,\mu_k}(\bw^{k}) = [ \calL_{\rho_{k+1}, \mu_{k+1}}(\bw^{k+1}) - \calL_{\rho_k, \mu_k}(\bw^{k+1}) ]  + [ \calL_{\rho_k, \mu_{k}}(\bw^{k+1}) - \calL_{\rho_k, \mu_k}(\bw^{k}) ]. \label{eq:function-value-decomposition}
    \end{flalign}

For the first term in \eqref{eq:function-value-decomposition}, we have
\begin{align}
\calL_{\rho_{k+1}, \mu_{k+1}}(\bw^{k+1}) - \calL_{\rho_k, \mu_k}(\bw^{k+1})  \leq \frac{\rho_{k+1}- \rho_{k}}{2}C^{2} + \Delta_{k+1}, \label{eq:penalty-term-bound}
\end{align}
where we use $\|\bc(\bx^{k+1})\|\le C$ from \Cref{Ass: boundedness of c and f}.

    For the second term of \eqref{eq:function-value-decomposition}, we decompose it as
    \begin{flalign}
        &\calL_{\rho_k, \mu_k}(\bw^{k+1}) - \calL_{\rho_k, \mu_k}(\bw^{k}) \notag \\
        =&\, [ \calL_{\rho_k, \mu_k}(\bx^{k+1}, \bz^{k+1}) - \calL_{\rho_k, \mu_k}(\bx^{k+1}, \bz^{k}) ] + [ \calL_{\rho_k, \mu_k}(\bx^{k+1}, \bz^{k}) - \calL_{\rho_k, \mu_k}(\bx^{k}, \bz^{k}) ]. \label{eq:objective-function-decomposition}
    \end{flalign}
    The first part of \eqref{eq:objective-function-decomposition} is
\begin{align}
&\calL_{\rho_k,\mu_k}(\bx^{k+1},\bz^{k+1})
 - \calL_{\rho_k,\mu_k}(\bx^{k+1},\bz^{k}) \notag \\
=&\, \frac{1}{2\mu_k}(\|\bx^{k+1}-\bz^{k+1}\|^{2}
    - \|\bx^{k+1}-\bz^{k}\|^{2}) +  \calM_{\mu_k g}(\bz^{k}) - \calM_{\mu_k g}(\bz^{k+1}) \notag \\
\leq&\,\frac{1}{2\mu_k}(\|\bx^{k+1}-\bz^{k+1}\|^{2}
    - \|\bx^{k+1}-\bz^{k}\|^{2})  - \frac{1}{\mu_k}\langle \bz^{k}-\prox_{\mu_k g}(\bz^{k}),
     \bz^{k+1}-\bz^{k}\rangle. \label{eq:z-first-part}
\end{align}
Using the identity
\(
\|\ba\|^{2}-\|\bb\|^{2}
= 2\langle \ba-\bb,\ba\rangle - \|\ba-\bb\|^{2},
\)
with $\ba=\bx^{k+1}-\bz^{k+1}$ and $\bb=\bx^{k+1}-\bz^{k}$, we obtain
\begin{align*}
\|\bx^{k+1}-\bz^{k+1}\|^{2}-\|\bx^{k+1}-\bz^{k}\|^{2}
= 2\langle \bz^{k}-\bz^{k+1},\bx^{k+1}-\bz^{k+1}\rangle
  - \|\bz^{k+1}-\bz^{k}\|^{2}.
\end{align*}
Substituting this into \eqref{eq:z-first-part} and combining the optimality condition of the subproblem \eqref{eq:bz_update},
\begin{align}
\label{Eq: the optimality of z}
    \prox_{\mu_k g}(\bz^k) -  \bx^{k+1} = \beta^{-1}(\bz^{k} - \bz^{k+1}),
\end{align}
we obtain
\begin{align}
\calL_{\rho_k,\mu_k}(\bx^{k+1},\bz^{k+1})
 - \calL_{\rho_k,\mu_k}(\bx^{k+1},\bz^{k})
\le -\frac{1}{\mu_k}(\beta^{-1}-\frac{1}{2})
   \|\bz^{k+1}-\bz^{k}\|^{2}. \label{eq:z-update-bound}
\end{align}

For the second part of \eqref{eq:objective-function-decomposition}, by the optimality condition of the subproblem \eqref{eq:bx_update_Polyak} or \eqref{eq:bx_update_storm}, together with its strong convexity, one has
\begin{align}
h(\bx^{k+1})
+ \langle \be^{k} + \nabla Q_{\rho_k}(\bx^{k}),\, \bx^{k+1}-\bx^{k}\rangle
   + \frac{1}{2\mu_k}\|\bx^{k+1}-\bz^{k}\|^{2} \le h(\bx^{k}) + \frac{1}{2\mu_k}\|\bx^{k}-\bz^{k}\|^{2}
   - \frac{1}{2\mu_k}\|\bx^{k+1}-\bx^{k}\|^{2}. \notag
\end{align}
Using the $L_{\rho_k}$-smoothness of $Q_{\rho_k}$ and the fact
$\langle \ba,\bb\rangle \le \tfrac{1}{4\mu_k}\|\ba\|^{2}+ \mu_k\|\bb\|^{2}$, one has
\begin{align}
\calL_{\rho_k,\mu_k}(\bx^{k+1},\bz^{k})
 - \calL_{\rho_k,\mu_k}(\bx^{k},\bz^{k})
\le -\frac{(2\mu_k)^{-1}- L_{\rho_k}}{2}\|\bx^{k+1}-\bx^{k}\|^{2}
+ \mu_k\|\be^{k}\|^{2}.
\label{eq:x-update-bound}
\end{align}
Then, combining \eqref{eq:z-update-bound} and \eqref{eq:x-update-bound} in
\eqref{eq:objective-function-decomposition}, we obtain
\begin{align}
\calL_{\rho_k,\mu_k}(\bw^{k+1}) - \calL_{\rho_k,\mu_k}(\bw^{k})
\le  - \min\{\frac{(2\mu_k)^{-1}-L_{\rho_k}}{2},\,
\frac{1}{\mu_k}(\beta^{-1}-\tfrac{1}{2})\}
\|\bw^{k+1}-\bw^{k}\|^{2}  + \mu_k\|\be^{k}\|^{2}.
\label{eq:combined-update-bound}
\end{align}
Since \eqref{Eq: basic parameter conditions for both algorithms} ensures $\frac{(2\mu_k)^{-1} - L_{\rho_k}}{2} \le \frac{1}{\mu_k}(\beta^{-1}-\tfrac{1}{2})$, one has \( \min\{\frac{(2\mu_k)^{-1}-L_{\rho_k}}{2},\,
\frac{1}{\mu_k}(\beta^{-1}-\tfrac{1}{2})\} = \frac{(2\mu_k)^{-1}-L_{\rho_k}}{2}\). Then, substituting \eqref{eq:penalty-term-bound} and \eqref{eq:combined-update-bound} into
\eqref{eq:function-value-decomposition} yields \eqref{eq:approximate-sufficient-descent}. Since the coefficient of $\| \bw^{k+1}- \bw^{k}\|^{2}$ in \eqref{eq:approximate-sufficient-descent} is positive, rearranging this inequality gives \eqref{eq:sequence-residual-bound}, completing the proof.
\end{proof}
We next derive the one-step bound on \(\bu\) and the associated criticality measure during iteration.
\begin{lemma}
    \label{Lemma: criticality-stationarity-measure} Suppose that \Cref{Ass: boundedness of c and f} holds.
    Then, for any
    $k \geq 0$, with $\bu^{k+1}$ defined by \eqref{Eq: u-define-stochastic-polyak} or \eqref{Eq: u-define-stochastic-storm}, it holds
    that
    \begin{align}
        \label{Eq: the boundedness of subgradient}\| \bu^{k+1}\|^{2}\leq 3\| \be^{k}\|^{2}+ 3\bigl( L_{\rho_k}^{2}+ (\mu_{k}\beta)^{-2}\bigr) \| \bw^{k+1}- \bw^{k}\|^{2},
    \end{align}
and
    \begin{align}
    \label{Eq: the criticality measure}\max\{\| \bu^{k+1}\|^{2}, \| \bx^{k+1}- \prox_{\mu_k g}(\bz^{k})\|^{2}\}
        \le  3\| \be^{k}\|^{2}+ \bigl(3 \left(L_{\rho_k}^{2}+ (\mu_{k}\beta)^{-2}\right) + \beta^{-2}\bigr) \| \bw^{k+1}- \bw^{k}\|^{2}.
    \end{align}
\end{lemma}
\begin{proof}
Recalling \eqref{Eq: the optimality of z} and substituting it into \eqref{Eq: u-define-stochastic-polyak} or \eqref{Eq: u-define-stochastic-storm}, one obtains
   \begin{align}
           \bu^{k+1}= - \be^{k}+ \nabla Q_{\rho_{k}}(\bx^{k+1}) -\nabla Q_{\rho_{k}}(\bx^{k}) + \frac{1}{\mu_{k}\beta}(\bz^{k}- \bz^{k+1}). \notag
    \end{align}
Consequently, it holds that
    \begin{align}
        \| \bu^{k+1}\|^{2} & \leq 3\| \be^{k}\|^{2}+ 3 \| \nabla Q_{\rho_{k}}(\bx^{k+1}) -\nabla Q_{\rho_{k}}(\bx^{k})\|^{2}+ 3(\mu_{k}\beta)^{-2}\| \bz^{k+1}- \bz^{k}\|^{2}\notag    \\
                           & \leq 3\| \be^{k}\|^{2}+ 3L_{\rho_k}^{2}\| \bx^{k+1}- \bx^{k}\|^{2}+ 3(\mu_{k}\beta)^{-2}\| \bz^{k+1}- \bz^{k}\|^{2}\notag                                 \\
                           & \leq 3\| \be^{k}\|^{2}+ 3\left(L_{\rho_k}^{2}+ (\mu_{k}\beta)^{-2}\right)\| \bw^{k+1}- \bw^{k}\|^{2},
                           \label{Eq: the boundedness of subgradient in proof}
    \end{align}
where the first inequality follows from
    $\| \ba + \bb + \bc \|^{2}\leq 3(\| \ba \|^{2}+ \| \bb \|^{2}+ \| \bc \|^{2})$ and the second inequality uses the $L_{\rho_k}$-smoothness of $Q_{\rho_k}$, establishing \eqref{Eq: the boundedness of subgradient}.

    For the left-hand side of \eqref{Eq: the criticality measure}, one has, by \eqref{Eq: the optimality of z},
    \begin{align}
        \max\{\|\bu^{k+1}\|^{2}, \|\bx^{k+1}- \prox_{\mu_k g}(\bz^{k})\|^{2}\} \leq \|\bu^{k+1}\|^{2}+ \beta^{-2}\|\bz^{k+1}- \bz^{k}\|^{2}. \notag
    \end{align}
Combining this with \eqref{Eq: the boundedness of subgradient in proof} gives \eqref{Eq: the criticality measure}.
\end{proof}
In general, finding a feasible solution for nonconvex constrained optimization problems is challenging. It is thus necessary to characterize infeasible stationarity measures \(\| \nabla \bc(\bx)\bc(\bx)\|^2\). Under a constraint qualification such as \Cref{Ass: constraints conditions}, we can further bound constraint violations \(\| \bc(\bx) \|^{2}\). The following lemma establishes key bounds relating \(\bu\) to these two measures.
\begin{lemma}
    \label{Lemma: the error bound of infeasible stationarity and kkt conditions}
    Suppose
    that the conditions of Lemma \ref{Lemma: approximate-sufficient-descent}
    hold. Let $\underline{\rho}_K:=\min_{0\leq k\leq K-1}\rho_k$. Then, for any $K \geq 1$, it holds that
    \begin{align}
        \label{Eq: the bound of infeasible stationarity-averaging}\frac{1}{K}\sum_{k=0}^{K-1}\| \nabla \bc(\bx^{k+1}) \bc(\bx^{k+1}) \|^{2}\leq \frac{2}{\underline{\rho}_K^{2}K}\sum_{k=0}^{K-1}\| \bu^{k+1}\|^{2}+ \frac{18 G^{2}}{\underline{\rho}_K^{2}}.
    \end{align}

    Moreover, under \Cref{Ass: constraints conditions}, it holds that
    \begin{align}
        \label{Eq: the constraint violation-averaging}
        \frac{1}{K}\sumK \| \bc(\bx^{k+1}) \|^{2}\leq \frac{2}{\underline{\rho}_K^{2}\delta^{2}K}\sumK \| \bu^{k+1}\|^{2}+ \frac{18 G^{2}}{\delta^{2}\underline{\rho}_K^{2}}.
    \end{align}
\end{lemma}
\begin{proof}
By \eqref{Eq: the definition of u for analysis}, there exist $\bv_h^{k+1} \in \partial h(\bx^{k+1})$ and $\nabla \calM_{\mu_k g}(\bz^{k}) \in \partial g(\prox_{\mu_k g}(\bz^k))$ such that
\(
\bu^{k+1} = \nabla Q_{\rho_k}(\bx^{k+1}) + \bv_{h}^{k+1} - \nabla \calM_{\mu_k g}(\bz^{k}).
\)
This yields
\begin{align}
    &\frac{1}{\rho_{k}}\| \rho_{k}\nabla \bc(\bx^{k+1}) \bc(\bx^{k+1}) \| \notag \\
    =&\, \frac{1}{\rho_{k}}\| \nabla f(\bx^{k+1}) + \rho_{k}\nabla \bc(\bx^{k+1}) \bc(\bx^{k+1}) - \nabla f(\bx^{k+1}) \| \notag  \\
    =&\, \frac{1}{\rho_{k}}\| \nabla Q_{\rho_k}(\bx^{k+1}) + \bv_{h}^{k+1}- \nabla \calM_{\mu_k g}(\bz^{k}) - \nabla f(\bx^{k+1}) - \bv_{h}^{k+1}+ \nabla \calM_{\mu_k g}(\bz^{k}) \| \notag \\
    \leq&\, \frac{1}{\rho_{k}}\bigl( \| \bu^{k+1}\| + \| \nabla f(\bx^{k+1}) \| + \| \bv_{h}^{k+1}\| + \| \nabla \calM_{\mu_k g}(\bz^{k}) \| \bigr) \notag \\
    \leq&\, \frac{\| \bu^{k+1}\|}{\rho_{k}}+ \frac{3 G}{\rho_{k}}, \label{eq:infeasible-stationarity-measure}
\end{align}
where the last inequality uses $\| \nabla f(\bx^{k+1}) \|$, $\| \bv_{h}^{k+1}\|$, $\| \nabla \calM_{\mu_k g}(\bz^{k}) \| \leq G$ from \Cref{Ass: boundedness of c and f}.
Squaring both sides of \eqref{eq:infeasible-stationarity-measure}
gives
    \begin{align}
        \| \nabla \bc(\bx^{k+1}) \bc(\bx^{k+1}) \|^{2}\leq \frac{2 \| \bu^{k+1}\|^{2}}{\rho_{k}^{2}}+ \frac{18 G^{2}}{\rho_{k}^{2}}.\label{eq:infeasible-stationarity-squared}
    \end{align}
Summing \eqref{eq:infeasible-stationarity-squared} over $k = 0, \dots, K-1$ and
    dividing by $K$, we obtain
    \begin{align*}
        \frac{1}{K}\sumK \| \nabla \bc(\bx^{k+1}) \bc(\bx^{k+1}) \|^{2}\leq \frac{2}{K}\sumK \frac{\| \bu^{k+1}\|^{2}}{\rho_{k}^{2}}+ \frac{18 G^{2}}{K}\sumK \frac{1}{\rho_{k}^{2}},
    \end{align*}
where using $\rho_{k}\geq \underline{\rho}_K$ for $0\leq k\leq K-1$
    establishes \eqref{Eq: the bound of infeasible stationarity-averaging}.

    Under \Cref{Ass: constraints conditions}, it is straightforward to obtain that
    \begin{align*}
        \frac{1}{K}\sumK \| \bc(\bx^{k+1}) \|^{2}\leq \frac{1}{ \delta^{2}K }\sumK \| \nabla \bc(\bx^{k+1}) \bc(\bx^{k+1}) \|^{2},
    \end{align*}
where substituting \eqref{Eq: the bound of infeasible stationarity-averaging}
    into the above inequality yields \eqref{Eq: the constraint violation-averaging},
    completing the proof.
\end{proof}

The following lemma establishes an expected averaged bound on $\EE[\| \bu^{k+1}\|^{2}]$.
\begin{lemma}
    \label{Lemma: key u measure averaged in expectation} Suppose that
    the conditions of Lemma \ref{Lemma: approximate-sufficient-descent}
    hold. Then, for $K \geq 1$, it holds that
\begin{align}
        \label{Eq: the boundedness of subgradient-averaging}
        \frac{1}{K}\sumK \EE[\| \bu^{k+1}\|^{2}]
        \leq{} & \, \frac{24 ( \tfrac{1}{16} + \beta^{-2})}{\mu_{K}K}
        ( \calL_{\rho_0, \mu_0}(\bw^{0}) - \calL^{*}) \notag \\
        &+ \frac{3 + 24 ( \tfrac{1}{16} + \beta^{-2})}{K}
        \sumK \EE[\| \be^{k}\|^{2}]\notag \\
        &+ \frac{24 ( \tfrac{1}{16} + \beta^{-2})}{\mu_{K}K}
        \sumK \left( \frac{\rho_{k+1}- \rho_{k}}{2}C^{2}
        + \EE[\Delta_{k+1}] \right),
    \end{align}
    where $\calL^{*}$  is as defined in \eqref{Eq: uniform lower bound L star}.
\end{lemma}
\begin{proof}
Taking expectations in \eqref{Eq: the boundedness of subgradient} and averaging over $k = 0,\dots, K-1$, one obtains
    \begin{align}
        \frac{1}{K}\sum_{k=0}^{K-1}\EE[\| \bu^{k+1}\|^{2}] \leq \frac{3}{K}\sum_{k=0}^{K-1}( L_{\rho_k}^{2}+ (\mu_{k}\beta )^{-2}) \EE[\| \bw^{k+1}- \bw^{k}\|^{2}] + \frac{3}{K}\sumK \EE[\| \be^{k}\|^{2}]. \label{eq:subgradient-averaging-first}
    \end{align}
Next, substituting \eqref{eq:sequence-residual-bound} into the first term of \eqref{eq:subgradient-averaging-first} gives
\begin{align}
  &\frac{3}{K}\sum_{k=0}^{K-1}
  ( L_{\rho_k}^{2}+ (\mu_{k}\beta)^{-2})
  \EE[\| \bw^{k+1}- \bw^{k}\|^{2}] \notag \\
   \leq{} & \,\frac{12}{K}\sum_{k=0}^{K-1}
   \frac{\mu_{k}^{2}L_{\rho_k}^{2}+ \beta^{-2}}
   {\mu_{k}(1 - 2 \mu_{k}L_{\rho_k})}\biggl(
   \EE\bigl[
   \calL_{\rho_k, \mu_k}(\bw^{k})
   - \calL_{\rho_{k+1}, \mu_{k+1}}(\bw^{k+1})
   \bigr] \notag \\
   &\qquad\qquad\qquad
   + \frac{\rho_{k+1}- \rho_{k}}{2}C^{2}
   + \EE[\Delta_{k+1}]
   +\mu_{k}\EE[\| \be^{k}\|^{2}]
   \biggr).
   \label{eq:subgradient-averaging-second}
    \end{align}
    Given $\mu_{k}\geq\mu_{K}$ and $\mu_{k}L_{\rho_k}\leq \tfrac{1}{4}$ from \eqref{Eq: basic parameter conditions for both algorithms}, it holds that
\begin{align}
    \frac{\mu_{k}^{2}L_{\rho_k}^{2}+ \beta^{-2}}{\mu_{k}( 1 - 2 \mu_{k}L_{\rho_k})}\leq \frac{2}{\mu_{K}}( \frac{1}{16}+ \beta^{-2}),
\end{align}
where we use the inequality $\frac{v^2 + \beta^{-2}}{1 - 2v}\leq 2 (\tfrac{1}{16} + \beta^{-2})$ with $v = \mu_{k}L_{\rho_k}$.
Then, substituting this into \eqref{eq:subgradient-averaging-second} yields
\begin{align}
    & \frac{3}{K}\sum_{k=0}^{K-1}( L_{\rho_k}^{2}+ (\beta \mu_{k})^{-2})
    \EE[\| \bw^{k+1}- \bw^{k}\|^{2}] \notag \\
    \leq &\,\frac{24 ( \tfrac{1}{16} + \beta^{-2})}{\mu_{K}K}
    \bigl( \calL_{\rho_0, \mu_0}(\bw^{0}) - \calL^{*}\bigr) + \frac{24 ( \tfrac{1}{16} + \beta^{-2})}{\mu_{K}K} \sumK \bigl( \frac{\rho_{k+1}- \rho_{k}}{2} C^{2} + \EE[\Delta_{k+1}]\bigr) \notag \\
    & + \frac{24 ( \tfrac{1}{16} + \beta^{-2})}{K}\sumK \EE[\| \be^{k}\|^{2}]. \notag
\end{align}
Combining this with the second term of \eqref{eq:subgradient-averaging-first}, we obtain the desired inequality, completing the proof.
\end{proof}
We now proceed to estimate the quantities in \eqref{Eq: approximate KKT point} and \eqref{Eq: approximate stationarity point} for the output point $\bx^{R+1}$. The result follows directly from \Cref{Lemma: the error bound of infeasible stationarity and kkt conditions}
and \Cref{Lemma: key u measure averaged in expectation}, and the proof is omitted for brevity.
\begin{lemma}
    \label{Lemma: the error bound of the infeasible stationarity in expectation}
    Suppose that
    the conditions of Lemma \ref{Lemma: approximate-sufficient-descent}
    hold. Let $\underline{\rho}_K:=\min_{0\leq k\leq K-1}\rho_k$. Then, it holds that
\begin{align}
        & \max\{ \EE[\| \bu^{R+1}\|^{2}],
        \EE [\| \bx^{R+1}- \prox_{\mu_R g}(\bz^{R})\|^{2}]\} \notag \\
        \leq{} & \frac{24( \tfrac{1}{16} + \beta^{-2}+ \tfrac{1}{3}(\mu_{0}/ \beta)^{2})}{\mu_{K}K}
        (\calL_{\rho_0, \mu_0}(\bw^{0}) - \calL^{*})
        + \frac{24( \tfrac{1}{16} + \beta^{-2}+ \tfrac{1}{3}(\mu_{0}/ \beta)^{2})}{\mu_{K}K}
        \sumK \left( \frac{\rho_{k+1}- \rho_{k}}{2} C^{2} + \EE[\Delta_{k+1}] \right)
        \notag\\
        +{} & \frac{24( \tfrac{1}{16} + \beta^{-2}+ \tfrac{1}{3}(\mu_{0}/ \beta)^{2}) + 3}{K}
        \sumK \EE[\| \be^{k}\|^{2}],
        \label{Eq: the criticality measure error bound in expectation}
    \end{align}
    and
\begin{align}
        & \EE [\|\nabla \bc(\bx^{R+1}) \bc(\bx^{R+1}) \|^{2}] \notag\\
        \leq{} & \frac{48(\tfrac{1}{16}+ \beta^{-2})}{\underline{\rho}_K^{2}\mu_{K}K}
        \biggl(\calL_{\rho_0, \mu_0}(\bw^{0}) - \calL^{*}
        + \sumK \left( \frac{\rho_{k+1}- \rho_{k}}{2}C^{2}
        + \EE [\Delta_{k+1}] \right) \biggr) \notag \\
        +{} & \frac{18 G^{2}}{\underline{\rho}_K^{2}}
        + \frac{6 + 48(\tfrac{1}{16}+ \beta^{-2})}{\underline{\rho}_K^{2}K }
        \sumK \EE[\| \be^{k}\|^{2}].
        \label{Eq: the infeasible stationarity error bound in expectation}
    \end{align}
    Moreover, under \Cref{Ass: constraints conditions}, it holds that
\begin{align}
        & \EE[\| \bc(\bx^{R+1}) \|^{2}] \notag\\
        \leq{} & \frac{48(\tfrac{1}{16}+ \beta^{-2})}{\delta^{2}\underline{\rho}_K^{2}\mu_{K}K}
        \biggl(\calL_{\rho_0, \mu_0}(\bw^{0}) - \calL^{*}
        + \sumK \left( \frac{\rho_{k+1}- \rho_{k}}{2}C^{2}
        + \EE[ \Delta_{k+1}] \right) \biggr) \notag\\
       +{} & \frac{18 G^{2}}{\delta^{2}\underline{\rho}_K^{2}}
       +\frac{6 + 48(\tfrac{1}{16}+ \beta^{-2})} {\delta^{2}\underline{\rho}_K^{2}K }
       \sumK \EE[\| \be^{k}\|^{2}].
       \label{Eq: the constraints violations error bound in expectation}
    \end{align}
    
\end{lemma}

\subsection{\bf{Oracle Complexity Analysis of MoSSP-P}}
\label{Subsection: Lemmas for Oracle complexity analysis of Algo 1}
We begin by establishing a recursive relationship between $\EE[\|\be^{k+1}\|^{2}]$ and $\EE[\|\be^{k}\|^{2}]$.
The following lemma is adapted from \citep[Lemma 5.2,][]{gao2024non}. For completeness, we provide the proof.
\begin{lemma}
    \label{Lemma: Recursive gradient error with Polyak momentum}
    Suppose that Assumptions \ref{Ass: boundedness of c and f}-\ref{Ass: unbiased gradient oracle} hold. Then, for any $k \geq 0$ and $0 < \alpha_k < 1$, it holds that
    \begin{align}
        \label{Eq: the recursive gradient relationship in Polyak momentum scheme}
        \EE[\|\be^{k+1}\|^{2}] \leq (1 - \alpha_{k})\EE[\|\be^{k}\|^{2}] + \frac{1}{\alpha_{k}}L^2_{f}\EE[\|\bw^{k+1}- \bw^{k}\|^{2}] + \alpha_{k}^{2}\sigma^{2}.
\end{align}

\end{lemma}

\begin{proof}
Let $\xi^{[k]} := \{\xi^0,\ldots,\xi^k\}$ denote the collection of i.i.d. samples drawn up to iteration $k$ in MoSSP-P, and let $\EE[\cdot \mid \xi^{[k]}]$ denote the corresponding conditional expectation.
Recalling \eqref{Eq: the stochastic error} and \eqref{Eq: enhanced gradient estimator Polyak}, one has
\begin{align}
&\EE[\|\be^{k+1}\|^2|\xi^{[k]}] \notag \\
=&\, \EE[\|\bS^{k+1}- \nabla Q_{\rho_{k+1}}(\bx^{k+1})\|^2|\xi^{[k]}] \notag\\
=&\, \EE[\|\bs^{k+1} - \nabla f(\bx^{k+1})\|^2|\xi^{[k]}] \notag \\
=&\, \EE[\|(1-\alpha_k)(\bs^k - \nabla f(\bx^{k+1})) + \alpha_k(\nabla \tf(\bx^{k+1},\xi^{k+1}) - \nabla f(\bx^{k+1}))\|^2 |\xi^{[k]}] \notag \\
=&\, (1-\alpha_k)^2 \EE[\|\bs^k - \nabla f(\bx^{k+1})\|^2|\xi^{[k]}]  + \alpha_k^2 \EE[\|\nabla \tf(\bx^{k+1},\xi^{k+1})-\nabla f(\bx^{k+1})\|^2|\xi^{[k]}] \notag \\
&\quad + 2(1-\alpha_k)\alpha_k \EE[\langle \bs^k - \nabla f(\bx^{k+1}),\nabla \tf(\bx^{k+1},\xi^{k+1})-\nabla f(\bx^{k+1})\rangle|\xi^{[k]}].  \label{g-bnd}
\end{align}
Notice that it follows from \Cref{Ass: unbiased gradient oracle} that
\begin{align}
&\EE[\langle\bs^k - \nabla f(\bx^{k+1}), \nabla \tf(\bx^{k+1},\xi^{k+1})-\nabla f(\bx^{k+1})\rangle|\xi^{[k]}]\notag \\
=&\,\langle\bs^k - \nabla f(\bx^{k+1}),  \EE[\nabla \tf(\bx^{k+1},\xi^{k+1})-\nabla f(\bx^{k+1})|\xi^{[k]}] \rangle
=0.
\end{align}
Then, substituting this into \eqref{g-bnd} yields
\begin{align}
&\EE[\|\be^{k+1}\|^2|\xi^{[k]}] \notag \\
=&\,  (1-\alpha_k)^2\EE[\|\bs^k - \nabla f(\bx^{k}) +  \nabla f(\bx^{k}) - \nabla f(\bx^{k+1})\|^2|\xi^{[k]}] \notag \\
&\quad + \alpha_k^2 \EE[\|\nabla \tf(\bx^{k+1},\xi^{k+1}) - \nabla f(\bx^{k+1})\|^2 |\xi^{[k]}] \notag \\
\leq&\, (1-\alpha_k)^2 ( \frac{1}{1-\alpha_k}\EE[\|\be^k\|^2 |\xi^{[k]}] +  \frac{1}{\alpha_k}\|\nabla f(\bx^{k}) - \nabla f(\bx^{k+1})\|^2) + \alpha_k^2 \sigma^2 \notag \\
=&\, (1-\alpha_k)\EE[\|\be^k\|^2|\xi^{[k]}] + \frac{(1-\alpha_k)^2}{\alpha_k}\|\nabla f(\bx^{k}) - \nabla f(\bx^{k+1})\|^2 + \alpha_k^2 \sigma^2 \notag \\
\leq&\, (1-\alpha_k)\EE[\|\be^k\|^2|\xi^{[k]}] + \frac{1}{\alpha_k}L_f^2\|\bw^{k+1}- \bw^{k}\|^{2} + \alpha_k^2 \sigma^2,
\end{align}
where the second inequality uses Young's inequality $\|a+b\|^2 \leq (1+\gamma)\|a\|^2 + (1+\gamma^{-1})\|b\|^2$ with $\gamma = \frac{\alpha_k}{1-\alpha_k} > 0$ and the bounded variance condition in \Cref{Ass: unbiased gradient oracle}, and the third inequality uses the $L_f$-smoothness in \Cref{Ass: boundedness of c and f} along with $(1-\alpha_k)^2 \leq 1$ for $0 < \alpha_k < 1$. Taking the full expectation on both sides of this inequality yields \eqref{Eq: the recursive gradient relationship in Polyak momentum scheme}, completing the proof.
\end{proof}
As can be seen from \Cref{Lemma: the error bound of the infeasible stationarity in expectation}, the convergence rate is affected by the accumulated error $\frac{1}{K}\sum_{k=0}^{K-1} \EE [\|\be^{k}\|^{2}]$. To establish an upper bound for this accumulated error, we impose the following parameter setting
\begin{align}
    \label{Eq: the additional parameters assumptions in Polyak momentum}
   \frac{1}{\alpha_{k}}\leq \frac{1}{2\mu_{k}}, \, L_{\rho_k}+ \frac{2}{\alpha_{k}}L^2_{f}\leq \frac{3}{4}\frac{1}{2\mu_{k}}.
\end{align}
\begin{lemma}
    \label{Lemma: Weighted error accumulation of Polyak} Suppose that Assumptions \ref{Ass: boundedness of c and f}-\ref{Ass: unbiased gradient oracle} hold, and the
    parameters satisfy \eqref{Eq: basic parameter conditions for both algorithms}
    and \eqref{Eq: the additional parameters assumptions in Polyak momentum}. Then,
    it holds that  for $K\geq 1$,
        \begin{align}
        \label{Eq: weighted cumulative error bound for Polyak}
        &\frac{1}{K}\sumK \mu_{k}\EE [\| \be^{k}\|^{2}]  \notag \\
        \leq&\,\frac{ 3( \calL_{\rho_0, \mu_0}(\bw^{0}) - \calL^{*})}{K}+ \frac{2 \EE [\| \be^{0}\|^{2}]}{K}+ \frac{3}{K}\sumK ( \tfrac{C^2}{2} (\rho_{k+1}- \rho_{k}) + \EE[\Delta_{k+1}] ) + \frac{2}{K}\sumK \alpha_{k}^{2}\sigma^{2}.
    \end{align}
Furthermore, if we set
    $\mu_{k}\equiv \mu , \, \alpha_{k}\equiv \alpha ,\, \rho_{k}\equiv \rho$
    for $k \geq 1$, it holds that for any $K\ge1,$
    \begin{align}
        \label{Eq: weighted cumulative error bound under constant parameters for Polyak}\frac{1}{K}\sumK \EE [ \| \be^{k}\|^{2}] \leq \frac{ 3 ( \calL_{\rho, \mu}(\bw^{0}) - \calL^{*}) }{\mu K}+ \frac{2\EE [\| \be^{0}\|^{2}]}{\mu K}+ \frac{2\alpha^{2}\sigma^{2}}{\mu}.
    \end{align}
\end{lemma}
\begin{proof}
Recalling \eqref{eq:approximate-sufficient-descent}, we take full expectation on both
    sides of the inequality and substitute it into \eqref{Eq: the recursive gradient relationship in Polyak momentum scheme}
    to yield
    \begin{equation}
        \label{Eq: the descent property of potential function with recursive gradient}
\begin{aligned}
&\EE [ \calL_{\rho_{k+1}, \mu_{k+1}}(\bw^{k+1}) + \| \be^{k+1}\|^{2}] \\
\leq{}&
\EE [ \calL_{\rho_k, \mu_k}(\bw^{k}) + \| \be^{k}\|^{2}] + \frac{1}{2}(- \frac{1}{2\mu_{k}}+ L_{\rho_k}+ \frac{2}{\alpha_{k}}L_f^{2})
\EE [ \| \bw^{k+1}- \bw^{k}\|^{2}] \\
&+ (\mu_{k} - \alpha_k)\EE [\| \be^{k}\|^{2}]
+ \frac{\rho_{k+1}- \rho_{k}}{2}C^{2}
+ \EE[\Delta_{k+1}] + \alpha_{k}^{2}\sigma^{2} \\
\leq{}&
\EE [ \calL_{\rho_k, \mu_k}(\bw^{k})
+ (1 - \alpha_{k}+ 2 \mu_{k}) \| \be^{k}\|^{2}] - \mu_{k}\EE [\| \be^{k}\|^{2}]
+ \frac{\rho_{k+1}- \rho_{k}}{2}C^{2} \\
&+ \frac{1}{2}(- \frac{1}{2\mu_{k}}+ L_{\rho_k}+ \frac{2}{\alpha_{k}}L_f^{2})
\EE [ \| \bw^{k+1}- \bw^{k}\|^{2}] + \EE[\Delta_{k+1}] + \alpha_{k}^{2}\sigma^{2} \\
\leq{}&
\EE [ \calL_{\rho_k, \mu_k}(\bw^{k}) + \| \be^{k}\|^{2}]
- \frac{\EE [ \| \bw^{k+1}- \bw^{k}\|^{2}]}{16\mu_k}- \mu_{k}\EE [\| \be^{k}\|^{2}]
+\frac{\rho_{k+1}- \rho_{k}}{2}C^{2}
+ \EE[\Delta_{k+1}]
+ \alpha_{k}^{2}\sigma^{2},
\end{aligned}
    \end{equation}
where the last inequality follows from $1 - \alpha_{k}+ 2\mu_{k}\leq 1$ and
    that $-\tfrac{1}{2}(2\mu_k)^{-1}+ \tfrac{1}{2}L_{\rho_k}+ \alpha_k^{-1}L_f^2\leq -\tfrac{1}{16\mu_k}$ under \eqref{Eq: the additional parameters assumptions in Polyak momentum}.

    Taking expectation in \eqref{eq:sequence-residual-bound} and using
    \(\mu_k L_{\rho_k}\leq \tfrac{1}{4}\), we obtain
    \begin{align}
        &\frac{1}{8}\frac{1}{2\mu_{k}}\EE[\| \bw^{k+1}- \bw^{k}\|^{2}] \notag \\
        \leq&\, \frac{1}{2}( \EE[\calL_{\rho_k, \mu_k}(\bw^{k})] - \EE[\calL_{\rho_{k+1}, \mu_{k+1}}(\bw^{k+1})] + \frac{\rho_{k+1}- \rho_{k}}{2}C^{2}+ \EE[\Delta_{k+1}]+ \mu_{k}\EE[\| \be^{k}\|^{2}] ). \notag
    \end{align}
Substituting this into \eqref{Eq: the descent property of potential function with recursive gradient}
    and rearranging the inequality yields
    \begin{align}
        \mu_{k}\EE [\| \be^{k}\|^{2}] \le &\, 3 \bigl( \EE[\calL_{\rho_k, \mu_k}(\bw^{k})] - \EE[\calL_{\rho_{k+1}, \mu_{k+1}}(\bw^{k+1})] + \frac{\rho_{k+1}- \rho_{k}}{2}C^{2} + \EE[\Delta_{k+1}] \bigr) \notag \\
        &+ 2( \EE [\| \be^{k}\|^{2}] - \EE [\| \be^{k+1}\|^{2}] ) + 2\alpha_{k}^{2}\sigma^{2}.
    \end{align}
Summing over $k = 0, \ldots, K-1$ and dividing by $K$ yields \eqref{Eq: weighted cumulative error bound for Polyak}.
    When $\mu_{k}\equiv \mu , \, \alpha_{k}\equiv \alpha ,\, \rho_{k}\equiv \rho$, one has that \(\rho_{k+1} - \rho_k = 0\) and \(\Delta_{k+1} = 0\), implying \eqref{Eq: weighted cumulative error bound under constant parameters for Polyak} and completing the proof.
\end{proof}

We now characterize the measurements in \eqref{Eq: approximate stationarity point} for MoSSP-P with constant parameters $\rho_{k}\equiv \rho$, $\mu_{k}\equiv \mu$, $\alpha_{k}\equiv \alpha$.
\begin{lemma}
    \label{Lemma: the error bound of the infeasible stationarity in expectation for Polyak}
    Suppose that the conditions of Lemma~\ref{Lemma: Weighted
    error accumulation of Polyak} hold. Then, it holds that
\begin{align}
\label{Eq: the criticality measure error bound in expectation for Polyak}
& \max\{ \EE[\| \bu^{R+1}\|^{2}] ,
\EE [\| \bx^{R+1}- \prox_{\mu g}(\bz^{R})\|^{2}]\} \notag\\
\leq{}&
\frac{1}{K}
\frac{96(\tfrac{1}{16} + \beta^{-2} + \tfrac{1}{3}(\mu/\beta)^{2}) + 9}{\mu}
( \calL_{\rho, \mu}(\bw^{0}) - \calL^{*}) + \frac{(48(\tfrac{1}{16} + \beta^{-2} + \tfrac{1}{3}(\mu/\beta)^{2}) + 6)
\EE [\| \be^{0}\|^{2}]}{\mu K } \notag\\
&+ \frac{(48(\tfrac{1}{16} + \beta^{-2} + \tfrac{1}{3}(\mu/\beta)^{2}) + 6)
\alpha^{2}\sigma^{2}}{\mu},
\end{align}
    and
\begin{align}
\label{Eq: the infeasible stationarity error bound in expectation for Polyak}
&\EE [\|\nabla \bc(\bx^{R+1}) \bc(\bx^{R+1}) \|^{2}]  \notag \\
\leq{}&
\frac{192(\tfrac{1}{16}+ \beta^{-2}) + 18}{\mu\rho^{2}K}
(\calL_{\rho, \mu}(\bw^{0}) - \calL^{*}) + \frac{(12 + 96(\tfrac{1}{16} + \beta^{-2}))\EE [\| \be^{0}\|^{2}]}
{\mu\rho^{2}K } \notag\\
&+ \frac{(12 + 96(\tfrac{1}{16} + \beta^{-2})) \alpha^{2}\sigma^{2}}{\mu\rho^{2}}
+\frac{18G^{2}}{\rho^{2}}.
\end{align}
\end{lemma}
\begin{proof}
    Substituting \eqref{Eq: weighted cumulative error bound under constant parameters for Polyak}
    into \eqref{Eq: the criticality measure error bound in expectation} and noticing that all parameters are constant,
    we derive the desired bound in \eqref{Eq: the criticality measure error bound in expectation for Polyak}.

    Then, substituting \eqref{Eq: weighted cumulative error bound under constant parameters for Polyak}
    into \eqref{Eq: the boundedness of subgradient-averaging} yields
    \begin{align}
             & \frac{1}{K}\sumK \EE[\| \bu^{k+1}\|^{2}] \notag   \\
        \leq&\, \frac{9 + 96(\tfrac{1}{16} + \beta^{-2})}{\mu K}( \calL_{\rho, \mu}(\bw^{0}) - \calL^{*}) + \frac{6 + 48(\tfrac{1}{16} + \beta^{-2})}{\mu K}\EE[\| \be^{0}\|^{2}] +\frac{6 + 48(\tfrac{1}{16} + \beta^{-2})}{\mu}\alpha^{2}\sigma^{2}. \notag
    \end{align}
Then, combining this bound with \eqref{Eq: the bound of infeasible stationarity-averaging}
    from Lemma \ref{Lemma: the error bound of infeasible stationarity and kkt conditions}
    directly gives \eqref{Eq: the infeasible stationarity error bound in expectation for Polyak},
    completing the proof.
\end{proof}

\subsubsection{Proof of \Cref{Lemma: the estimation of the convergence rate of Polyak for infeasible stationary point}}
\label{subsubsec: proof Lemma Polyak}

Selecting appropriate parameters is crucial as the final oracle complexity depends on these choices. To satisfy \eqref{Eq: basic parameter conditions for both algorithms} and \eqref{Eq: the additional parameters assumptions in Polyak momentum}, we adopt the following parameter setting:
\begin{align}
\label{Eq: the parameter conditions in Polyak momentum for analysis}
\rho_{k}\equiv \rho = \rho_0 K^{l}, \,  \mu_{k}\equiv \mu = \frac{\mu_0}{K^{\tau}\max\{L_f, \tilde{L}\}}, \,  \alpha_{k}\equiv \alpha = \frac{\alpha_0 \mu_0}{K^{\tau}}, \, 0 < \beta \leq 1,
\end{align}
where $0<l\leq \tau \leq 2l<1$, $0 < \mu_0 \leq \min\{\frac{1}{4\rho_0}, \frac{1}{4L_f}\}$, $\alpha_0 = \frac{2\gamma}{\max\{L_f, \tilde{L}\}}$, and $\gamma \geq \max\{1, 8L_f^{2}\}$ are constants independent of $K$. Then, substituting \eqref{Eq: the parameter conditions in Polyak momentum for analysis} into \Cref{Lemma: the error bound of the infeasible stationarity in expectation for Polyak} directly yields \eqref{Eq: convergence rate of polyak}. To reach the rate of $\calO(K^{-1/2})$ under approximate feasible initialization, we set $\tau = 2l = \tfrac{1}{2}$. We are now ready to prove \Cref{Theorem: the convergence rate of polyak with initial feasibility}.

\subsubsection{Proof of \Cref{Theorem: the convergence rate of polyak with initial feasibility}}
\label{proof: polyak with feasibility}
\begin{proof}
First, under \Cref{Ass: constraints conditions}, combining \eqref{Eq: the constraint violation-averaging} with \eqref{Eq: the infeasible stationarity error bound in expectation for Polyak} yields that the upper bound on $\EE[\|\bc(\bx^{R+1})\|^{2}]$ is in the same order as $\EE[\|\nabla\bc(\bx^{R+1})\bc(\bx^{R+1})\|^{2}]$.

Second, by the approximate feasibility condition $\|\bc(\bx^{0})\|^{2} = \calO(K^{-l})$, we have $\calL_{\rho,\mu}(\bw^{0}) + 1 = \calO(1)$ and the two upper bounds in \eqref{Eq: convergence rate of polyak} are simplified to
\[
\calO\!\left(\max\{K^{\tau-1},\,K^{-\tau}\}\right) \quad\text{and}\quad \calO\!\left(\max\{K^{-2l+\tau-1},\,K^{-2l}\}\right).
\]
If we choose $l = \tfrac{1}{4}$ and $\tau = \tfrac{1}{2}$, both reduce to $\calO(K^{-1/2})$. Hence, to obtain a stochastic $\varepsilon$-stationary point, $K$ is of order $\calO(\varepsilon^{-4})$, and the associated oracle complexity is $\calO(\varepsilon^{-4})$. Under \Cref{Ass: constraints conditions}, the complexity to reach a stochastic $\varepsilon$-KKT point is $\calO(\varepsilon^{-4})$.
\end{proof}

\subsubsection{Proof of \Cref{Theorem: the basic convergence rate of Polyak without initial feasibility}}
\label{app:polyak_without_initial_feasibility}
Without assuming initial approximate feasibility, we have $\calL_{\rho,\mu}(\bw^0)+1=\calO(K^l)$ under $\|\bc(\bx^0)\|^2=\calO(1)$ and $\rho=\calO(K^l)$. Moreover, the one-sample Polyak initialization and \Cref{Ass: unbiased gradient oracle} give $\EE[\|\be^0\|^2]=\calO(1)$. Combining these bounds with \eqref{Eq: convergence rate of polyak}, we obtain
\[
\max\bigl\{\EE[\|\bu^{R+1}\|^2],
\EE[\|\bx^{R+1}-\prox_{\mu g}(\bz^R)\|^2]\bigr\}
=
\calO\!\left(\max\{K^{l+\tau-1},\,K^{\tau-1},\,K^{-\tau}\}\right),
\]
and
\[
\EE[\|\nabla\bc(\bx^{R+1})\bc(\bx^{R+1})\|^2]
=
\calO\!\left(\max\{K^{-l+\tau-1},\,K^{-2l+\tau-1},\,K^{-2l}\}\right).
\]
Choosing $l=\tfrac{1}{5}$ and $\tau=\tfrac{2}{5}$ gives $\rho=\calO(K^{1/5})$, $\mu=\calO(K^{-2/5})$, and $\alpha=\calO(K^{-2/5})$. The two bounds reduce to
\[
\max\bigl\{\EE[\|\bu^{R+1}\|^2],
\EE[\|\bx^{R+1}-\prox_{\mu g}(\bz^R)\|^2]\bigr\}
=\calO(K^{-2/5}),
\]
and
\[
\EE[\|\nabla\bc(\bx^{R+1})\bc(\bx^{R+1})\|^2]
=\calO(K^{-2/5}).
\]
Hence, to make the squared residuals no larger than $\varepsilon^2$, it suffices to take $K=\calO(\varepsilon^{-5})$. Since MoSSP-P uses one stochastic first-order oracle call per iteration, the total oracle complexity for finding a stochastic $\varepsilon$-stationary point is $\calO(\varepsilon^{-5})$. Under \Cref{Ass: constraints conditions}, the same bound on $\EE[\|\nabla\bc(\bx^{R+1})\bc(\bx^{R+1})\|^2]$ yields $\EE[\|\bc(\bx^{R+1})\|^2]=\calO(K^{-2/5})$, and the same oracle complexity holds for finding a stochastic $\varepsilon$-KKT point. The corresponding certificate is $\bar{\by}=\prox_{\mu g}(\bz^R)$, $\bar{\blambda}=\rho\bc(\bx^{R+1})$, and $\bar{\bu}=\bu^{R+1}$; the inclusion in \eqref{Def: the definition of u} follows almost surely from \eqref{Eq: the definition of u for analysis}.

\subsection{\bf{Oracle Complexity Analysis of MoSSP-R}}
\label{Subsection: Lemmas for Oracle complexity analysis of Algo 2}
We now establish a recursive relationship between $\EE[\| \be^{k+1}\|^{2}]$ and $\EE[\| \be^{k}\|^{2}]$ in the following lemma. The proof follows from \citet{xu2023momentum}.
For completeness, we provide the proof.
\begin{lemma}
    \label{Lemma: Recursive bound on gradient error for storm}
    Suppose that Assumptions \ref{Ass: boundedness of c and f}, \ref{Ass: unbiased gradient oracle} and \ref{Ass: the expected smoothness of f} hold.
    Then, for any $k \geq 0$, it holds that
\begin{align}
        \label{Eq: the stochastic error of the variance estimator}
        \EE\left[\| \be^{k+1}\|^{2}\right]
        \leq{} & (1 - \alpha_{k}) \EE\left[\| \be^{k}\|^{2}\right] + 2L^2_{f}\EE\left[ \| \bw^{k+1}- \bw^{k}\|^{2}\right]
        + 2\alpha_{k}^{2}\sigma^{2}.
    \end{align}
\end{lemma}
\begin{proof}
Let $\xi^{[k]} := \{\xi^0,\ldots,\xi^k\}$ denote the collection of i.i.d. samples drawn up to iteration $k$ in MoSSP-R, and let $\EE[\cdot \mid \xi^{[k]}]$ denote the corresponding conditional expectation. Recalling \eqref{Eq: the stochastic error} and \eqref{Eq: gradient estimator storm}, it holds that
\begin{align}
    &\EE[\|\be^{k+1}\|^2|\xi^{[k]}] \notag \\
    =&\, \EE[\|\bD^{k+1}- \nabla Q_{\rho_{k+1}}(\bx^{k+1})\|^2|\xi^{[k]}] \notag\\
    =&\, \EE[\|\bd^{k+1} - \nabla f(\bx^{k+1})\|^2|\xi^{[k]}] \notag \\
    =&\, \EE[\| \nabla \tf(\bx^{k+1}, \xi^{k+1}) - \nabla f(\bx^{k+1}) + (1 - \alpha_{k})(\bd^{k}- \nabla \tf(\bx^{k}, \xi^{k+1})) \|^2 |\xi^{[k]}] \notag \\
    =&\, \EE[\| (1- \alpha_{k})\be^{k}+ \nabla \tf(\bx^{k+1}, \xi^{k+1}) - \nabla f(\bx^{k+1}) + (1 - \alpha_{k})(\nabla f(\bx^{k}) - \nabla \tf(\bx^{k}, \xi^{k+1})) \|^2 |\xi^{[k]}] \notag \\
     =&\,(1 - \alpha_{k})^{2}\EE[\|\be^{k}\|^{2}|\xi^{[k]}] + \EE\bigl[\| \nabla \tf(\bx^{k+1}, \xi^{k+1}) - \nabla f(\bx^{k+1}) + (1 - \alpha_{k})(\nabla f(\bx^{k}) - \nabla \tf(\bx^{k}, \xi^{k+1}))\|^{2} |\xi^{[k]}\bigr],
    \label{Eq: the recursive error of stochastic estimator}
\end{align}
where \eqref{Eq: the recursive error of stochastic estimator} follows from \Cref{Ass: unbiased gradient oracle}, which ensures that
\begin{align}
    \EE[\langle \nabla \tf(\bx^{k+1}, \xi^{k+1}) - \nabla f(\bx^{k+1}), \be^{k}\rangle \mid \xi^{[k]}] &= 0, \notag \\
    \EE[\langle \nabla \tf(\bx^{k}, \xi^{k+1}) - \nabla f(\bx^{k}), \be^{k}\rangle \mid \xi^{[k]}] &= 0. \notag
\end{align}
    Let us focus on the second term \eqref{Eq: the recursive error of stochastic estimator} by rewriting this term as
   \begin{align}
     &\EE[\| \nabla \tf(\bx^{k+1}, \xi^{k+1}) - \nabla f(\bx^{k+1}) + (1 - \alpha_{k})(\nabla f(\bx^{k}) - \nabla \tf(\bx^{k}, \xi^{k+1}))\|^{2}|\xi^{[k]}] \notag \\
=&\, \EE[\| \alpha_{k}(\nabla \tf(\bx^{k+1}, \xi^{k+1}) - \nabla f(\bx^{k+1})) + (1 - \alpha_{k})(\nabla f(\bx^{k}) - \nabla f(\bx^{k+1}) \notag \\
     &\qquad\qquad + \nabla \tf(\bx^{k+1}, \xi^{k+1}) - \nabla \tf(\bx^{k}, \xi^{k+1}))\|^{2}|\xi^{[k]}] \notag \\
\leq&\, 2\alpha_{k}^{2}\EE[\| \nabla \tf(\bx^{k+1}, \xi^{k+1}) - \nabla f(\bx^{k+1})\|^{2}|\xi^{[k]}]\notag \\
     &+ 2(1 - \alpha_{k})^{2}\EE[\| \nabla f(\bx^{k}) - \nabla f(\bx^{k+1})+ \nabla \tf(\bx^{k+1}, \xi^{k+1}) - \nabla \tf(\bx^{k}, \xi^{k+1})\|^{2}|\xi^{[k]}] \notag \\
\leq&\, 2\alpha_{k}^{2}\sigma^{2} + 2(1 - \alpha_{k})^{2}\EE[\| \nabla \tf(\bx^{k}, \xi^{k+1}) - \nabla \tf(\bx^{k+1}, \xi^{k+1})\|^{2}|\xi^{[k]}] \label{Eq: eq1 in lm recursive} \\
\leq&\, 2\alpha_{k}^{2}\sigma^{2}+ 2L^2_{f} \EE [\| \bw^{k+1}- \bw^{k}\|^{2}| \xi^{[k]}], \label{Eq: the second term of the recursive stochastic error}
\end{align}
where \eqref{Eq: eq1 in lm recursive} uses
    $\EE\left[\nabla \tf(\bx^{k+1},\xi^{k+1}) - \nabla \tf(\bx^{k},\xi^{k+1}
    ) | \xi^{[k]} \right] = \nabla f(\bx^{k+1}) - \nabla f(\bx^{k})$
    and \eqref{Eq: the second term of the recursive stochastic error} uses \Cref{Ass:
    the expected smoothness of f} and $0 < \alpha_{k}\leq 1$. Then, substituting
    \eqref{Eq: the second term of the recursive stochastic error} into \eqref{Eq: the recursive error of stochastic estimator}
    and taking full expectation on both sides of this inequality, we obtain
    \eqref{Eq: the stochastic error of the variance estimator} and complete the proof.
\end{proof}
For MoSSP-R, we similarly use the recursive structure of $\EE [ \|\be^{k}\|^{2}]$ from \Cref{Lemma: Recursive bound on gradient error for storm} to establish a weighted cumulative error bound. We assume the following additional parameter conditions on $\alpha_k$ and \(\mu_k\):
\begin{align}
    \label{Eq: the additional parameters assumptions in storm}
    0 < 32 \mu_{k}^{2}L_f^{2}\leq \alpha_{k}\leq 1, \quad \forall\, k \geq 0.
\end{align}
\begin{lemma}
    \label{Lemma: Weighted error accumulation of storm} Suppose the assumptions of
    Lemma \ref{Lemma: Recursive bound on gradient error for storm} hold and the parameters satisfy
    both \eqref{Eq: basic parameter conditions for both algorithms} and \eqref{Eq: the additional parameters assumptions in storm}.
    Then, for any $K \geq 1$, it holds that
    \begin{align}
        \label{Eq: weighted cumulative error bound for storm}
        &\frac{1}{K}\sumK \alpha_{k}\EE[\| \be^{k}\|^{2}] \notag \\
        \leq&\, \frac{32\mu_{0}L_f^{2}( \calL_{\rho_0, \mu_0}(\bw^{0}) - \calL^{*})}{K}+ \frac{2\EE [\| \be^{0}\|^{2}]}{K} + \frac{32\mu_0L_f^2}{K}\sumK \left( \frac{\rho_{k+1}- \rho_{k}}{2}C^2 + \EE[\Delta_{k+1}]\right)+ \frac{1}{K}\sum\limits_{k=0}^{K-1}4 \alpha_{k}^{2}\sigma^{2}.
    \end{align}
Furthermore, if we set $\mu_{k}\equiv \mu$, $\alpha_{k}\equiv \alpha$,
and $\rho_{k}\equiv \rho$ for $k \geq 0$, it holds that
    \begin{align}
        \label{Eq: weighted cumulative error bound under constant parameters for storm}\frac{1}{K}\sumK \EE[\| \be^{k}\|^{2}]\leq \frac{32\mu L_f^{2}( \calL_{\rho, \mu}(\bw^{0}) - \calL^{*})}{\alpha K}+ \frac{2\EE [\| \be^{0}\|^{2}]}{\alpha K}+ 4 \alpha \sigma^{2}.
    \end{align}
\end{lemma}
\begin{proof}
    From \eqref{Eq: the stochastic error of the variance estimator}, we obtain
    \begin{align}
        \label{Eq: the decay relation of decay gradient}\alpha_{k}\EE[\| \be^{k}\|^{2}] \leq \EE[\| \be^{k}\|^{2}] - \EE [ \| \be^{k+1}\|^{2}] + 2 \alpha_{k}^{2}\sigma^{2}+ 2L_f^{2}\EE[ \| \bw^{k+1}- \bw^{k}\|^{2}].
    \end{align}
On the other hand, taking expectation on both sides of \eqref{eq:sequence-residual-bound}
    yields
    \begin{align}
        &\label{Eq: the upper bound of w}\EE [ \| \bw^{k+1}- \bw^{k}\|^{2}]\notag \\
        \leq & \, 8\mu_{k}\EE [ \calL_{\rho_k, \mu_k}(\bw^{k}) - \calL_{\rho_{k+1}, \mu_{k+1}}(\bw^{k+1}) + \frac{\rho_{k+1}- \rho_{k}}{2}C^{2}+ \Delta_{k+1}+ \mu_{k}\|\be^{k}\|^{2}],
    \end{align}
where the inequality follows from
    $1 - 2\mu_{k}L_{\rho}\geq \frac{1}{2}$, which is implied by
    $\mu_{k}L_{\rho}\leq \frac{1}{4}$.

    Then, substituting \eqref{Eq: the upper bound of w} into \eqref{Eq: the decay relation of decay gradient},
    using \(\mu_k\leq \mu_0\) for the non-error terms,
    summing over $k = 0,\ldots, K-1$, and dividing by $K$, we obtain
    \begin{align}
             & \frac{1}{K}\sumK ( \alpha_{k}- 16\mu_{k}^{2}L_f^{2})\EE[\|\be^{k}\|^{2}] \notag\\          
             \leq&\, \frac{1}{K}\sumK ( \EE [\|\be^{k}\|^{2}] - \EE[ \|\be^{k+1}\|^{2}] ) + \frac{1}{K}\sumK 2\alpha^{2}_{k}\sigma^{2} \notag \\
             &+ \frac{16\mu_{0}L_f^{2}}{K}\sumK \bigl( \calL_{\rho_k, \mu_k}(\bw^{k}) - \calL_{\rho_{k+1}, \mu_{k+1}}(\bw^{k+1}) + \frac{\rho_{k+1}- \rho_{k}}{2}C^{2} + \EE[\Delta_{k+1}]\bigl).
    \end{align}
Since $\alpha_{k}\geq 32 \mu_{k}^{2}L_f^{2}$, we have $\alpha_{k}
    - 16\mu_{k}^{2}L_f^{2}\geq \frac{\alpha_{k}}{2}$, yielding
    \begin{align}
            &\frac{1}{K}\sumK \frac{\alpha_{k}}{2}\EE[\|\be^{k}\|^{2}] \notag \\
            \leq&\, \frac{\EE
        [\|\be^{0}\|^{2}]}{K}+ \frac{16\mu_{0}L_f^{2}( \calL_{\rho_0, \mu_0}
        (\bw^{0}) - \calL^{*}) }{K}+ \frac{16\mu_0L_f^2}{K}\sumK ( \frac{\rho_{k+1}-
        \rho_{k}}{2}C^{2} + \EE[\Delta_{k+1}] )+ \frac{1}{K}\sumK 2\alpha^{2}_{k}\sigma^{2}, \notag
    \end{align}
where rearranging the inequality yields \eqref{Eq: weighted cumulative error bound for storm} and completes the proof.
\end{proof}
We now characterize the measurements in \eqref{Eq: approximate stationarity point} for MoSSP-R with constant parameters $\rho_{k}\equiv \rho $, $\mu_{k}\equiv \mu$, $\alpha_{k}\equiv \alpha$ for any \(k \geq 0\).
\begin{lemma}
    \label{Lemma: the error bound of the infeasible stationarity in expectation for VR}
    Suppose that the conditions of Lemma~\ref{Lemma: Weighted
    error accumulation of storm} hold. Then, it holds that
   \begin{align}
\label{Eq: the criticality measure error bound in expectation for VR} & \max\{ \EE[\| \bu^{R+1}\|^{2}] , \EE [\| \bx^{R+1}- \prox_{\mu g}(\bz^{R})\|^{2}]\} \notag  \\
\leq & \, \frac{1}{K}\left( \frac{(96+768(\tfrac{1}{16}+\beta^{-2}+\tfrac{1}{3}(\mu/\beta)^2)) \mu L_f^{2}}{\alpha}+ \frac{24(\tfrac{1}{16} + \beta^{-2}+\tfrac{1}{3}(\mu/\beta)^2)}{\mu}\right) ( \calL_{\rho, \mu}(\bw^{0}) - \calL^{*}) \notag\\
& + \frac{( 48( \tfrac{1}{16} + \beta^{-2}+\tfrac{1}{3}(\mu/\beta)^2) + 6 )\EE [\| \be^{0}\|^{2}]}{\alpha K }+ ( 96 ( \tfrac{1}{16} + \beta^{-2}+\tfrac{1}{3}(\mu/\beta)^2) + 12)\alpha\sigma^{2},
    \end{align}
and
\begin{align}
\label{Eq: the infeasible stationarity error bound in expectation for VR}
&\EE [\|\nabla \bc(\bx^{R+1}) \bc(\bx^{R+1}) \|^{2}] \notag \\
\leq & \, \frac{2}{\rho^{2}K}( \frac{(144 +768 \beta^{-2}) \mu L^2_{f}}{\alpha}+ \frac{24(\tfrac{1}{16}+ \beta^{-2})}{\mu}) (\calL_{\rho, \mu}(\bw^{0}) - \calL^{*}) + \frac{( 12 + 96( \tfrac{1}{16} + \beta^{-2}))\EE [\| \be^{0}\|^{2}]}{\alpha\rho^{2}K }\notag \\
&+ \frac{( 24 + 192( \tfrac{1}{16} + \beta^{-2})) \alpha\sigma^{2}}{\rho^{2}}+\frac{18G^{2}}{\rho^{2}}.
\end{align}
\end{lemma}
\begin{proof}
    For \eqref{Eq: the criticality measure error bound in expectation for VR},
    we substitute \eqref{Eq: weighted cumulative error bound under constant parameters for storm}
    into \eqref{Eq: the criticality measure error bound in expectation} to
    obtain the desired results. 
    
    For \eqref{Eq: the infeasible stationarity error bound in expectation for VR},
    under the constant parameter setting, we substitute \eqref{Eq: weighted cumulative error bound under constant parameters for storm}
    into \eqref{Eq: the boundedness of subgradient-averaging}, obtaining
\begin{align}
&\frac{1}{K}\sumK \EE [\| \bu^{k+1}\|^{2}] \notag \\
\leq & \, \frac{1}{K}\left( \frac{(96+768(\tfrac{1}{16}+\beta^{-2})) \mu L_f^{2}}{\alpha}+ \frac{24(\tfrac{1}{16} + \beta^{-2})}{\mu}\right) ( \calL_{\rho, \mu}(\bw^{0}) - \calL^{*})  \notag\\
& + \frac{ 2( 3 + 24(\tfrac{1}{16} + \beta^{-2}) ) \EE [\| \be^{0}\|^{2}]}{\alpha K} + (12 + 96(\tfrac{1}{16} + \beta^{-2}))\alpha\sigma^{2}. \notag
\end{align}
Combining this inequality with \eqref{Eq: the bound of infeasible stationarity-averaging}
    directly yields \eqref{Eq: the infeasible stationarity error bound in expectation for VR}, completing the proof.
\end{proof}
To ensure the parameter conditions \eqref{Eq: basic parameter conditions for both algorithms} and \eqref{Eq: the additional parameters assumptions in storm}, we set the parameters as follows
\begin{equation}
    \label{Eq: the parameter conditions in VR for analysis}
\begin{aligned}
\rho_{k} &\equiv \rho= \rho_0 K^{l}, \\
\mu_{k} &\equiv \mu= \frac{\mu_0}{K^{l}\max\{L_f, \tilde{L}\}}, \\
\alpha_{k} &\equiv \alpha = \frac{16\alpha_0 \mu_0^{2}}{K^{\tau}},
\quad 0 < \beta \leq 1,
\end{aligned}
\end{equation}
where $\rho_0 > 0$, $0 < \mu_0 \leq \min\{\frac{1}{4\rho_0},\frac{\max\{L_f,\tilde{L}\}}{4\sqrt{2}L_f}\}$, $0 < \tau \leq 2 l < 1$ and
$\alpha_0 \in [\frac{2L_f^2}{\max\{L_f,\tilde{L}\}^2},\,\frac{1}{16\mu_0^{2}}]$ are given constants independent of $K$. Then,
\Cref{Lemma: the error bound of the infeasible stationarity in expectation for VR}
yields the following convergence rate of MoSSP-R for finding a \textit{stochastic \(\varepsilon\)-stationary point}:
\begin{align}
    \label{Eq: convergence rate of storm}
    \begin{cases}\max\{ \EE [\| \bu^{R+1}\|^{2}], \EE[\| \bx^{R+1}- \prox_{\mu  g}(\bz^{R})\|^{2}]\} \\
    \qquad \qquad =  \calO\!\left(\max\{ K^{l-1}\left( \calL_{\rho, \mu}(\bw^{0})+1\right),\, K^{\tau - 1}\EE [\| \be^{0}\|^{2}],\, K^{-\tau}\}\right), \\ \EE [\|\nabla \bc(\bx^{R+1})\bc(\bx^{R+1})\|^{2}] \\
    \qquad \qquad = \calO\!\left(\max\{ K^{-l-1}\left( \calL_{\rho, \mu }(\bw^{0}) + 1\right),\, K^{-2l + \tau -1 }\EE [\| \be^{0}\|^{2}],\, K^{-\tau}\}\right).
\end{cases}
\end{align}
As can be seen in \eqref{Eq: convergence rate of storm}, \(\EE[\| \be^0 \|^2]\) affects the order in \eqref{Eq: convergence rate of storm}. Therefore, with an initial batch size \(b_0 = \calO(K^{l})\), one has \( \EE[\| \be^0 \|^2] = \calO(K^{-l}) \). Similar to the proof of \Cref{proof: polyak with feasibility}, under \Cref{Ass: constraints conditions}, combining \eqref{Eq: the constraint violation-averaging} with \eqref{Eq: the infeasible stationarity error bound in expectation for VR} yields that the upper bound on \( \EE[\|\bc(\bx^{R+1})\|^{2}] \) is in the same order as \( \EE[\|\nabla \bc(\bx^{R+1}) \bc(\bx^{R+1})\|^{2}] \).

We are now ready to prove \Cref{Theorem: the convergence rate of storm with initial feasibility}.

\subsubsection{Proof of \Cref{Theorem: the convergence rate of storm with initial feasibility}}
\label{proof: storm with feasibility}
\begin{proof}
Consider the parameter choices in \eqref{Eq: the parameter conditions in VR for analysis}.
When $\|\bc(\bx^{0})\|^{2} = \calO(K^{-l})$ and the initial batch size is chosen as $b_0 = \calO(K^{l})$, we have $\calL_{\rho,\mu}(\bw^{0}) = \calO(1)$ and $\EE[\| \be^0\|^2] = \calO(K^{-l})$. Combining it with \eqref{Eq: convergence rate of storm}, we observe that the two bounds reduce to
\begin{align}
    \begin{cases}
    \max\left\{ \EE[\|\bu^{R+1}\|^{2}], \EE[\|\bx^{R+1}- \prox_{\mu g}(\bz^{R})\|^{2}] \right\} = \calO\!\left(\max\{K^{l-1},\,K^{\tau-l-1},\,K^{-\tau}\}\right), \\
    \EE[\|\nabla \bc(\bx^{R+1}) \bc(\bx^{R+1})\|^{2}] = \calO\!\left(\max\{K^{-l-1},\,K^{-\tau}\}\right).
    \end{cases} \label{Eq: the reduced complexity order of storm for analysis}
\end{align}
To achieve the optimal rate $\calO(K^{-2/3})$, we set $\tau = 2l = \tfrac{2}{3}$. Hence, to obtain a stochastic $\varepsilon$-stationary point, $K$ should be of order $\calO(\varepsilon^{-3})$. Noting that each iteration calls two stochastic gradients $\nabla \tf(\bx^{k},\xi^{k})$ and $\nabla \tf(\bx^{k-1},\xi^{k})$, and the initial batch requires $b_0 = \calO(K^{1/3}) = \calO(\varepsilon^{-1})$ gradient evaluations, the total oracle complexity is
\[
b_0 + 2K = \calO(\varepsilon^{-1}) + 2\calO(\varepsilon^{-3}) = \calO(\varepsilon^{-3}).
\]
Under \Cref{Ass: constraints conditions}, the established bound on $\EE[\|\bc(\bx^{R+1})\|^{2}]$ ensures that the oracle complexity to find a stochastic $\varepsilon$-KKT point is also $\calO(\varepsilon^{-3})$.
\end{proof}

\subsubsection{Proof of \Cref{Theorem: the basic convergence rate of storm without initial feasibility}}
\label{app:storm_without_initial_feasibility}
Without assuming initial approximate feasibility, we have $\calL_{\rho,\mu}(\bw^0)+1 = \calO(K^{l})$. With a constant initial batch size $b_0=\calO(1)$, \Cref{Ass: unbiased gradient oracle} gives $\EE[\|\be^0\|^2]=\calO(1)$. Invoking \Cref{Lemma: the error bound of the infeasible stationarity in expectation for VR} and combining it with \eqref{Eq: convergence rate of storm}, we observe that the two bounds admit
\[
    \calO\!\left(\max\{K^{2l-1},\,K^{\tau-1},\,K^{-\tau}\}\right)
    \qquad\text{and}\qquad
    \calO\!\left(\max\{K^{-1},\,K^{\tau-1},\,K^{-\tau}\}\right).
\]
If we choose $\tau=2l=\tfrac{1}{2}$, both reduce to $\calO(K^{-1/2})$. Hence, to obtain a stochastic $\varepsilon$-stationary point, $K$ should be of order $\calO(\varepsilon^{-4})$ and the total oracle complexity to find an $\varepsilon$-stationary point is of order $\calO(\varepsilon^{-4})$. Under \Cref{Ass: constraints conditions}, the same complexity holds for finding a stochastic $\varepsilon$-KKT point, establishing the result in \Cref{Theorem: the basic convergence rate of storm without initial feasibility}.

\section{Experimental Results}
\label{appendix:complete experiment results}
\subsection{Implementation Details}
\label{app:implementation_details}
All experiments were conducted in MATLAB R2018b on a MacBook Pro (4-core processor, 16 GB RAM) under macOS 15.3.2.

\vspace{3pt}
\noindent\textbf{Baselines}. We compare our algorithms with two double-loop baselines: SPDC \cite{xu2019stochastic,nitanda2017stochastic} for simple convex-constrained DC(-regularized) optimization and SALM \cite{sun2023algorithms} for linearly constrained DC-regularized optimization:
\begin{enumerate}
    \item[(i)] \textbf{SPDC \cite{xu2019stochastic,nitanda2017stochastic}:} This is a double-loop algorithm designed for solving DC problems by linearizing the concave component and adding a quadratic proximal term to construct a strongly convex subproblem at each outer iteration \( k \). To handle the constraint \( \|\mathbf{x}\|_2^2 = 1 \), we apply the same quadratic penalty approach for nonconvex constraints in our comparison. To ensure strong convexity of the subproblem, we adopt an appropriate proximal parameter. The inner subproblem is solved iteratively using the stochastic subgradient descent method (SPG) described in \cite{xu2019stochastic}.
    \item[(ii)] \textbf{SALM \cite{sun2023algorithms}:} This double-loop method is designed for deterministic, linearly constrained composite DC optimization. The method handles constraints by constructing the Augmented Lagrangian Method (ALM), linearizing the AL function at each outer iteration, and solving the inner loop using a proximal gradient method. We adopt a similar strategy and adapt it to nonlinear constraints by applying the same linearization approach. Specifically, the concave part is linearized at \(\bx^k\), and the inner solver is used to update the solution. For the stochastic part, we use the same gradient estimators as in our algorithm.
\end{enumerate}

\noindent\textbf{Hyperparameter Settings}. 
We initialize all algorithms from a feasible point, generated by normalizing a random vector. Given the data scale $N$, we use a batch size of 32 for the \texttt{a9a} and \texttt{phishing} datasets, and a batch size of 16 for the \texttt{australian} dataset. The momentum parameter \(\alpha\) is set to $0.905$ for Polyak momentum and $0.9$ for recursive momentum across all baselines. The maximum number of iterations is set to \(K = 25,000\) for all experiments. For MoSSP-P, the smoothing and penalty parameters are set according to the theoretical rates $\mu_k=\calO(K^{-1/2})$ and $\rho_k=\calO(K^{1/4})$; for MoSSP-R, they are set as $\mu_k=\calO(K^{-1/3})$ and $\rho_k=\calO(K^{1/3})$, with the initial batch size chosen according to $b_0=\calO(K^{1/3})$. For SPDC and SALM, we fix $\mu_k$ as a constant independent of $K$, tuned from $\{0.05, 0.2, 0.5, 1\}$. For SALM, $\rho$ is tuned from $\{0.01, 0.1, 0.5, 1\}$, while for SPDC, it is tuned from $\{1, 5, 10, 20\}$. The step size $\beta$ in both MoSSP variants is set to $\beta = 1$. For SALM, the dual update step size is set equal to the penalty parameter \(\rho\), while the step sizes for the SPDC subgradient update and the inner-loop update of SPDC are tuned from \(\{0.001, 0.01, 0.05, 0.1\}\). The regularization parameter \(\lambda\) is validated over the set \(\{0.005, 0.05, 0.1\}\), and the optimal value is selected. The number of inner iterations is chosen from the range \(5-10\).

We perform five independent runs for each dataset and algorithm combination. For fair time comparison, we first run MoSSP-P (\Cref{alg:mossp-p}) for \(K = 25{,}000\) iterations and record its total CPU time; all Polyak-momentum baselines are run under the same CPU-time budget, and their trajectories are plotted against the cumulative number of stochastic gradients. The same procedure applies to MoSSP-R (\Cref{alg:mossp-r}) and recursive momentum-based baselines. \Cref{Fig: phishing_results,Fig: australian_results} show convergence trajectories on the \texttt{phishing} and \texttt{australian} datasets (averaged over five runs).
\begin{figure}[t]
    \centering
    \includegraphics[width=0.95\columnwidth]{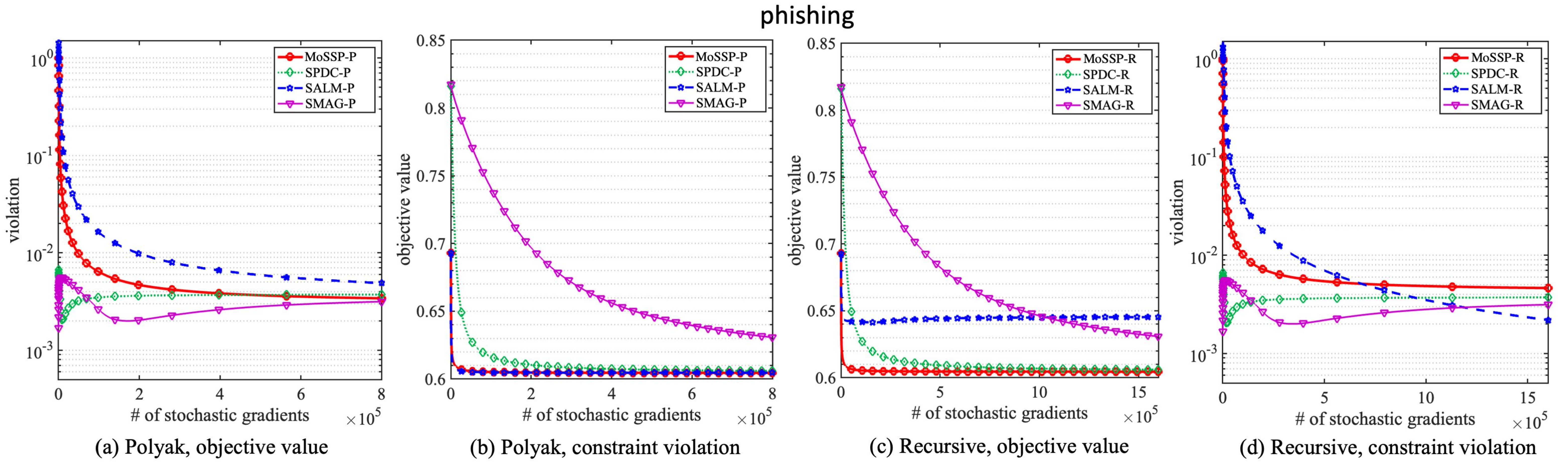}
    \caption{Comparison of MoSSP variants, SPDC, and SALM for solving the constrained binary classification problem \eqref{eq.libsvm} on the \texttt{phishing} dataset. (a) Objective value (Polyak). (b) Constraint violation (Polyak). (c) Objective value (Recursive). (d) Constraint violation (Recursive). Results are averaged over five independent runs.}
    \label{Fig: phishing_results}

    \vspace{0.8em}

    \includegraphics[width=0.95\columnwidth]{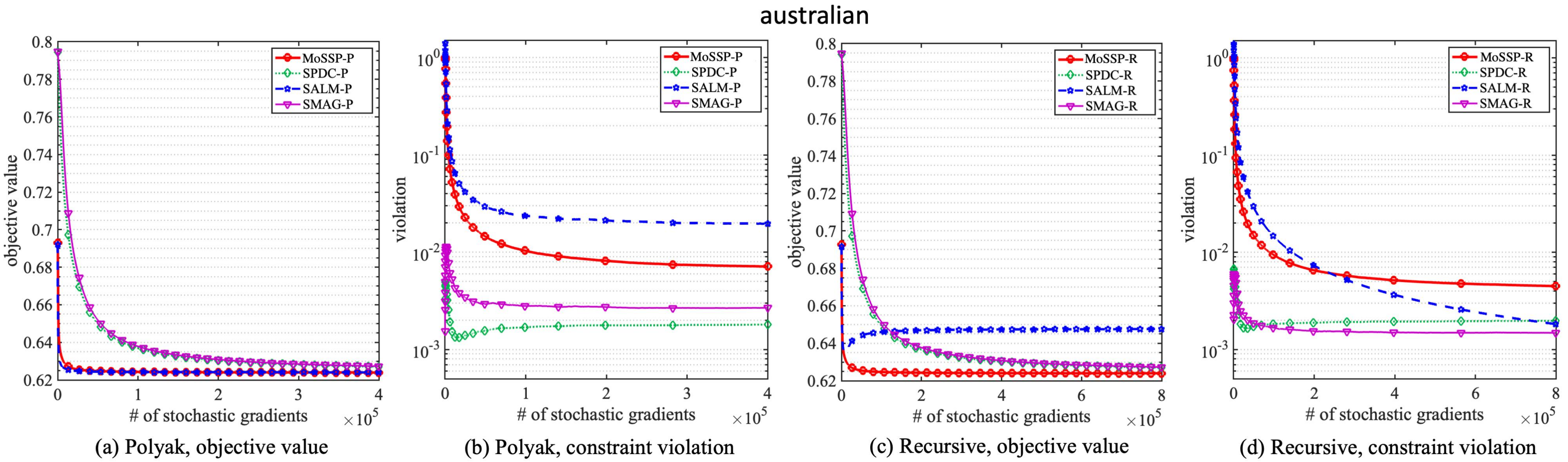}
    \caption{Comparison of MoSSP variants, SPDC, and SALM for solving the constrained binary classification problem \eqref{eq.libsvm} on the \texttt{australian} dataset. (a) Objective value (Polyak). (b) Constraint violation (Polyak). (c) Objective value (Recursive). (d) Constraint violation (Recursive). Results are averaged over five independent runs.}
    \label{Fig: australian_results}
     \vskip -0.1in
\end{figure}

\subsection{Additional Experimental Results}
\label{app:additional_experimental_results}

As shown in \Cref{Fig: phishing_results}, on the large-scale dataset \texttt{phishing}, both MoSSP variants converge faster in objective value and attain the lowest or highly competitive final results. Notably, the MoSSP variants rapidly reduce constraint violations from an initial level of $10^{0}$ to approximately $10^{-3}$, demonstrating effective feasibility control. In contrast, SPDC and SALM are slower at achieving feasibility, with violations remaining around $10^{-2}$ or higher throughout most of the optimization.

On the smaller-scale \texttt{australian} dataset (\Cref{Fig: australian_results}), the relative performance shifts. While the MoSSP variants still achieve faster objective descent, SPDC attains superior feasibility with violations stabilizing around the $10^{-3}$ level, outperforming the MoSSP variants in constraint satisfaction. This suggests that the simpler landscape and lower stochastic noise enable subgradient-based constraint handling to be more effective, whereas the MoSSP design prioritizes rapid objective reduction over strict feasibility in such settings. This scale-dependent behavior warrants further investigation.

\subsection{Experimental Results on Multiple Quadratic Equality Constraints}
\label{subsec:multiple_quadratic_equalities}
We test the proposed method on a quadratically constrained DC-regularized logistic regression problem
\cite{jin2022stochastic,shi2025momentum}:
\begin{align}
\label{eq:multiple_quadeq}
\min_{\bx \in \RR^n}\quad &
\frac{1}{N}\sum_{i=1}^{N}\log\bigl(1+\exp(-y_i X_i^\top \bx)\bigr)
    + \lambda\bigl(\|\bx\|_1-\|\bx\|_2\bigr) \notag \\
\mathrm{s.t.}\quad &
c_j(\bx):=\tfrac{1}{2}\sum_{\ell=1}^{n}q_{j,\ell}\,
x_\ell^2+\ba_j^\top \bx-b_j=0,
\quad j=1,\ldots,M.
\end{align}
Each constraint can be equivalently written as
\(c_j(\bx)=\tfrac{1}{2}\bx^\top Q_j\bx+\ba_j^\top\bx-b_j\), where
\(Q_j=\Diag(q_{j,1},\ldots,q_{j,n})\). Since all diagonal entries
\(q_{j,\ell}\) are positive, each \(Q_j\) is positive definite. 
We test on two LIBSVM datasets used in the main experiments:
\texttt{a9a} with \(N=32{,}561\) and \(n=123\), and
\texttt{phishing} with \(N=11{,}055\) and \(n=68\). Each sample vector is
normalized to unit \(\ell_2\) norm. We set \(M=20\). For each
constraint~\(j\), the coefficients \(q_{j,\ell}\) are sampled uniformly
from \([0.5,1]/n\), and
\(\ba_j \sim \mathcal{N}(\zero,\bI/n)\). A random unit vector
\(\bx_\star\) is generated first, and \(b_j\) is set to
\[
b_j=\tfrac{1}{2}\sum_{\ell=1}^{n}q_{j,\ell}x_{\star,\ell}^{2}
+\ba_j^\top\bx_\star,
\]
so that \(\bx_\star\) is feasible for all constraints.

\noindent\textbf{Hyperparameter Settings}.
Given the data scale~$N$, we use a batch size of~$32$ for both
\texttt{a9a} and \texttt{phishing} datasets. The momentum parameter is set to
$\alpha=0.905$ for Polyak momentum and $\alpha=0.9$ for recursive
momentum across all methods. The maximum number of iterations is set to
$K=20{,}000$, and the regularization parameter is fixed as
$\lambda=0.01$. For MoSSP-P, we use $\beta=1$,
$\mu_k=\calO(K^{-1/2})$, and $\rho_k=\calO(K^{1/4})$; for
MoSSP-R, we use $\beta=1$, $\mu_k=\calO(K^{-1/3})$, and
$\rho_k=\calO(K^{1/3})$, with the initial batch size chosen according
to $b_0=\calO(K^{1/3})$. For SPDC and SALM, $\mu_k$ is tuned from
$\{0.05,0.2,0.5,1\}$; $\rho_k$ is tuned from $\{0.01,0.1,0.5,1\}$ for
SALM and from $\{1,5,10,20\}$ for SPDC; the step sizes for the SPDC
subgradient update and the inner-loop update are tuned from
$\{0.001,0.01,0.05,0.1\}$. For SALM, the dual update step size is set
equal to~$\rho_k$. The double-loop baselines use $5$ inner iterations.

Each method is run five times with a shared constraint instance and
initialization. For each momentum setting, all baselines are evaluated under the same computational budget as the corresponding MoSSP variant. We report the objective value and the aggregate constraint violation $\textstyle\sum_{j=1}^{M}|c_j(\bx)|$. \Cref{Fig: quadeq_a9a_results,Fig: quadeq_phishing_results} show convergence trajectories on both datasets (averaged over five runs), while \Cref{table:quadeq_result_polyak,table:quadeq_result_recursive} report final values as mean $\pm$ std.  
\begin{figure}[t]
    \centering
    \includegraphics[width=0.95\columnwidth]{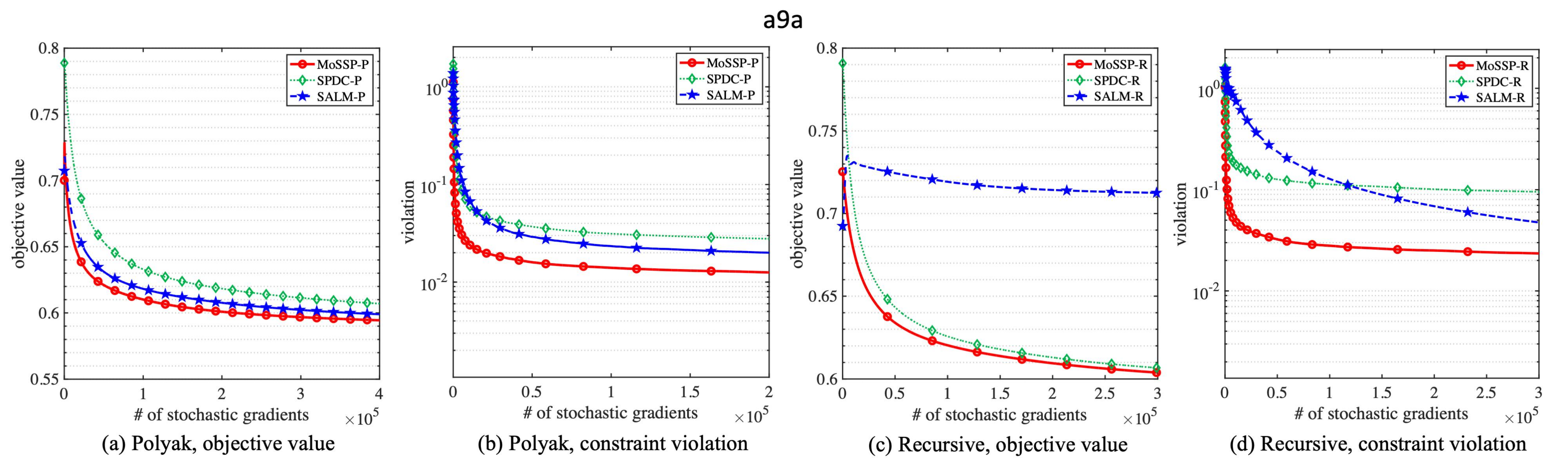}
    \caption{Comparison of MoSSP variants, SPDC, and SALM for solving the problem with multiple quadratic equality constraints \eqref{eq:multiple_quadeq} on the \texttt{a9a} dataset. (a) Objective value (Polyak). (b) Constraint violation (Polyak). (c) Objective value (Recursive). (d) Constraint violation (Recursive). Results are averaged over five independent runs.}
    \label{Fig: quadeq_a9a_results}

    \vspace{0.8em}

    \includegraphics[width=0.95\columnwidth]{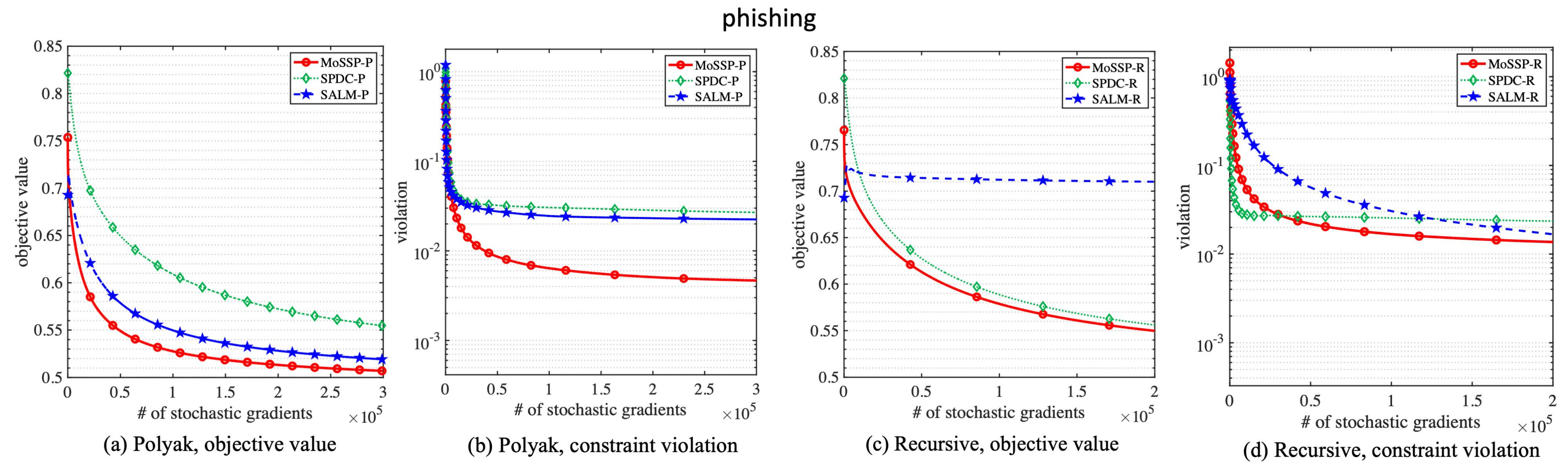}
    \caption{Comparison of MoSSP variants, SPDC, and SALM for solving the problem with multiple quadratic equality constraints \eqref{eq:multiple_quadeq} on the \texttt{phishing} dataset. (a) Objective value (Polyak). (b) Constraint violation (Polyak). (c) Objective value (Recursive). (d) Constraint violation (Recursive). Results are averaged over five independent runs.}
    \label{Fig: quadeq_phishing_results}
    \vskip -0.1in
\end{figure}
\begin{table*}[t]
\centering
\caption{Mean \(\pm\) std of objective value (Obj. Value) and constraint violation (Const. Viol.) for MoSSP-P and two baseline methods with Polyak momentum on the multiple quadratic equality experiment. Results are reported over five independent runs. Bold font denotes the best result.}
\label{table:quadeq_result_polyak}
\begin{sc}
\resizebox{0.9\textwidth}{!}{%
\begin{tabular}{lcccc}
\toprule
Dataset & Metric & MoSSP-P & SPDC-P & SALM-P \\ 
\midrule
\multirow{2}{*}{\texttt{a9a}} 
 & Obj. Value & $\mathbf{0.5811 \pm 1.15\times 10^{-5}}$ & $0.6017 \pm 2.40\times 10^{-5}$ & $0.5914 \pm 2.51\times 10^{-4}$ \\
 & Const. Viol. & $\mathbf{9.93\times 10^{-3} \pm 8.62\times 10^{-5}}$ & $2.27\times 10^{-2} \pm 5.56\times 10^{-5}$ & $1.90\times 10^{-2} \pm 2.46\times 10^{-4}$ \\ 
\midrule
\multirow{2}{*}{\texttt{phishing}} 
 & Obj. Value & $\mathbf{0.4951 \pm 9.20\times 10^{-5}}$ & $0.5294 \pm 7.07\times 10^{-5}$ & $0.5067 \pm 9.16\times 10^{-5}$ \\
 & Const. Viol. & $\mathbf{9.20\times 10^{-3} \pm 2.14\times 10^{-5}}$ & $2.40\times 10^{-2} \pm 5.37\times 10^{-5}$ & $1.22\times 10^{-2} \pm 1.07\times 10^{-4}$ \\ 
\bottomrule
\end{tabular}}%
\end{sc}
\vspace{0.8em}

\caption{Mean \(\pm\) std of objective value (Obj. Value) and constraint violation (Const. Viol.) for MoSSP-R and two baseline methods with recursive momentum on the multiple quadratic equality experiment. Results are reported over five independent runs. Bold font denotes the best result.}
\label{table:quadeq_result_recursive}
\begin{sc}
\resizebox{0.9\textwidth}{!}{%
\begin{tabular}{lcccc}
\toprule
Dataset & Metric & MoSSP-R & SPDC-R & SALM-R \\ 
\midrule
\multirow{2}{*}{\texttt{a9a}} 
 & Obj. Value & $\mathbf{0.5911 \pm 1.11\times 10^{-6}}$ & $0.5930 \pm 3.53\times 10^{-5}$ & $0.7085 \pm 9.04\times 10^{-6}$ \\
 & Const. Viol. & $\mathbf{4.07\times 10^{-3} \pm 3.72\times 10^{-5}}$ & $1.79\times 10^{-2} \pm 4.05\times 10^{-5}$ & $1.70\times 10^{-2} \pm 1.16\times 10^{-4}$ \\ 
\midrule
\multirow{2}{*}{\texttt{phishing}} 
 & Obj. Value & $\mathbf{0.5063 \pm 1.04\times 10^{-6}}$ & $0.5085 \pm 1.27\times 10^{-4}$ & $0.7030 \pm 3.36\times 10^{-6}$ \\
 & Const. Viol. & $9.93\times 10^{-3} \pm 2.58\times 10^{-5}$ & $1.94\times 10^{-2} \pm 2.42\times 10^{-5}$ & $\mathbf{3.27\times 10^{-3} \pm 1.42\times 10^{-5}}$ \\ 
\bottomrule
\end{tabular}}%
\end{sc}
\vskip -0.1in
\end{table*}

\noindent\textbf{Experimental Results}. As shown in \Cref{Fig: quadeq_a9a_results,Fig: quadeq_phishing_results}, the MoSSP variants achieve faster objective decrease than the baselines on both datasets. With Polyak momentum, MoSSP-P also yields lower constraint violation than SPDC-P and SALM-P. With recursive momentum, MoSSP-R maintains the fastest objective convergence and achieves competitive, often lower, constraint violation compared with SPDC-R and SALM-R. \Cref{table:quadeq_result_polyak,table:quadeq_result_recursive} provide the corresponding final quantitative comparisons on the \texttt{a9a} and \texttt{phishing} datasets. Under Polyak momentum, MoSSP-P achieves both the lowest objective value and the smallest constraint violation on both datasets; in particular, it reaches an objective value of \(0.5811\) and a violation of \(9.93\times 10^{-3}\) on the \texttt{a9a} dataset, and keeps the violation below \(10^{-2}\) on the \texttt{phishing} dataset. With recursive momentum, MoSSP-R attains the best objective values on both datasets and the smallest constraint violation on the \texttt{a9a} dataset. On the \texttt{phishing} dataset, SALM-R achieves the smallest final violation, but at the cost of a substantially larger objective value than MoSSP-R and SPDC-R.

\end{document}